\documentclass[a4paper,12pt]{amsart}
\usepackage{graphicx}
\usepackage{amssymb}
\usepackage{amsmath}
\usepackage{latexsym}
\usepackage{amsthm}
\usepackage{autograph,epic,latexsym,bezier,amsbsy,psfrag,color,enumerate,amsfonts,amsmath,amscd,amssymb}
\usepackage[latin1]{inputenc}

\setloopdiam{17}\setprofcurve{10}
\newtheorem{thm}{Theorem}[section]
\newtheorem{prop}{Proposition}[section]
\newtheorem{cor}{Corollary}[section]
\newtheorem{lemma}{Lemma}[section]
\newtheorem{defi}{Definition}[section]
\newtheorem{example}{Example}[section]
\newtheorem{remark}{Remark}[section]
\newcommand{\zig}{{\textcircled z}}

\newcommand{\rot}{\operatorname{Rot}}
\newcommand{\rr}{{\textcircled r}}

\begin{document}

\title[Connectedness and isomorphisms of zig-zag products of graphs]
{Connectedness and isomorphism properties of the zig-zag product of graphs}

\author{Daniele D'Angeli}
\address{Institut f\"{u}r mathematische Strukturtheorie (Math C)\\
Technische Universit\"{a}t Graz \ \ Steyrergasse 30, 8010 Graz, Austria}
\email{dangeli@math.tugraz.at}
\author{Alfredo Donno}
\address{Universit\`{a} degli Studi Niccol\`{o} Cusano - Via Don Carlo Gnocchi, 3 00166 Roma, Italia \\
Tel.: +39 06 45678356, Fax: +39 06 45678379}
\email{alfredo.donno@unicusano.it}
\author{Ecaterina Sava-Huss}
\address{Institut f\"{u}r mathematische Strukturtheorie (Math C)\\
Technische Universit\"{a}t Graz \ \ Steyrergasse 30, 8010 Graz, Austria}
\email{sava-huss@tugraz.at}


\date{\today, preprint}

\begin{abstract}
In this paper we investigate the connectedness and the isomorphism
problems for zig-zag products of two graphs. A sufficient
condition for the zig-zag product of two graphs to be connected is
provided, reducing to the study of the connectedness property
of a new graph which depends only on the second factor of the graph product. We
show that, when the second factor is a cycle graph, the study of
the isomorphism problem for the zig-zag product is equivalent to
the study of the same problem for the associated
pseudo-replacement graph. The latter is defined in a natural way, by
a construction generalizing the classical replacement product, and
its degree is smaller than the degree of the zig-zag product graph.

Two particular classes of products are studied in detail: the
zig-zag product of a complete graph with a cycle graph, and the
zig-zag product of a $4$-regular graph with the cycle graph of
length $4$. Furthermore, an example coming from the theory of Schreier graphs
associated with the action of self-similar groups is also considered:
the graph products are completely determined and their spectral
analysis is developed.
\end{abstract}

\maketitle

\begin{center}
\end{center}

\section{Introduction}

\indent The fruitful idea of constructing new graphs starting from
smaller factor graphs is very popular in Mathematics and it has
been largely studied and developed in the literature for its
theoretical interest, as well as for its numerous applications in
several branches like Combinatorics, Probability, Theoretical
Computer Science, Statistical Mechanics.

This paper is devoted to the study of the connectedness and
isomorphism properties of zig-zag products of graphs. This
combinatorial construction, which applies to regular graphs, was introduced in \cite{annals} by O.
Reingold, S. Vadhan and A. Wigderson, in order to provide new
sequences of constant degree expanders of arbitrary size.
Informally, a graph is expander if it is simultaneously sparse,
i.e., it has relatively few edges, and highly connected. What is
mostly fascinating about expander graphs, is the fact that the
expansion property can be described from several points of view -
combinatorial, algebraic and probabilistic. Expander graphs
have many interesting applications in different areas of Computer
Science, such as design and analysis of communication networks and
error correcting codes, as well as in many computational problems, by
playing a crucial role also in Statistical Physics, Computational
Group Theory, and Optimization \cite{expander, lubotullio}.

The zig-zag product is strictly related to a simpler
construction, called replacement product of graphs. The replacement
and the zig-zag product play an important role in Geometric
Group Theory, since it turns out that, when applied to Cayley
graphs of two finite groups, they provide the Cayley graph of the
semidirect product of these groups \cite{groups}. Further results
about the relationship between graph products and group operations
are given in \cite{donnozigzag2}.

The structure of the paper is as follows. In Section \ref{sectionpreliminaries} we recall the definition and the basic
properties of the replacement and zig-zag product of graphs. In Section \ref{sectionconnectedness}, we attack the
connectedness problem for zig-zag products of regular graphs. In
Section \ref{sectionclassifications}, we focus our attention on
the classification of the isomorphism classes of zig-zag products
in the case where the second factor graph is a cycle graph of even length. In
this context, we prove that the connected components of the
zig-zag product are in one-to-one correspondence with the
so-called parity blocks, introduced in Subsection \ref{subsectionblocks}.
These are subgraphs of the first factor of the zig-zag product,
considered together with the bi-labelling of its edges. The
isomorphism problem is treated by associating with any parity
block (and so with any connected component of the zig-zag product)
a new simpler graph, that we call pseudo-replacement graph. The pseudo-replacement graph contains, in general, less
vertices and edges, and has a smaller degree than the
corresponding connected component of the zig-zag product. Nevertheless, it
completely encodes the isomorphism properties of each connected
component (see Subsection \ref{subsectioniso}). In the case where the cycle graph has length $4$, we show that
the structure of the zig-zag
product is very regular: it consists of highly
symmetric graphs that we call double cycle graphs (see Subsection
\ref{subsectionquattro}). In particular, this implies that the
zig-zag product of graphs is not injective, as two non isomorphic
graphs can produce isomorphic zig-zag products. In this setting,
we are also able to perform a complete spectral analysis, by using
the fact that the adjacency matrices of the double cycle graphs
are circulant.

Interesting sequences of increasing regular graphs can be obtained
by considering the Schreier graphs of the action of groups generated by finite automata.
In Section \ref{Basilicasection}, we describe an application of the zig-zag product to this setting. The class of automata groups became very
popular after the introduction of the Grigorchuk group, that was
the first example of a group with intermediate growth (see
\cite{grigorchuksolved} for the definition and further
references). Surprising deep connections between groups generated
by automata, complex dynamics, fractal geometry have been
discovered, and they constitute a very exciting topic of
investigation in modern mathematics \cite{nekrashevych,
nekrateplyaev}. In particular, sequences of finite Schreier graphs
represent a discrete approximation of fractal limit
objects associated with such groups. This point of view can also
be exploited in the study of models coming from Statistical Mechanics
\cite{DIMERI, ISING, tutte1, tutte2}.

The main results achieved in the current paper can be summarized as follows.
\begin{itemize}
\item A sufficient condition for the connectedness of the zig-zag product $G_1\zig G_2$ is given in terms of
the connectedness of a new graph $\mathcal{N}$, called the neighborhood graph of $G_2$. The construction of $\mathcal{N}$
depends only on the structure of $G_2$ and the number of vertices equals the number of vertices of $G_2$ (Theorem \ref{teoremaconnessione}).

\item There exists a one-to-one correspondence between the parity blocks in the parity block decomposition of
$G_1$, and the connected components of the graph $G_1\zig G_2$ (Theorem \ref{decomposition}).

\item There exists a one-to-one correspondence between the isomorphism classes of
the connected components of the zig-zag product and the isomorphism classes of the
corresponding pseudo-replacement graphs (Theorem \ref{iso}).

\item In the case $G_2\simeq C_4$, the connected
components of the zig-zag product are isomorphic to double cycle graphs
$DC_n$, for some $n$ (Proposition \ref{propprop}).

\item If $\{\Gamma_n\}_{n\geq 1}$ is the sequence of
Schreier graphs associated with the action of the Basilica group,
then, for each $n\geq 1$, the graph $\Gamma_n\zig C_4$ is
connected and isomorphic to the double cycle graph $DC_{2^{n+1}}$
(Proposition \ref{basilicaruote}); the spectral analysis of the graphs
$\Gamma_n\zig C_4$ is explicitly performed (Theorem
\ref{basilicaspettro}).
\end{itemize}

\section{Preliminaries}\label{sectionpreliminaries}

In this section we introduce the replacement and the
zig-zag product of two regular graphs. For this, we recall first
some basic definitions and properties of regular graphs, and we fix the notation
for the rest of the paper.

Let $G = (V,E)$ be a finite undirected graph, where $V$ and $E$ denote
the vertex set and the edge set of $G$, respectively. In other words, the
elements of the edge set $E$ are unordered pairs of type
$e=\{u,v\}$, with $u,v\in V$. If $e=\{u,v\}\in E$, we say that the
vertices $u$ and $v$ are adjacent in $G$, and we use the notation
$u\sim v$. We will also say that the edge $e$ joins $u$ and $v$.
Loops and multi-edges are also allowed. A \textit{path} in
$G$ is a sequence $\{u_0,u_1,\ldots, u_t\}$ of vertices of $V$ such that
$u_i\sim u_{i+1}$. The graph $G$ is \textit{connected} if, for every
$u,v\in V$, there exists a path $u_0,u_1,\ldots, u_t$ in $G$ such
that $u_0=u$ and $u_t = v$.
The \textit{degree} of a vertex $v\in V$ of $G$ is defined as $deg
(v) = |\{e\in E : v\in e\}|$. We assume that a loop at the vertex $v$ counts twice
in the degree of $v$. We say that
$G$ is a \textit{regular} graph of degree $d$, or a $d$-regular
graph, if $deg (v)=d$ for every $v\in V$.

Let $|V|=n$ and denote by $A_G = (a_{u,v})_{u,v\in V}$ the \textit{adjacency
matrix} of $G$, that is, the square matrix of size $n$ indexed by
$V$, whose entry $a_{u,v}$ equals the number of edges joining
$u$ and $v$. Note that $a_{u,u}=2k$ if there are $k$ loops at the vertex $u$. As the graph $G$ is undirected, $A_G$ is a symmetric
matrix, so that it admits $n$ real eigenvalues $\lambda_1\geq
\lambda_2\geq \ldots \geq \lambda_n$. One has $deg
(u) = \sum_{v\in V}a_{u,v}$: in particular, the $d$-regularity
condition can be rewritten as $\sum_{v\in V}a_{u,v}=d$, for each
$u\in V$. For a $d$-regular graph $G$, the {\it normalized
adjacency matrix} is defined as $A_G' = \frac{1}{d}A_G$. It is
known \cite{libro, valettebook} that, if $G=(V,E)$ is a
$d$-regular graph, with $|V|=n$, and $A_G$ is its adjacency
matrix, then $d$ is an eigenvalue of $A_G$. Its multiplicity as an
eigenvalue of $A_G$ equals the number of connected components of
$G$, and any other eigenvalue $\lambda_i$ satisfies the condition
$|\lambda_i|\leq d$, for each $i=1,\ldots, n$.

\subsection{Replacement product of graphs}\label{subsectionreplacement}

The replacement product of two graphs is a simple and intuitive
construction, which is well known in the literature, where it was
often used in order to reduce the vertex degree without losing the
connectivity property. It has been widely used in many areas
including Combinatorics, Probability, Group theory, in the study
of expander graphs and graph-based coding schemes
\cite{expander, communications}. It is worth
mentioning that Gromov studied the second eigenvalue of an
iterated replacement product of a $d$-dimensional cube with a
lower dimensional cube \cite{gromov}.

Let us introduce some notation. Let $G = (V,E)$ be a finite connected $d$-regular graph
(loops and multi-edges are allowed). Suppose that we have a set of $d$ colors (labels), that
we identify with the set of natural numbers $[d]:=\{1,2,\ldots,d\}$. We assume
that, for each vertex $v\in V$, the edges incident to $v$ are
labelled by a color $h\in [d]$ near $v$, and that any two distinct
edges issuing from $v$ have a different color near $v$. A
\textit{rotation map} $\textrm{Rot}_{G}:V\times [d]\longrightarrow
V\times [d]$ is defined by
$$
\textrm{Rot}_{G}(v,h) = (w,k), \qquad \forall v\in V, \ h\in [d],
$$
if there exists an edge joining $v$ and $w$ in $G$, which is
colored by the color $h$ near $v$ and by the color $k$ near $w$.
We may have $h\neq k$. Moreover, it follows from the
definition that the composition $\textrm{Rot}_{G}\circ
\textrm{Rot}_{G}$ is the identity map. Since an edge of $G$ joining the
vertices $u$ and $v$ is colored by some color $h$ near $u$ and by
some color $k$ near $v$, we will say that the graph $G$ is
\textit{bi-labelled}.

\begin{defi}\label{defireplacement}
Let $G_1 = (V_1,E_1)$ be a connected $d_1$-regular graph, and let $G_2 =
(V_2,E_2)$ be a connected $d_2$-regular graph, satisfying the condition
$|V_2| = d_1$. The {\it replacement product} $G_1\rr G_2$ is the
regular graph of degree $d_2+1$ with vertex set $V_1\times V_2$,
that we can identify with the set $V_1\times [d_1]$, and whose
edges are described by the following rotation map:
$$
\textrm{Rot}_{G_1\rr G_2}((v,k),i) =
\begin{cases}
((v,m),j)&\text{if}\ i\in [d_2]\ \text{and}\ \textrm{Rot}_{G_2}(k,i)=(m,j)\\
(\textrm{Rot}_{G_1}(v,k),i)&\text{if}\ i=d_2+1,
\end{cases}
$$
for all $v\in V_1, k\in [d_1], i\in [d_2+1]$.
\end{defi}
One can imagine that the vertex set of $G_1\rr G_2$
is partitioned into clouds, which are indexed by the vertices of
$G_1$, where by definition the $v$-cloud, for $v\in V_1$,
consists of vertices $(v,1), (v,2), \ldots, (v,d_1)$. Within this construction, the
idea is to put a copy of $G_2$ around each vertex $v$ of
$G_1$, while keeping edges of both $G_1$ and $G_2$.
Every vertex of $G_1\rr G_2$ will be connected to its original neighbors
within its cloud (by edges coming from $G_2$), but also to one
vertex of a different cloud, according to the rotation map of
$G_1$. Note that the degree of $G_1\rr G_2$ depends only on the
degree of the second factor graph $G_2$.
\begin{remark}     \rm
Notice that the definition of $G_1\rr G_2$ depends on the bi-labelling of $G_1$.
In general, there may exist two different bi-labellings of $G_1$, such that
the associated replacement products are non isomorphic graphs \cite[Example 2.3]{abdollahi}.
\end{remark}

\subsection{Zig-zag product of graphs.}\label{subsectionzigzag}

The zig-zag product of two graphs was introduced in \cite{annals}
as a construction which produces, starting from a large graph
$G_1$ and a small graph $G_2$, a new graph $G_1\zig G_2$. This new graph
inherits the size from the large graph $G_1$, the degree from the small
graph $G_2$, and the expansion property from both graphs.
The most important feature of the zig-zag product is that $G_1\zig
G_2$ is a good expander if both $G_1$ and $G_2$ are; see Reingold, Vadhan, Wigderson \cite[Theorem
3.2]{annals}. There it is explicitly described how iteration of the zig-zag construction,
together with the standard squaring, provides
an infinite family of constant-degree expander graphs, starting
from a particular graph representing the building block of this
construction.

\begin{defi}\label{defizigzag}
Let $G_1 = (V_1,E_1)$ be a connected $d_1$-regular graph, and let $G_2 =
(V_2,E_2)$ be a connected $d_2$-regular graph such that $|V_2| = d_1$ (as
usual, graphs are allowed to have loops or multi-edges). Let
$\textrm{Rot}_{G_1}$ (resp. $\textrm{Rot}_{G_2}$) be the
rotation map of $G_1$ (resp. $G_2$). The {\it zig-zag product}
$G_1\zig G_2$ is a regular graph of degree $d_2^2$ with vertex
set $V_1\times V_2$, that we identify with the set $V_1\times
[d_1]$, and whose edges are described by the rotation map
$$
\textrm{Rot}_{G_1\zig G_2}((v,k),(i,j)) = ((w,l),(j',i')),
$$
for all $v\in V_1, k\in [d_1], i,j \in [d_2]$, if:
\begin{enumerate}
\item $\textrm{Rot}_{G_2}(k,i) = (k',i')$,
\item $\textrm{Rot}_{G_1}(v,k') = (w,l')$,
\item $\textrm{Rot}_{G_2}(l',j) = (l,j')$,
\end{enumerate}
where $w\in V_1$, $l,k',l'\in [d_1]$ and $i',j'\in [d_2]$.
\end{defi}
Observe that labels in $G_1\zig G_2$ are elements of $[d_2]^2$. As
in the case of the replacement product, the vertex set of $G_1\zig
G_2$ is partitioned into clouds, indexed by the vertices of $G_1$.
By definition the $v$-cloud consists of vertices
$(v,1), (v,2), \ldots, (v,d_1)$, for every $v\in V_1$. Two
vertices $(v,k)$ and $(w,l)$ of $G_1\zig G_2$ are adjacent in
$G_1\zig G_2$ if it is possible to go from $(v,k)$ to $(w,l)$ by a
sequence of three steps of the following form:
\begin{enumerate}
\item a first step \lq\lq zig\rq\rq within the initial cloud, from the vertex $(v,k)$ to the vertex $(v,k')$,
described by $\textrm{Rot}_{G_2}(k,i) = (k',i')$;
\item a second step jumping from the $v$-cloud to the $w$-cloud, from the vertex $(v,k')$ to the vertex $(w,l')$, described by
$\textrm{Rot}_{G_1}(v,k') = (w,l')$;
\item a third step \lq\lq zag\rq\rq within the new cloud, from the vertex $(w,l')$ to the vertex $(w,l)$, described by $\textrm{Rot}_{G_2}(l',j) =
(l,j')$.
\end{enumerate}

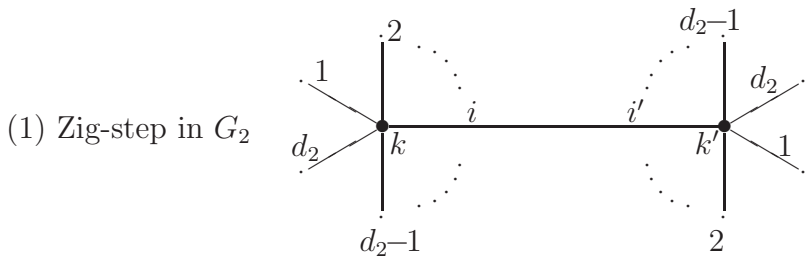
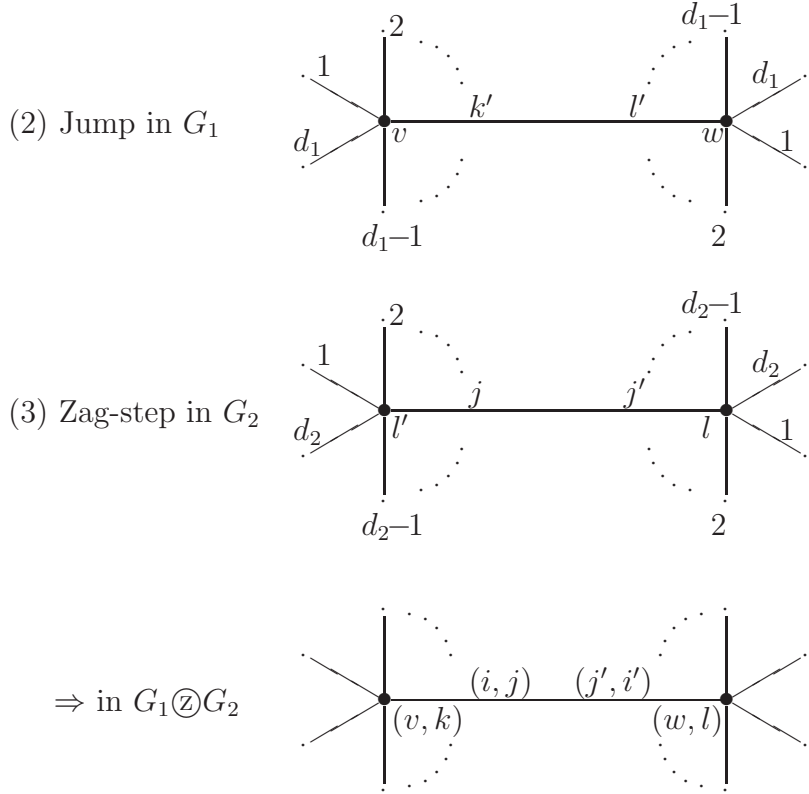
\begin{center}
\begin{picture}(400,110)
\letvertex A=(125,90)\letvertex B=(90,70)\letvertex C=(90,30)
\letvertex D=(125,10)\letvertex E=(125,50)\letvertex F=(275,90)\letvertex G=(275,50)
\letvertex H=(275,10)\letvertex I=(310,30)\letvertex L=(310,70)

\drawvertex(A){\circle*{1}}\drawvertex(B){\circle*{1}}\drawvertex(C){\circle*{1}}\drawvertex(D){\circle*{1}}
\drawundirectededge(E,B){}\drawundirectededge(E,A){}\drawundirectededge(E,C){}\drawundirectededge(E,D){}
\drawvertex(F){\circle*{1}}\drawvertex(L){\circle*{1}}\drawvertex(H){\circle*{1}}\drawvertex(I){\circle*{1}}
\drawundirectededge(G,F){}\drawundirectededge(G,H){}\drawundirectededge(G,I){}\drawundirectededge(G,L){}

\put(-40,45){$(1)$ Zig-step in $G_2$}

\put(162,52){$i$}\put(232,52){$i'$}
\put(128,38){$k$}\put(262,38){$k'$}

\put(95,70){$1$}\put(127,88){$2$}\put(115,-5){$d_2\!\!-\!\!1$}\put(85,37){$d_2$}
\put(286,67){$d_2$}\put(255,93){$d_2\!\!-\!\!1$}\put(268,-5){$2$}\put(298,37){$1$}

\drawvertex(E){$\bullet$}\drawvertex(G){$\bullet$}

\put(153,21){\circle*{1}} \put(148,17){\circle*{1}}
\put(141,15){\circle*{1}}
\put(157,27){\circle*{1}}\put(159,33){\circle*{1}}

\put(153,79){\circle*{1}} \put(148,83){\circle*{1}}
\put(141,85){\circle*{1}}
\put(157,73){\circle*{1}}\put(159,67){\circle*{1}}

\put(247,79){\circle*{1}} \put(252,83){\circle*{1}}
\put(259,85){\circle*{1}}
\put(243,73){\circle*{1}}\put(241,67){\circle*{1}}

\put(247,21){\circle*{1}} \put(252,17){\circle*{1}}
\put(259,15){\circle*{1}}
\put(243,27){\circle*{1}}\put(241,33){\circle*{1}}

%\dashline[0]{4}(125,90)(125,50) \dashline[0]{4}(90,70)(125,50)
%\dashline[0]{4}(90,30)(125,50) \dashline[0]{4}(125,10)(125,50)
%\dashline[0]{4}(275,50)(275,90) \dashline[0]{4}(275,50)(275,10)
%\dashline[0]{4}(310,30)(275,50) \dashline[0]{4}(275,50)(310,70)

\drawundirectededge(E,G){}

\end{picture}
\end{center}

\begin{center}
\begin{picture}(400,110)
\letvertex A=(125,90)\letvertex B=(90,70)\letvertex C=(90,30)
\letvertex D=(125,10)\letvertex E=(125,50)\letvertex F=(275,90)\letvertex G=(275,50)
\letvertex H=(275,10)\letvertex I=(310,30)\letvertex L=(310,70)

\drawvertex(A){\circle*{1}}\drawvertex(B){\circle*{1}}\drawvertex(C){\circle*{1}}\drawvertex(D){\circle*{1}}
\drawundirectededge(E,B){}\drawundirectededge(E,A){}\drawundirectededge(E,C){}\drawundirectededge(E,D){}
\drawvertex(F){\circle*{1}}\drawvertex(L){\circle*{1}}\drawvertex(H){\circle*{1}}\drawvertex(I){\circle*{1}}
\drawundirectededge(G,F){}\drawundirectededge(G,H){}\drawundirectededge(G,I){}\drawundirectededge(G,L){}

\put(-40,45){$(2)$ Jump in $G_1$}

\put(162,52){$k'$}\put(232,52){$l'$}
\put(128,41){$v$}\put(264,41){$w$}

\put(95,70){$1$}\put(127,88){$2$}\put(115,-5){$d_1\!\!-\!\!1$}\put(85,37){$d_1$}
\put(286,67){$d_1$}\put(255,93){$d_1\!\!-\!\!1$}\put(268,-5){$2$}\put(298,37){$1$}

\drawvertex(E){$\bullet$}\drawvertex(G){$\bullet$}

\put(153,21){\circle*{1}} \put(148,17){\circle*{1}}
\put(141,15){\circle*{1}}
\put(157,27){\circle*{1}}\put(159,33){\circle*{1}}

\put(153,79){\circle*{1}} \put(148,83){\circle*{1}}
\put(141,85){\circle*{1}}
\put(157,73){\circle*{1}}\put(159,67){\circle*{1}}

\put(247,79){\circle*{1}} \put(252,83){\circle*{1}}
\put(259,85){\circle*{1}}
\put(243,73){\circle*{1}}\put(241,67){\circle*{1}}

\put(247,21){\circle*{1}} \put(252,17){\circle*{1}}
\put(259,15){\circle*{1}}
\put(243,27){\circle*{1}}\put(241,33){\circle*{1}}

%\dashline[0]{4}(125,90)(125,50) \dashline[0]{4}(90,70)(125,50)
%\dashline[0]{4}(90,30)(125,50) \dashline[0]{4}(125,10)(125,50)
%\dashline[0]{4}(275,50)(275,90) \dashline[0]{4}(275,50)(275,10)
%\dashline[0]{4}(310,30)(275,50) \dashline[0]{4}(275,50)(310,70)

\drawundirectededge(E,G){}
\end{picture}
\end{center}

\begin{center}
\begin{picture}(400,110)
\letvertex A=(125,90)\letvertex B=(90,70)\letvertex C=(90,30)
\letvertex D=(125,10)\letvertex E=(125,50)\letvertex F=(275,90)\letvertex G=(275,50)
\letvertex H=(275,10)\letvertex I=(310,30)\letvertex L=(310,70)

\drawvertex(A){\circle*{1}}\drawvertex(B){\circle*{1}}\drawvertex(C){\circle*{1}}\drawvertex(D){\circle*{1}}
\drawundirectededge(E,B){}\drawundirectededge(E,A){}\drawundirectededge(E,C){}\drawundirectededge(E,D){}
\drawvertex(F){\circle*{1}}\drawvertex(L){\circle*{1}}\drawvertex(H){\circle*{1}}\drawvertex(I){\circle*{1}}
\drawundirectededge(G,F){}\drawundirectededge(G,H){}\drawundirectededge(G,I){}\drawundirectededge(G,L){}

\put(-40,45){$(3)$ Zag-step in $G_2$}

\put(162,54){$j$}\put(230,54){$j'$}
\put(128,38){$l'$}\put(264,38){$l$}

\put(95,70){$1$}\put(127,88){$2$}\put(115,-5){$d_2\!\!-\!\!1$}\put(85,37){$d_2$}
\put(286,67){$d_2$}\put(255,93){$d_2\!\!-\!\!1$}\put(268,-5){$2$}\put(298,37){$1$}

\drawvertex(E){$\bullet$}\drawvertex(G){$\bullet$}

\put(153,21){\circle*{1}} \put(148,17){\circle*{1}}
\put(141,15){\circle*{1}}
\put(157,27){\circle*{1}}\put(159,33){\circle*{1}}

\put(153,79){\circle*{1}} \put(148,83){\circle*{1}}
\put(141,85){\circle*{1}}
\put(157,73){\circle*{1}}\put(159,67){\circle*{1}}

\put(247,79){\circle*{1}} \put(252,83){\circle*{1}}
\put(259,85){\circle*{1}}
\put(243,73){\circle*{1}}\put(241,67){\circle*{1}}

\put(247,21){\circle*{1}} \put(252,17){\circle*{1}}
\put(259,15){\circle*{1}}
\put(243,27){\circle*{1}}\put(241,33){\circle*{1}}

%\dashline[0]{4}(125,90)(125,50) \dashline[0]{4}(90,70)(125,50)
%\dashline[0]{4}(90,30)(125,50) \dashline[0]{4}(125,10)(125,50)
%\dashline[0]{4}(275,50)(275,90) \dashline[0]{4}(275,50)(275,10)
%\dashline[0]{4}(310,30)(275,50) \dashline[0]{4}(275,50)(310,70)

\drawundirectededge(E,G){}

\end{picture}
\end{center}

\begin{center}
\begin{picture}(400,110)
\letvertex A=(125,90)\letvertex B=(90,70)\letvertex C=(90,30)
\letvertex D=(125,10)\letvertex E=(125,50)\letvertex F=(275,90)\letvertex G=(275,50)
\letvertex H=(275,10)\letvertex I=(310,30)\letvertex L=(310,70)

\put(-20,45){$\Rightarrow$ in $G_1\zig G_2$}

\put(162,54){$(i,j)$}\put(208,54){$(j',i')$}
\put(128,37){$(v,k)$}\put(242,37){$(w,l)$}

\drawvertex(A){\circle*{1}}\drawvertex(B){\circle*{1}}\drawvertex(C){\circle*{1}}\drawvertex(D){\circle*{1}}
\drawundirectededge(E,B){}\drawundirectededge(E,A){}\drawundirectededge(E,C){}\drawundirectededge(E,D){}
\drawvertex(F){\circle*{1}}\drawvertex(L){\circle*{1}}\drawvertex(H){\circle*{1}}\drawvertex(I){\circle*{1}}
\drawundirectededge(G,F){}\drawundirectededge(G,H){}\drawundirectededge(G,I){}\drawundirectededge(G,L){}

\drawvertex(E){$\bullet$}\drawvertex(G){$\bullet$}

\put(148,18){\circle*{1}} \put(143,14){\circle*{1}}
\put(136,12){\circle*{1}}
\put(152,24){\circle*{1}}\put(154,30){\circle*{1}}

\put(148,82){\circle*{1}} \put(143,86){\circle*{1}}
\put(136,88){\circle*{1}}
\put(152,76){\circle*{1}}\put(154,70){\circle*{1}}

\put(252,82){\circle*{1}} \put(257,86){\circle*{1}}
\put(264,88){\circle*{1}}
\put(248,76){\circle*{1}}\put(246,70){\circle*{1}}

\put(252,18){\circle*{1}} \put(257,14){\circle*{1}}
\put(264,12){\circle*{1}}
\put(248,24){\circle*{1}}\put(246,30){\circle*{1}}

%\dashline[0]{4}(125,90)(125,50) \dashline[0]{4}(90,70)(125,50)
%\dashline[0]{4}(90,30)(125,50) \dashline[0]{4}(125,10)(125,50)
%\dashline[0]{4}(275,50)(275,90) \dashline[0]{4}(275,50)(275,10)
%\dashline[0]{4}(310,30)(275,50) \dashline[0]{4}(275,50)(310,70)
\drawundirectededge(E,G){}
\end{picture}
\end{center}

From the definition of the replacement and the zig-zag
product it follows that the edges of $G_1\zig G_2$ arise from paths of length
$3$ in $G_1\rr G_2$ of type:
\begin{enumerate}
\item a first step within one cloud (the zig-step);
\item a second step which is a jump to a new cloud;
\item a third step within the new cloud (the zag-step).
\end{enumerate}
In other words, $G_1\zig G_2$ is a regular subgraph of the graph
obtained by taking the third power of $G_1\rr G_2$. This fact can
be explicitly expressed in terms of normalized adjacency matrices. More precisely, let $A_1'$ (resp. $A_2'$) be the
normalized adjacency matrix of the graph $G_1$ (resp. $G_2$), and suppose that $|V_1|=n_1$. Then
the normalized adjacency matrix of $G_1 \zig G_2$ is
$M_{\zig}=\widetilde{A}_2\widetilde{A}_1\widetilde{A}_2$ (see
\cite{annals}), with $\widetilde{A}_2 = I_{n_1}\otimes A_2'$,
where the symbol $\otimes$ denotes the tensor product, or
Kronecker product, and $\widetilde{A}_1$ is the permutation matrix
on $V_1\times [d_1]$ associated with the map $\textrm{Rot}_{G_1}$,
i.e,
$$
\widetilde{A}_{1 \!\ {(v,k),(w,l)}} =
\begin{cases}
1&\text{if}\ \rot_{G_1}(v,k) = (w,l)\\
0&\text{otherwise}.
\end{cases}
$$
The matrix $\widetilde{A}_1$ has exactly one entry $1$ in each row
and each column and $0$'s elsewhere. Note that $\widetilde{A}_1$
and $\widetilde{A}_2$ are both symmetric matrices, due to the
undirectedness of $G_1$ and $G_2$. On the other hand, it is easy to check that the normalized
adjacency matrix of $G_1 \rr G_2$ is
$M_{\rr}=\frac{\widetilde{A}_1+d_2\widetilde{A}_2}{d_2+1}$, and
that the following decomposition holds:
$$
M_{\rr}^3 = \frac{d_2^2}{(d_2+1)^3}M_{\zig} +
\left(1-\frac{d_2^2}{(d_2+1)^3}\right)C,
$$
where $C$ is the normalized adjacency matrix of a regular graph.

\begin{remark}\rm
Note that the replacement product $G_1\rr G_2$ and the zig-zag product $G_1\zig G_2$ are defined for finite connected regular graphs $G_1$ and $G_2$.
It follows from the definition of replacement product that the graph $G_1\rr G_2$ is connected; on the other hand, the connectedness of
$G_1$ and $G_2$ does not ensure the connectedness of the graph $G_1\zig G_2$. One of the goals of the current work is to investigate this property for the zig-zag construction.
\end{remark}

Recall that a \textit{complete bipartite graph} $G=(V,E)$ is a
graph whose vertex set can be partitioned into two subsets $U_1$
and $U_2$ such that, for every two vertices $u_1\in U_1$, $u_2\in
U_2$, one has $\{u_1,u_2\}\in E$, but there is no edge joining two vertices belonging to the same subset $U_i$. A complete bipartite graph is
usually denoted by $K_{m,n}$, if $|U_1|=m$ and $|U_2|=n$.

The following basic result will be very useful for the rest of the
paper. It shows that the graph $G_1\zig G_2$ consists of unions of special \lq\lq elementary blocks\rq\rq, each isomorphic
to a complete bipartite graph.
\begin{lemma}\label{papillonlemma}
Let $G_1$ be a $d_1$-regular graph, and let $G_2$ be a $d_2$-regular
graph on $d_1$ vertices. Suppose that the vertices $v$ and $w$ are
adjacent in $G_1$, with $\rot_{G_1}(v,k)= (w,l)$, and $k,l\in
[d_1]$. Let $\{k_1, \ldots, k_{d_2}\}$ be the set of vertices adjacent
to $k$ in $G_2$; similarly, let $\{l_1, \ldots, l_{d_2}\}$ be the set
of vertices adjacent to $l$ in $G_2$. Then the edge connecting $v$
and $w$ in $G_1$ produces in $G_1\zig G_2$ the following subgraph
isomorphic to $K_{d_2,d_2}$:
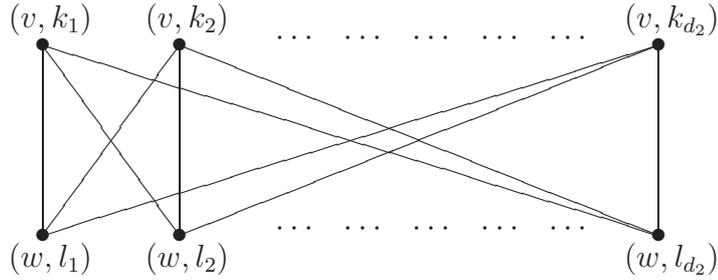
\begin{center}
\begin{picture}(300,110)\unitlength=0.18mm
\letvertex A=(25,160)\letvertex B=(125,160)\letvertex C=(475,160)
\letvertex a=(25,20)\letvertex b=(125,20)\letvertex c=(475,20)

\drawvertex(A){$\bullet$}\drawvertex(B){$\bullet$}
\drawvertex(C){$\bullet$}\drawvertex(a){$\bullet$}\drawvertex(b){$\bullet$}
\drawvertex(c){$\bullet$}

\drawundirectededge(A,a){}\drawundirectededge(A,b){}\drawundirectededge(A,c){}
\drawundirectededge(B,a){}\drawundirectededge(B,b){}\drawundirectededge(B,c){}
\drawundirectededge(C,a){}\drawundirectededge(C,b){}\drawundirectededge(C,c){}

\put(0,173){$(v,k_1)$}
\put(100,173){$(v,k_2)$}\put(450,173){$(v,k_{d_2})$}
\put(0,-3){$(w,l_1)$}
\put(100,-3){$(w,l_2)$}\put(450,-3){$(w,l_{d_2})$}

\put(195,160){$\cdots$}\put(245,160){$\cdots$}\put(295,160){$\cdots$}\put(345,160){$\cdots$}\put(395,160){$\cdots$}
\put(195,20){$\cdots$}\put(245,20){$\cdots$}\put(295,20){$\cdots$}\put(345,20){$\cdots$}\put(395,20){$\cdots$}
\end{picture}
\end{center}
\end{lemma}

\begin{proof}
The hypothesis ensures that the graphs $G_1$ and $G_2$ contain the
subgraphs depicted below.
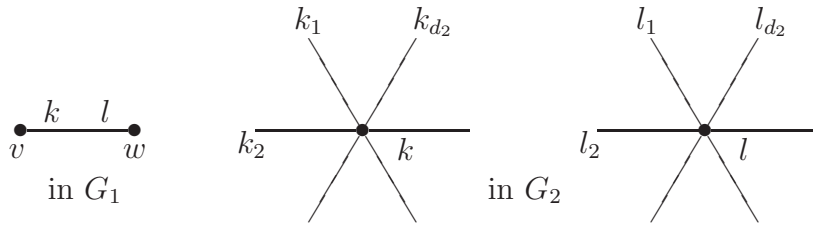
\begin{center}
\begin{picture}(450,90)
\letvertex A=(50,50)\letvertex B=(100,50)

\letvertex C=(175,93)
\letvertex D=(150,50)\letvertex E=(175,7)\letvertex F=(225,7)\letvertex G=(250,50)
\letvertex H=(225,93)\letvertex I=(200,50)

\letvertex c=(325,93)
\letvertex d=(300,50)\letvertex e=(325,7)\letvertex f=(375,7)\letvertex g=(400,50)
\letvertex h=(375,93)\letvertex i=(350,50)

\drawvertex(A){$\bullet$}\drawvertex(B){$\bullet$}
\drawvertex(I){$\bullet$}\drawvertex(i){$\bullet$}

\drawundirectededge(A,B){}

\drawundirectededge(I,C){}\drawundirectededge(I,D){}\drawundirectededge(I,E){}
\drawundirectededge(I,F){}\drawundirectededge(I,G){}\drawundirectededge(I,H){}

\drawundirectededge(i,c){}\drawundirectededge(i,d){}\drawundirectededge(i,e){}
\drawundirectededge(i,f){}\drawundirectededge(i,g){}\drawundirectededge(i,h){}

\put(45,38){$v$}\put(95,38){$w$}\put(60,54){$k$}\put(85,54){$l$}

\put(215,37){$k$}\put(170,95){$k_1$}\put(145,40){$k_2$}\put(222,95){$k_{d_2}$}
\put(365,37){$l$}\put(320,95){$l_1$}\put(295,40){$l_2$}\put(372,95){$l_{d_2}$}

\put(62,20){in $G_1$}\put(255,20){in $G_2$}
\end{picture}
\end{center}
Now it suffices to apply the definition of the zig-zag product in order to get the assertion.
\end{proof}
\begin{remark}\rm
In the case where $G_2$ is a $2$-regular graph, for instance, if
$G_2$ is a cycle graph, we call the graph $K_{2,2}$ the
\textit{papillon graph}. We will say that the vertices $(v,k_1), (v,k_2), (w,l_1), (w,l_2)$ form the papillon graph depicted below.
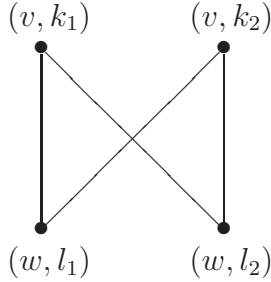
\begin{center}
\begin{picture}(360,115)
\letvertex A=(140,100)\letvertex B=(140,20)\letvertex C=(220,100)
\letvertex D=(220,20)

\put(125,110){$(v,k_1)$} \put(125,2){$(w,l_1)$}
\put(205,110){$(v,k_2)$} \put(205,2){$(w,l_2)$}

\drawvertex(A){$\bullet$}\drawvertex(B){$\bullet$}
\drawvertex(C){$\bullet$}\drawvertex(D){$\bullet$}

\drawundirectededge(A,B){}\drawundirectededge(A,D){}\drawundirectededge(C,B){}\drawundirectededge(C,D){}
\end{picture}
\end{center}
\end{remark}

\begin{example}\label{cubecyclicreplacement}\rm
Let $G_1=(V_1,E_1)$ be the $3$-dimensional Hamming cube, so that
$V_1=\{0,1\}^3$ is the set of binary words of length $3$, and two
words $u=x_0x_1x_2$ and $v=y_0y_1y_2$ are adjacent if and only if
$x_i=y_i$ for all but one index $i\in \{0,1,2\}$. Observe that
$G_1$ can be interpreted as the Cayley graph of the group
$\mathbb{Z}_2^3$ with respect to the generating set ${\bf E}_3 =
\{{\bf e}_0,{\bf e}_1, {\bf e}_2\}$, where ${\bf e}_i$ denotes the
triple with $1$ at the $i$-th coordinate and $0$
elsewhere. Now let $G_2=(V_2,E_2)$ be the cycle graph of length $3$, which can be
interpreted as the Cayley graph of the cyclic group
$\mathbb{Z}_3=\{0,1,2\}$, with respect to the generating set
$\{\pm 1\}$ (see Figure \ref{figure7}).

\begin{figure}[h]
\begin{picture}(400,135)\unitlength=0,3mm
\letvertex A=(40,90)\letvertex B=(40,10)\letvertex C=(120,10)\letvertex D=(160,50)\letvertex E=(160,130)\letvertex F=(80,130)\letvertex G=(80,50)
\letvertex H=(120,90)

\letvertex I=(260,20)\letvertex L=(340,20)\letvertex M=(300,90)

\put(70,135){$001$}\put(20,95){$000$}\put(20,-5){$010$}\put(110,-5){$110$}\put(165,45){$111$}\put(155,135){$101$}
\put(105,95){$100$}\put(77,37){$011$}\put(255,7){$1$}\put(336,7){$2$}\put(296,95){$0$}

\drawvertex(A){$\bullet$}\drawvertex(B){$\bullet$}\drawvertex(C){$\bullet$}\drawvertex(D){$\bullet$}
\drawvertex(E){$\bullet$}\drawvertex(F){$\bullet$}\drawvertex(G){$\bullet$}\drawvertex(I){$\bullet$}\drawvertex(M){$\bullet$}
\drawvertex(H){$\bullet$}\drawvertex(L){$\bullet$}

\dashline[0]{4}(80,130)(80,50)\dashline[0]{4}(40,10)(80,50)
\dashline[0]{4}(160,50)(80,50)

\drawundirectededge(A,B){}\drawundirectededge(A,H){}\drawundirectededge(A,F){}
\drawundirectededge(B,C){}\drawundirectededge(C,H){}\drawundirectededge(C,D){}\drawundirectededge(D,E){}
\drawundirectededge(E,H){}\drawundirectededge(E,F){}
\drawundirectededge(I,L){}\drawundirectededge(L,M){}
\drawundirectededge(M,I){}
\end{picture}\caption{The graphs $Cay(\mathbb{Z}_2^3,{\bf E}_3)$ and $Cay(\mathbb{Z}_3,\{\pm
1\})$.}\label{figure7}
\end{figure}
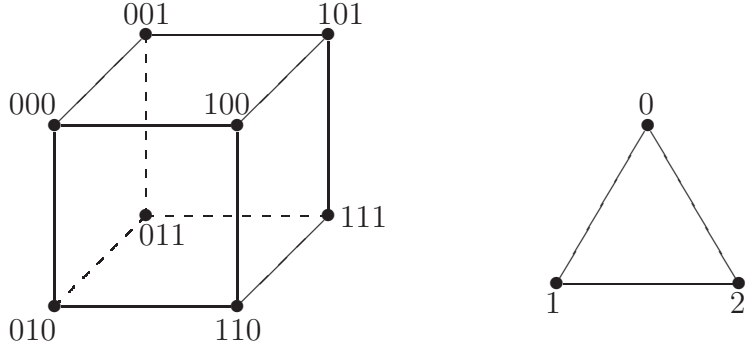

Having these interpretations in our mind, we label the edges of
$G_1$ as follows: the edge connecting two vertices $u$ and $v$ is
labelled by $i$ both near $u$ and $v$, if the corresponding words
$u=x_0x_1x_2$ and $v=y_0y_1y_2$ differ in the $i$-th letter (this
corresponds to moving by using the generator ${\bf e}_i$ in the
Cayley graph of $\mathbb{Z}_2^3$). Similarly, we label the edges of $G_2$ in such a way that
$\rot_{G_2}(u,\pm 1) = (u\pm 1,\mp 1)$, where the integers $u,u\pm
1,\pm 1, \mp 1 $ are taken modulo $3$ (this corresponds to moving
by using the generators $\pm 1$ in the Cayley graph of
$\mathbb{Z}_3$). In the replacement product, every vertex of $\mathbb{Z}_2^3$ is
replaced by a cloud of $3$ vertices representing a copy of
$\mathbb{Z}_3$. Moreover, each vertex of any cloud is connected to
exactly one vertex of a neighboring cloud according with the
following rule: the vertex $(v, i)$ is connected to the vertex $(v
+ {\bf e}_i, i)$.
The replacement product is depicted in Figure
\ref{figure8}. It is worth mentioning that the replacement product
$Cay(\mathbb{Z}_2^3,{\bf E}_3)\rr Cay(\mathbb{Z}_3,\{\pm 1\})$ is
isomorphic to the Cayley graph of the semidirect product
$\mathbb{Z}_2^3\rtimes \mathbb{Z}_3$, with respect to the
generating set $\{(000,\pm 1),({\bf e}_0,0)\}$ (see
\cite{alfredoIJGT}).

\begin{center}
\begin{figure}[h]
\begin{picture}(250,260)\unitlength=0,18mm
\letvertex A=(10,340)\letvertex B=(10,60)\letvertex C=(60,10)\letvertex D=(340,10)\letvertex E=(390,60)\letvertex F=(390,340)\letvertex G=(340,390)
\letvertex H=(60,390)\letvertex I=(60,340)\letvertex L=(60,60)\letvertex M=(340,60)\letvertex N=(340,340)\letvertex O=(120,280)
\letvertex P=(120,230)\letvertex Q=(120,170)\letvertex R=(120,120)\letvertex S=(170,120)\letvertex T=(230,120)\letvertex U=(280,120)
\letvertex V=(280,170)\letvertex Z=(280,230)\letvertex X=(280,280)\letvertex Y=(230,280)\letvertex W=(170,280)

\put(45,398){$000,0$}\put(-50,335){$000,1$}\put(68,335){$000,2$}\put(40,-10){$010,0$}\put(-50,55){$010,1$}\put(70,55){$010,2$}
\put(325,398){$100,0$}\put(393,335){$100,1$}\put(278,335){$100,2$}\put(320,-10){$110,0$}\put(393,55){$110,1$}\put(273,55){$110,2$}
\put(58,223){$001,1$}\put(140,287){$001,0$}\put(58,272){$001,2$}\put(58,116){$011,2$}\put(140,100){$011,0$}\put(58,165){$011,1$}
\put(285,116){$111,2$}\put(215,100){$111,0$}\put(285,165){$111,1$}\put(285,271){$101,2$}\put(215,287){$101,0$}\put(285,225){$101,1$}

\drawvertex(A){$\bullet$}\drawvertex(B){$\bullet$}\drawvertex(C){$\bullet$}\drawvertex(D){$\bullet$}
\drawvertex(E){$\bullet$}\drawvertex(F){$\bullet$}\drawvertex(G){$\bullet$}\drawvertex(I){$\bullet$}\drawvertex(M){$\bullet$}\drawvertex(N){$\bullet$}
\drawvertex(H){$\bullet$}\drawvertex(L){$\bullet$}\drawvertex(O){$\bullet$}\drawvertex(T){$\bullet$}
\drawvertex(P){$\bullet$}\drawvertex(U){$\bullet$}\drawvertex(Q){$\bullet$}\drawvertex(V){$\bullet$}
\drawvertex(R){$\bullet$}\drawvertex(Z){$\bullet$}\drawvertex(Y){$\bullet$}\drawvertex(Z){$\bullet$}
\drawvertex(S){$\bullet$}\drawvertex(X){$\bullet$}\drawvertex(W){$\bullet$}

\drawundirectededge(A,B){}\drawundirectededge(A,H){}\drawundirectededge(A,I){}\drawundirectededge(B,L){}
\drawundirectededge(B,C){}\drawundirectededge(C,L){}\drawundirectededge(C,D){}\drawundirectededge(D,E){}
\drawundirectededge(D,M){}\drawundirectededge(E,M){}\drawundirectededge(E,F){}
\drawundirectededge(F,G){}\drawundirectededge(F,N){}\drawundirectededge(G,N){}\drawundirectededge(G,H){}
\drawundirectededge(H,I){}\drawundirectededge(I,O){}\drawundirectededge(O,P){}\drawundirectededge(O,W){}\drawundirectededge(P,Q){}
\drawundirectededge(Q,R){}\drawundirectededge(Q,S){}\drawundirectededge(R,L){}
\drawundirectededge(R,S){}\drawundirectededge(S,T){}\drawundirectededge(T,U){}\drawundirectededge(T,V){}
\drawundirectededge(U,M){}\drawundirectededge(U,V){}\drawundirectededge(V,Z){}\drawundirectededge(Z,Y){}\drawundirectededge(Z,X){}
\drawundirectededge(N,X){}\drawundirectededge(X,Y){}\drawundirectededge(Y,W){}\drawundirectededge(P,W){}
\end{picture}\caption{The graph $Cay(\mathbb{Z}_2^3,{\bf E}_3)\rr Cay(\mathbb{Z}_3,\{\pm
1\})$.}\label{figure8}
\end{figure}\end{center}

The zig-zag product $G_1\zig G_2$ is represented in Figure
\ref{figure9}. It is known that it can be regarded as the Cayley
graph of the group $\mathbb{Z}_2^3\rtimes \mathbb{Z}_3$, with
respect to the generating set $\{({\bf e}_1,2),({\bf e}_1,0),({\bf
e}_2,0),({\bf e}_2,1)\}$ (see \cite{alfredoIJGT}). Observe that
one edge in this graph is obtained as a sequence of three steps in
the graph $G_1\rr G_2$ in Figure \ref{figure8}: the first step
within some initial cloud, the second step jumping to a new cloud
and finally a third step within the new cloud. For instance, the
three steps $(111,1)\to (111,2)\to(110,2)\to (110,0)$ produce the
edge connecting the vertices $(111,1)$ and $(110,0)$ in $G_1\zig
G_2$. The papillon subgraphs structure of $G_1\zig G_2$ is
well-rendered in Figure \ref{figure9}.

\begin{figure}[h]
\begin{picture}(300,340)\unitlength=0,22mm
\letvertex A=(10,340)\letvertex B=(10,60)\letvertex C=(60,10)\letvertex D=(340,10)\letvertex E=(390,60)\letvertex F=(390,340)\letvertex G=(340,390)
\letvertex H=(60,390)\letvertex I=(60,340)\letvertex L=(60,60)\letvertex M=(340,60)\letvertex N=(340,340)\letvertex O=(170,280)
\letvertex P=(120,230)\letvertex Q=(120,170)\letvertex R=(170,120)\letvertex S=(170,170)\letvertex T=(230,170)\letvertex U=(230,120)
\letvertex V=(280,170)\letvertex Z=(280,230)\letvertex X=(230,280)\letvertex Y=(230,230)\letvertex W=(170,230)

\put(35,397){$000,1$}\put(-32,348){$000,0$}\put(24,348){$000,2$}
\put(35,-9){$010,1$}\put(-29,42){$010,0$}\put(24,43){$010,2$}
\put(325,397){$100,1$}\put(387,348){$100,0$}\put(330,348){$100,2$}
\put(330,-9){$110,1$}\put(388,42){$110,0$}\put(332,43){$110,2$}
\put(67,225){$001,0$}\put(127,263){$001,1$}\put(130,235){$001,2$}
\put(129,153){$011,2$}\put(128,125){$011,1$}\put(67,165){$011,0$}
\put(226,154){$111,2$}\put(227,125){$111,1$}\put(285,165){$111,0$}
\put(226,235){$101,2$}\put(231,264){$101,1$}\put(285,224){$101,0$}

\drawvertex(A){$\bullet$}\drawvertex(B){$\bullet$}\drawvertex(C){$\bullet$}\drawvertex(D){$\bullet$}
\drawvertex(E){$\bullet$}\drawvertex(F){$\bullet$}\drawvertex(G){$\bullet$}\drawvertex(I){$\bullet$}\drawvertex(M){$\bullet$}\drawvertex(N){$\bullet$}
\drawvertex(H){$\bullet$}\drawvertex(L){$\bullet$}\drawvertex(O){$\bullet$}\drawvertex(T){$\bullet$}
\drawvertex(P){$\bullet$}\drawvertex(U){$\bullet$}\drawvertex(Q){$\bullet$}\drawvertex(V){$\bullet$}
\drawvertex(R){$\bullet$}\drawvertex(Z){$\bullet$}\drawvertex(Y){$\bullet$}\drawvertex(Z){$\bullet$}
\drawvertex(S){$\bullet$}\drawvertex(X){$\bullet$}\drawvertex(W){$\bullet$}

\drawundirectededge(A,B){}\drawundirectededge(A,P){}\drawundirectededge(A,L){}\drawundirectededge(A,O){}
\drawundirectededge(B,I){}\drawundirectededge(B,R){}\drawundirectededge(B,Q){}\drawundirectededge(C,Q){}
\drawundirectededge(C,R){}\drawundirectededge(C,M){}\drawundirectededge(C,D){}
\drawundirectededge(D,U){}\drawundirectededge(D,V){}\drawundirectededge(D,L){}\drawundirectededge(E,N){}
\drawundirectededge(E,F){}\drawundirectededge(E,U){}\drawundirectededge(E,V){}\drawundirectededge(F,M){}\drawundirectededge(F,X){}
\drawundirectededge(F,Z){}\drawundirectededge(G,H){}\drawundirectededge(G,I){}
\drawundirectededge(G,X){}\drawundirectededge(G,Z){}\drawundirectededge(H,O){}\drawundirectededge(H,P){}
\drawundirectededge(H,N){}\drawundirectededge(I,L){}\drawundirectededge(I,N){}\drawundirectededge(L,M){}\drawundirectededge(M,N){}
\drawundirectededge(O,X){}\drawundirectededge(O,Y){}\drawundirectededge(P,Q){}\drawundirectededge(P,S){}

\drawundirectededge(Q,W){}\drawundirectededge(R,T){}\drawundirectededge(R,U){}\drawundirectededge(S,W){}
\drawundirectededge(S,T){}\drawundirectededge(S,U){}\drawundirectededge(T,Y){}\drawundirectededge(T,Z){}\drawundirectededge(V,Y){}
\drawundirectededge(V,Z){}\drawundirectededge(X,W){}\drawundirectededge(Y,W){}
\end{picture}\caption{The graph $Cay(\mathbb{Z}_2^3,{\bf E}_3)\zig Cay(\mathbb{Z}_3,\{\pm
1\})$.}\label{figure9}
\end{figure}
\end{example}

\section{Connectedness}\label{sectionconnectedness}
In this section we discuss the connectedness problem of the
zig-zag product of graphs. We will see (Example \ref{example6}), that in general
such graph product is not connected. The upcoming result relates the investigation of the
connectedness properties of the graph $G_1\zig G_2$ to the study of a new graph that one constructs
starting from $G_2$, independently of $G_1$.

Like before, let $G_1=(V_1,E_1)$ be a $d$-regular  connected graph and
$G_2=(V_2,E_2)$, with $d=|V_2|$. For any $h\in V_2$, we put
$N_h=\{h'\in V_2 : h\sim h'\}$.
Now we associate with $G_2$ a new graph
$\mathcal{N}=(\mathcal{V},\mathcal{E})$, called the \textit{neighborhood graph} of $G_2$.

The neighborhood graph $\mathcal{N}=(\mathcal{V},\mathcal{E})$ of $G_2$ is defined by:
\begin{itemize}
\item $\mathcal{V}= \{N_h, h\in V_2\}$;
\item $\mathcal{E} = \big\{ \{N_h,N_k\} : N_h \cap N_k\neq \emptyset, h\neq k\big\}$.
\end{itemize}
In other words, two vertices $N_h, N_k$ of $\mathcal{N}$ are adjacent if $h$ and $k$ have at least a
common neighbor in $G_2$. We do not put any label on the edges of $\mathcal{N}$.

\begin{thm}\label{teoremaconnessione}
Let $G_1=(V_1,E_1)$ be a $d$-regular graph and let
$G_2=(V_2,E_2)$, with $d=|V_2|$. If the neighborhood graph $\mathcal{N}$ of $G_2$ is connected,
then $G_1\zig G_2$ is connected as well.
\end{thm}
\begin{proof}
It is enough to show that for all $h, h'\in V_2$ and $v\in V_1$,
the vertices $(v,h)$ and $(v,h')$ are in the same connected
component of $G_1\zig G_2$. In fact, if $v\neq v'$ are adjacent
vertices in $G_1$, there exist $k,k'\in V_2$ such that
$\rot_{G_1}(v,k)=(v',k')$. This implies that the vertices
$(v,\widetilde{k})$ and $(v', \widetilde{k'})$ are connected in
$G_1\zig G_2$, for all $\widetilde{k}\in N_k$ and
$\widetilde{k'}\in N_{k'}$. Hence, if two vertices $v$ and $w$ are
connected by a path in $G_1$, there exists a path in
 $G_1\zig G_2$ connecting $(v,
h)$ and $(w,k)$, for suitable $h,k\in V_2$. By combining this property
with the fact that $(v,h)$ and $(v,h')$ are in the same connected
component of $G_1\zig G_2$ for every $v\in V_1, h,h'\in V_2$, we get the assertion.

So, let us show now that $(v,h)$ and $(v,h')$ belong to the same connected component. Since
$\mathcal{N}$ is connected, there exists
a sequence $N_{h_1},\ldots, N_{h_i}$ such that $h\in N_{h_1}$,
$h'\in N_{h_i}$ and $N_{h_j}\cap N_{h_{j+1}}\neq \emptyset$, for
$j=1,\ldots,i-1$. Let $k_j\in
N_{h_j}\cap N_{h_{j+1}}$. Notice that $(v,h)$ and $(v,k_1)$ are
connected in $G_1\zig G_2$, since they have the common neighbor $(w,s)$, where
$\rot_{G_1}(v,h_1)=(w,s')$ and $s\in N_{s'}$. The same can be said
for $(v,k_1)$ and $(v,k_2)$, as they share a neighbor of the form
$(u,t)$ where $\rot_{G_1}(v,h_2)=(u,t')$ and $t\in N_{t'}$. By using the same argument,
we can say that $(v,k_{i-1})$ and $(v,h')$ are in the same connected
component of $G_1\zig G_2$. This ensures that $(v,h), (v,k_1), \ldots, (v,k_{i-1}), (v,h')$
(and in particular $(v,h)$ and $(v,h')$)
are in the same connected component of $G_1\zig G_2$ and this concludes the proof.
\end{proof}
From now on, the complete graph on $n$ vertices will be denoted by
$K_n$, and the cycle graph on $n$ vertices will be denoted by
$C_n$.
\begin{cor}\label{corollariomisterioso}
Let $G=(V,E)$ be a $d$-regular graph. Then for every $d\geq 3$,
the graph $G\zig K_d$ is connected. If $d\geq 3$ is odd, then the
graph $G\zig C_d$ is connected.
\end{cor}
\begin{proof}
It suffices to observe that the neighborhood graph associated with
$K_d$ is isomorphic to $K_d$ itself; similarly, it is
straightforward to check that the neighborhood graph associated
with $C_d$ is isomorphic to $C_d$.
\end{proof}
\begin{remark}\rm
Notice that the condition of the previous theorem is not a
necessary condition. In the case of the Schreier graphs
$\{\Gamma_n\}_{n\geq 1}$ of the Basilica group discussed in
Section \ref{Basilicasection}, the zig-zag product $\Gamma_n \zig
C_4$ is connected, for every $n\geq 1$, even if the neighborhood graph $\mathcal{N}$
associated with the cycle graph $C_4$ consists of two connected
components.
\end{remark}

\section{Isomorphism properties}\label{sectionclassifications}
In what follows we focus our attention on the case when the factor graph $G_2$ in the zig-zag product $G_1\zig G_2$
is a cycle graph. This assumption allows us to give precise results about the
structure of the connected components of $G_1\zig G_2$.

We have seen in Corollary \ref{corollariomisterioso} that if
$d\geq 3$ is an odd integer, then the zig-zag product $G\zig C_d$
is always connected, independently of the bi-labelling of the edges of
$G$. For this reason, our analysis will be restricted to the case
$G\zig C_d$, with an even $d$.

\subsection{Parity blocks}\label{subsectionblocks}
Let $G=(V,E)$ be a $d$-regular bi-labelled graph, where $d$
is an even natural number; recall that an edge joining two
vertices $u$ and $w$ in $G$ is colored by some color $i$ near $u$
and by some color $j$ near $w$. As usual, we identify the set of
colors with the set $[d]$. Let $[d_e]$ (resp. $[d_o]$) be the
subset of $[d]$ consisting of the even (resp. odd) numbers
from $1$ to $d$, so that $[d]=[d_e]\sqcup [d_o]$. Given $v\in V$,
and chosen one of the sets $[d_i]$, $i\in \{e,o\}$, the
\textit{parity block} $P(v,i)=(V(v,i),E(v,i))$ is the subgraph of
$G$ defined as follows:
\begin{itemize}
\item $V(v,i)$ is the set of all vertices $w\in V$ with the property
that there exists a path
$\mathcal{P}=\{v=v_0,v_1,\ldots,v_{n-1},w=v_n\}$ in $G$ such that
the following \textit{parity properties} are satisfied:
\begin{enumerate}
\item $\rot_G(v_k,i_k)=(v_{k+1},j_k)$, for $k=0,\ldots, n-1$;
\item $i_0\in [d_i]$; \item $i_{k+1}\equiv j_k \pmod{2}$;
\end{enumerate}
\item $E(v,i)$ consists of the edges joining two consecutive vertices $v_k$ and
$v_{k+1}$ in $\mathcal{P}$, and bi-labelled according with the bi-labelling of $G$, described by the
rotation map $\rot_G(v_k,i_k)=(v_{k+1},j_k)$.
\end{itemize}
A vertex $w\in P(v,i)$ is said to be \textit{even} or with parity
$e$ (resp. \textit{odd} or with parity $o$) in $P(v,i)$ if the
path $\{v=v_0,v_1,\ldots,v_{n-1},w=v_n\}$ is such that $j_{n-1}\in
[d_e]$ (resp. $j_{n-1}\in [d_o]$). If a vertex $w$ is both even
and odd, we will say that $w$ is \textit{odden} or with parity
$e-o$. In other words, the vertex $w$ is odden if and only if
$P(w,e)$ and $P(w,o)$ coincide.

Since $G$ is finite, $G$ decomposes into a finite number of
$P(v,i)$'s, in the sense that every edge in $E$ belongs to some
graph $P(v,i)$ for some $v\in V$ and $i\in \{e,o\}$. We write
$G=\cup_{j} P(v_j,i_j)$ where $j$ runs over an opportune finite
index set. Notice that a vertex which is either even or odd in a
parity block has degree $d/2$ in that parity block, whereas an odden vertex has
degree $d$ in the parity block.
\begin{lemma}
Let $P(v,i)$ be a parity block of $G=(V,E)$. If $w$ is a
vertex in $G$ with parity $i_w$ in $P(v,i)$, then $P(v,i) =
P(w,i_w)$.
\end{lemma}
\begin{proof}
By definition $w$ has parity $i_w$ in $P(v,i)$ if there
exists a path $\{v=v_0,v_1,\ldots,v_{n-1},w \\
=v_n\}$ in
$P(v,i)$, with $\rot_G(v_k,i_k)=(v_{k+1},j_k)$, such that $i_0\in [d_i]$ and
$j_{n-1}\in [d_{i_w}]$.
In order to prove our statement, it is enough to show that $v\in
P(w, i_w)$. Notice that the inverse path
$\{w=w_0,w_1,\ldots,w_{n-1},v=w_n\}$, with $w_i = v_{n-i}$,
satisfies the parity conditions, with the property that $v\in
P(w,i_w)$ with parity $i_v=i$, as $\rot_G(w_{n-1},j_0) = (v,i_0)$.
\end{proof}

This result ensures that, given a bi-labelled graph $G$, a
decomposition of $G$ into parity blocks is uniquely determined, and we are allowed to use the
notation $G=\cup_{j} P_j$,
without explicitly expressing the dependence of the parity blocks
on the particular vertices. Given a bi-labelled graph $G$, we will
refer to its decomposition into parity blocks as its
\textit{parity block decomposition}. We will often write $|P_j|$ to denote the number of vertices belonging to the parity block $P_j$.
\begin{example}\label{esempiocusano} \rm
Consider the graph $K_5$ endowed with the bi-labelling in the
figure below. Its parity block decomposition consists of two
parity blocks  $P_1$ and $P_2$, with
$$
P_1 = P(0,e) = P(0,o) = P(1,o) = P(2,o) = P(3,e)=P(4,e)
$$
$$
P_2 = P(1,e) = P(2,e) = P(3,o) = P(4,o),
$$
and so $K_5 = P_1\cup P_2$. Observe that the vertex $0$ is odden, the
vertices $1,2$ are odd, and the vertices $3,4$ are even in $P_1$; the vertices $1$ and $2$ are
even and the vertices $3$ and $4$ are odd in $P_2$.
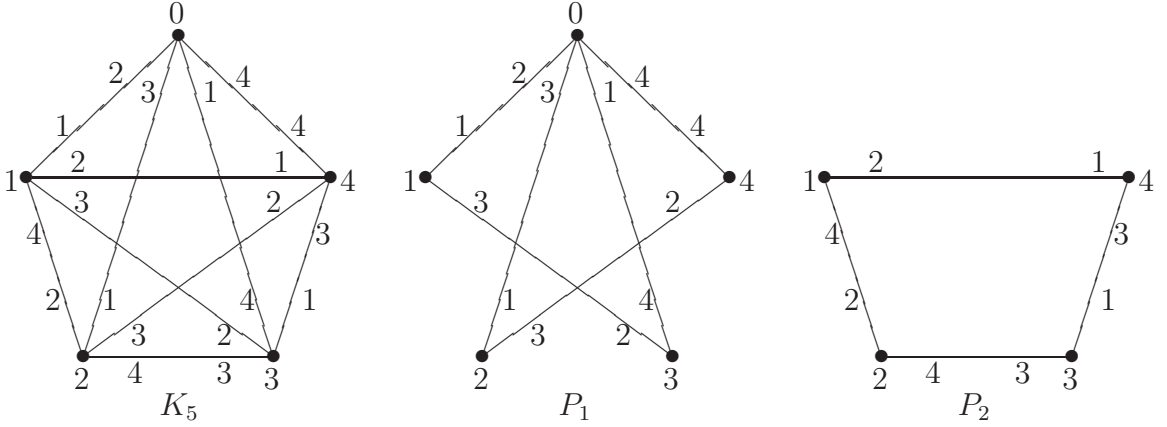
\begin{center}\unitlength=0,25mm
\begin{picture}(750,190)
\letvertex A=(120,185)\letvertex B=(40,110)\letvertex C=(70,15)
\letvertex D=(170,15)\letvertex E=(200,110)

\put(110,-16){$K_5$}\put(320,-16){$P_1$}\put(530,-16){$P_2$}

\drawundirectededge(A,B){}\drawundirectededge(B,C){}\drawundirectededge(C,D){}\drawundirectededge(D,E){}\drawundirectededge(A,E){}
\drawundirectededge(A,C){}\drawundirectededge(E,C){}\drawundirectededge(E,B){}\drawundirectededge(D,B){}\drawundirectededge(A,D){}

\drawvertex(A){$\bullet$}\drawvertex(B){$\bullet$}
\drawvertex(C){$\bullet$}\drawvertex(D){$\bullet$}\drawvertex(E){$\bullet$}

\put(115,192){$0$} \put(28,103){$1$} \put(65,-2){$2$}
\put(165,-2){$3$}\put(205,103){$4$}
%%%%%%%%%%%%%%%%%%%%%%%%%%%%%%%%%%%%%%%%%%%%%%%%%%%%%%%%%%%%%%
\put(83,160){$2$} \put(100,150){$3$} \put(133,150){$1$}
\put(150,158){$4$}

\put(55,132){$1$}\put(63,113){$2$}\put(65,92){$3$}
\put(40,75){$4$}

\put(50,40){$2$}\put(80,40){$1$}\put(95,21){$3$}\put(93,1){$4$}

\put(140,1){$3$}\put(140,21){$2$}\put(152,40){$4$}\put(185,40){$1$}

\put(192,75){$3$}\put(166,92){$2$}\put(170,113){$1$}\put(179,132){$4$}

%%%%%%%%%%%%%%%%%%%%%%%%%%%%%%%%%%%%%%%%%%%%%%%%%%%%%%%%%%%%%%%%%%%%%%%%%%%%%%%%%%%%%%%%%%%%%%%%%%%%%%%%%%PRIMO SUBGRAPH

\letvertex AA=(330,185)\letvertex BB=(250,110)\letvertex CC=(280,15)
\letvertex DD=(380,15)\letvertex EE=(410,110)

\drawundirectededge(AA,BB){}\drawundirectededge(AA,EE){}
\drawundirectededge(AA,CC){}\drawundirectededge(EE,CC){}\drawundirectededge(DD,BB){}\drawundirectededge(AA,DD){}

\drawvertex(AA){$\bullet$}\drawvertex(BB){$\bullet$}
\drawvertex(CC){$\bullet$}\drawvertex(DD){$\bullet$}\drawvertex(EE){$\bullet$}

\put(325,192){$0$} \put(238,103){$1$} \put(275,-2){$2$}
\put(375,-2){$3$}\put(415,103){$4$}

\put(295,160){$2$} \put(310,150){$3$} \put(343,150){$1$}
\put(360,158){$4$}

\put(265,132){$1$}\put(275,92){$3$}\put(290,40){$1$}\put(305,21){$3$}\put(350,21){$2$}\put(362,40){$4$}\put(376,92){$2$}\put(389,132){$4$}

\letvertex BBB=(460,110)\letvertex CCC=(490,15)
\letvertex DDD=(590,15)\letvertex EEE=(620,110)

\drawundirectededge(BBB,CCC){}\drawundirectededge(CCC,DDD){}\drawundirectededge(DDD,EEE){}\drawundirectededge(BBB,EEE){}

\drawvertex(BBB){$\bullet$}\drawvertex(CCC){$\bullet$}\drawvertex(DDD){$\bullet$}\drawvertex(EEE){$\bullet$}

\put(448,103){$1$}\put(485,-2){$2$}
\put(585,-2){$3$}\put(625,103){$4$}

%%%%%%%%%%%%%%%%%%%%%%%%%%%%%%%%%%%%%%%%%%%%%%%%%%%%%%%%%%%%%%%%%%%%%%%%%%%%%%%%%%%%%%%%%%%%%%%%%%%%%%%%%%%%%%%%%%%%%%%%%%%%%%%%%%%%%%%%%%%%%%%
\put(483,113){$2$}\put(460,75){$4$}

\put(470,40){$2$}\put(513,1){$4$}

\put(560,1){$3$}\put(605,40){$1$}

\put(612,75){$3$}\put(600,113){$1$}
\end{picture}
\end{center}
\vspace{0.5cm}
\end{example}\vspace{0.5cm}
In what follows, we will use the convention that the sum of two
elements $i,j\in [d]$ is given by
\begin{eqnarray}\label{somma}
i+j= \left\{
                                             \begin{array}{ll}
                                               i+j, & \hbox{if $i+j\leq d$;} \\
                                               i+j-d, & \hbox{if $i+j > d$.}
                                             \end{array}
                                           \right.
\end{eqnarray}
The interest in the parity block decomposition is justified by the
following result.
\begin{thm}\label{decomposition}
Let $G=(V,E)$ be a regular graph of even degree $d\geq 3$, with
parity block decomposition $G=\cup_{j\in J} P_j$. Let $C_d$ be
the cycle graph of length $d$, with vertex set $[d]$. Then there
is a one-to-one correspondence between the set of parity blocks
$\{P_j\}_{j\in J}$ and the set of connected components
$\{S_j\}_{j\in J}$ of the zig-zag product $G\zig C_d$.
\end{thm}
\begin{proof}
Let $P$ be a parity block in the parity block decomposition of $G$
and let $v\in P$ with parity $i$, where $i\in \{e,o\} $. Notice that
$P$ contains all edges issuing from $v$ and labelled by $j\in
[d_i]$ near $v$. Given $j\in [d_i]$, let $w\in P$ with $\rot_G(v,j)=(w,h)$. Lemma
\ref{papillonlemma} implies that the vertices $(v,j-1),(v,j+1)$
and $(w,h-1), (w,h+1)$ form a papillon subgraph in $G\zig C_d$, so
that these vertices belong to the same connected component of
$G\zig C_d$.

Moreover, one has that all vertices of the type $(v,k)$, with $k\in
[d]\setminus [d_i]$, belong to the same connected component in
$G\zig C_d$. To see that, it is enough to observe that $(v,k)$
belongs exactly to two papillon subgraphs, which connect $(v,k)$
to $(v,k-2)$ and to $(v,k+2)$: by iteration, this implies that there
is a path connecting all $(v,k)$'s with $k\in [d]\setminus [d_i]$.
Hence, we have shown that, if there exists in $P$ a path from $v$
(with parity $i$) to $w$ (with parity $j$ not necessarily different from $i$), then the vertices
$(v,i')$, for every $i' \in [d]\setminus [d_i]$, and $(w,j')$,
for every $j'\in [d]\setminus [d_j]$, belong to the same connected
component of $G\zig C_d$. In order to complete the proof, we show
that, if $v$ has parity $i$ in $P_i$ and $w$ has parity $j$ in
$P_j$, and $P_i\neq P_j$, then the vertices of type $(v,i')$,
$i'\in [d]\setminus [d_i]$ and $(w,j')$, $j'\in [d]\setminus
[d_j]$ belong to distinct connected components $S_i$ and $S_j$ of
$G\zig C_d$. If this is not the case, so that $(v,i')$ and
$(w,j')$ are in the same connected component, then there is a path
$\mathcal{P}=\{(v,i')=(v_0,i_0'), (v_1,i_1'),\ldots,
(w,j')=(v_n,i_n')\}$ in $G\zig C_d$. The vertices $(v_k, i_k')$
and $(v_{k+1}, i_{k+1}')$ are connected if and only if they belong to
a papillon graph, corresponding to the edge $\{v_k,v_{k+1}\}$ in
$G$, such that $\rot_G(v_k, i_k)=(v_{k+1}, j_{k})$ with $i_k\in
\{i_k'\pm 1\}$ and $j_k\in \{i_{k+1}'\pm 1\}$. We want to show
that the path $\mathcal{P}$ in $G\zig C_d$ can be projected onto
$G$, giving rise to a path connecting $v$ and $w$ satisfying the parity
properties: this will give a contradiction. In other words, we
claim that for each $k=0,1,\ldots, n$, the vertex $v_k$ has parity
$t_k$, with $i'_k\in [d]\setminus [d_{t_k}]$, in the parity block
$P_i$. This follows by observing that any $(v_k,i'_k)\in
\mathcal{P}$ exactly belongs to two papillon subgraphs, containing
the vertices $(v_k,i'_k),(v_k,i'_k-2)$ and
$(v_k,i'_k),(v_k,i'_k+2)$, respectively. The corresponding edges
$\{v_{k-1}, v_k\}$ and $\{v_k,v_{k+1}\}$ in $G$ satisfy
$\rot_G(v_{k-1},i_{k-1}) = (v_k, j_{k-1})$ and
$\rot_G(v_{k},i_{k}) = (v_{k+1}, j_{k})$. Since $i_k, j_{k-1}\in
\{i_k'\pm 1\}$, we have $j_{k-1}\equiv i_k \pmod{2}$ and the proof
is completed.
\end{proof}

\begin{remark}\rm
We have shown in the proof of Theorem \ref{decomposition} that, if a vertex $v$ in a parity block $P$ is
even (resp. odd), so that the elements in $[d_e]$ (resp. $[d_o]$)
label the edges issuing from that vertex in $P$, then in the connected
component of $G\zig C_d$ associated with that parity block we find all
vertices $(v, j)$, with $j\in [d_o]$ (resp. $[d_e]$). Analogously,
if a vertex $v$ in a parity block $P$ is odden, then in the
connected component of $G\zig C_d$ associated with that parity block we
find the vertices $(v, j)$, with $j$ belonging to both the sets
$[d_o]$ and $[d_e]$.
\end{remark}

\subsection{Isomorphism classification via pseudo-replacement graphs}\label{subsectioniso}
Once we have described a method for distinguishing the connected
components of the zig-zag product, it is straightforward to investigate the isomorphism classes problem. As before, $G$ is a
regular graph of even degree: in order to simplify the exposition, we put the degree of $G$ equal to $2d$, where $d$ is a positive integer.

We introduce now new graphs, called the
\textit{pseudo-replacement} graphs, which are in a one-to-one
correspondence with the parity blocks of a parity block
decomposition of $G$ (and so with the connected components of
$G\zig C_{2d}$ by virtue of Theorem \ref{decomposition}). We will see that the isomorphism problem for the
connected components of $G\zig C_{2d}$ is equivalent to the analogous
problem for such graphs.

Let $P$ be a parity block in the block decomposition of $G$. The
\textit{pseudo-replacement graph} $R$ associated with $P$ in $G\zig
C_{2d}$ is the graph obtained as described below. We distinguish two
cases.
\begin{enumerate}
\item Suppose that each vertex $v_i\in P$ is either even or odd, so that the degree of $v_i$ in $P$ is $d$.
Let $C_d^e$ (resp. $C^o_d$) be the cycle graph of length $d$ with
vertex set $[(2d)_e]$ (resp. $[(2d)_o]$). Then the pseudo-replacement graph associated with the parity block $P$ is the graph
$R$ in which every even (resp. odd) vertex $v_i$ is replaced by a copy of
the graph $C_d^e$ (resp. $C_d^o$): this implies that $R$ has
$d|P|$ vertices. The vertex $j$ of $C^e_d$ (resp. $C^o_d$) belonging to the copy associated with $v_i$ will be denoted by $(v_i,j)$; therefore, the vertex $(v_i,j)$ is joint to $(v_k,h)$ in $R$ if either
$\rot_G(v_i, j)=(v_k, h)$ or $i=k$ and $h=j\pm 2$.
\item If $v$ is odden, one associates with $v$ two vertices $v_e$, $v_o$ and two disjoint copies $C_d^e$ and $C_d^o$,
and for each of them we proceed as in the case (1).
\end{enumerate}
\begin{example}\label{examplefigure}   \rm
The picture below represents the pseudo-replacement graph $R_1$
corresponding to the parity block $P_1$ of Example
\ref{esempiocusano}. Recall that the vertex $0$ is odden; the
vertices $1,2$ are odd; the vertices $3,4$ are even.
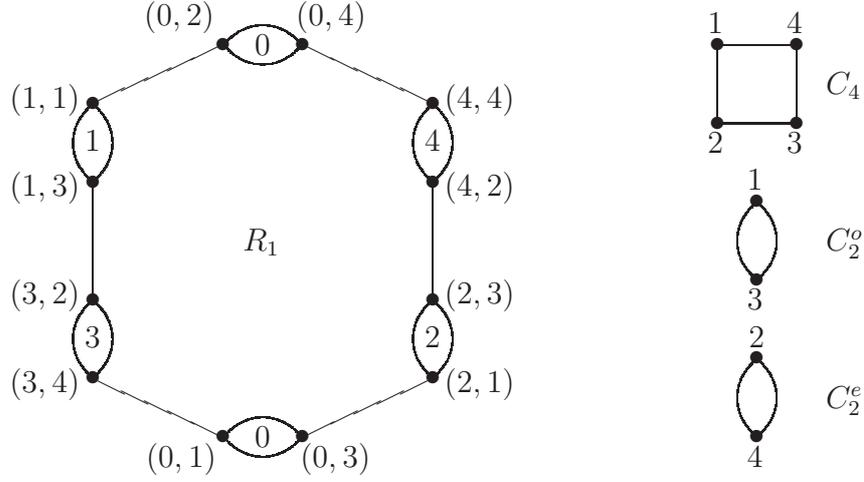
\begin{center}\unitlength=0,26mm
\begin{picture}(480,240)
\letvertex A=(64,190)\letvertex B=(64,150)\letvertex C=(64,90)
\letvertex D=(64,50)\letvertex E=(130,20)
\letvertex F=(170,20)\letvertex G=(236,50)\letvertex H=(236,90)
\letvertex I=(236,150)\letvertex L=(236,190)\letvertex M=(170,220)\letvertex N=(130,220)

\letvertex e=(380,220)
\letvertex f=(380,180)\letvertex g=(420,180)\letvertex h=(420,220)
\letvertex i=(400,140)\letvertex l=(400,100)\letvertex m=(400,60)\letvertex n=(400,20)

\drawundirectededge(B,C){}\drawundirectededge(D,E){}\drawundirectededge(F,G){}\drawundirectededge(H,I){}\drawundirectededge(L,M){}
\drawundirectededge(N,A){}

\drawvertex(A){$\bullet$}\drawvertex(B){$\bullet$}
\drawvertex(C){$\bullet$}\drawvertex(D){$\bullet$}\drawvertex(E){$\bullet$}
\drawvertex(F){$\bullet$}\drawvertex(G){$\bullet$}
\drawvertex(H){$\bullet$}\drawvertex(I){$\bullet$}\drawvertex(L){$\bullet$}\drawvertex(M){$\bullet$}\drawvertex(N){$\bullet$}

\drawvertex(e){$\bullet$}\drawvertex(f){$\bullet$}
\drawvertex(g){$\bullet$}\drawvertex(h){$\bullet$}\drawvertex(i){$\bullet$}
\drawvertex(l){$\bullet$}\drawvertex(m){$\bullet$}
\drawvertex(n){$\bullet$}

\drawundirectededge(e,f){}\drawundirectededge(f,g){}\drawundirectededge(g,h){}\drawundirectededge(h,e){}

\drawundirectedcurvededge(i,l){}\drawundirectedcurvededge(l,i){}\drawundirectedcurvededge(m,n){}\drawundirectedcurvededge(n,m){}

\drawundirectedcurvededge(A,B){}\drawundirectedcurvededge(B,A){}\drawundirectedcurvededge(C,D){}\drawundirectedcurvededge(D,C){}
\drawundirectedcurvededge(E,F){}\drawundirectedcurvededge(F,E){}\drawundirectedcurvededge(G,H){}\drawundirectedcurvededge(H,G){}
\drawundirectedcurvededge(I,L){}\drawundirectedcurvededge(L,I){}\drawundirectedcurvededge(M,N){}\drawundirectedcurvededge(N,M){}

\put(375,226){$1$}\put(375,165){$2$}

\put(415,165){$3$} \put(415,226){$4$}\put(395,146){$1$}

\put(395,85){$3$} \put(396,66){$2$}\put(395,5){$4$}

  \put(435,195){$C_4$} \put(435,115){$C_2^o$}\put(435,35){$C_2^e$}

\put(146,215){$0$} \put(90,230){$(0,2)$}\put(167,230){$(0,4)$}

\put(60,166){$1$} \put(22,143){$(1,3)$}\put(22,188){$(1,1)$}

\put(60,66){$3$} \put(22,43){$(3,4)$}\put(22,88){$(3,2)$}

\put(146,15){$0$} \put(91,5){$(0,1)$}\put(169,5){$(0,3)$}

\put(232,66){$2$} \put(242,43){$(2,1)$}\put(242,88){$(2,3)$}

\put(232,166){$4$} \put(242,143){$(4,2)$}\put(242,188){$(4,4)$}

\put(140,115){$R_1$}

\end{picture}
\end{center}
\end{example}
\begin{remark} \rm
The name pseudo-replacement is justified by the fact that, if $P$
contains no odden vertex (so that $P$ is a $d$-regular graph) and $C'$ is a graph
isomorphic to $C_d^e$ (or $C_d^o$), then $R$ is isomorphic to $P\rr
C'$. Notice that the number of vertices of $R$ is
$$
2d\cdot |\{\text{odden vertices in }P\}| + d\cdot|\{\text{non odden vertices in } P\}|.
$$
\end{remark}

The next result shows that it is enough to consider the graphs
$R$ in order to study the isomorphism classes of the connected components $S$
of the zig-zag product.
\begin{thm}\label{iso}
Let $G=(V,E)$ be a $2d$-regular graph and $C_{2d}$ be the cycle graph of
length $2d$. Let $G=\cup_{k\in J} P_k$ be the parity block
decomposition of $G$. For each $k\in J$, let $S_k$ be the
connected component of $G\zig C_{2d}$ associated with $P_k$, and let $R_k$
be the pseudo-replacement graph associated with $P_k$. Then $S_k\simeq
S_{k'}$ if and only if $R_k\simeq R_{k'}$.
\end{thm}
\begin{proof}
Let $\phi: S_k \rightarrow S_{k'}$ be an isomorphism. By Lemma
\ref{papillonlemma}, $S_k$ is the union of a finite number of papillon
graphs. In particular, let $(v,i), (v,i+2)$ and $(w,j),(w,j+2)$ be
joint in a papillon graph, corresponding to an edge connecting $v$
and $w$ in $P_k$, such that $\rot_G(v,i+1)=(w,j+1)$. The image of
such a papillon graph under $\phi$ will consist of the four vertices $(v',i'), (v',i'+2)$ and
$(w',j'),(w',j'+2)$. Let $\psi : R_k \rightarrow R_{k'}$ such that
$\psi(v, i+1)=(v',i'+1)$. Let us show that the map $\psi$ is indeed a
bijection. The injectivity follows from the fact that a pair of
vertices $(v,i),(v,i+2)$ in a papillon graph uniquely determines
the vertex $(v,i+1)$ in $R_k$. The map $\psi$ is surjective by the
hypothesis on $\phi$. Hence, it is enough to prove that $\psi$
preserves adjacency. We have that $(v,i+1)$ and $(w,j+1)$ are
joint by an edge in $R_k$ if and only if the vertices $v,w$ in $P_k$
satisfy $\rot_G(v,i+1)=(w,j+1)$. By Lemma
\ref{papillonlemma}, this is equivalent to saying that there is a
papillon graph in $S_k$ containing $(v,i), (v,i+2)$ and $(w,j),(w,j+2)$.
Since $\phi$ is an isomorphism, this is true if and only if there
is a papillon graph consisting of vertices $(v',i'),
(v',i'+2)$ and $(w',j'),(w',j'+2)$ in $S_{k'}$. This fact is equivalent to saying
that in $R_{k'}$ there is an edge connecting $(v', i'+1)$ and
$(w', j'+1)$. We have proven in this way that $\psi$ preserves the
adjacency.

Conversely, suppose that there exists an isomorphism $\psi :
R_k\rightarrow R_{k'}$ and let $v\in P_k$, $i\in [2d]$ such that
$\psi(v,i+1)=(v',i'+1)$, for some $v'\in P_{k'}$, $i'\in [2d]$. Notice
that the vertices $(v,i+1)$ and $(v',i'+1)$ univocally correspond
to the sets of vertices $(v,i),(v,i+2)$ in $S_k$ and
$(v',i'),(v',i'+2)$ in $S_{k'}$, and each of these pairs of vertices
univocally determine two papillon graphs in $S_k$ and $S_{k'}$,
respectively. Since $\psi$ is an isomorphism, one has $\psi(v,i+3)\sim \psi(v,i+1)$ and
$\psi(v,i-1)\sim \psi(v,i+1)$, as $(v,i+3)\sim (v,i+1)$ and
$(v,i-1)\sim (v,i+1)$ in $R_k$. We define a map $\phi: S_k\rightarrow S_{k'}$ such
that
$$
\phi(v,i) = \left\{
                                  \begin{array}{ll}
                                    (v',i') & \hbox{if } \psi(v,i+3)=(v',i'+3)\\
                                    (v',i'+2) & \hbox{if } \psi(v,i+3) = (v',i'-1).
                                  \end{array}
                                \right.
 $$
As before, the bijectivity of $\psi$
ensures the bijectivity of $\phi$. Let us show now that $\phi$ preserves the adjacency relation. We have that $(v,i)$
and $(w,j)$ are adjacent in $S_k$ if and only if there
exists in $S_k$ a papillon graph containing $(v,i), (v,i\ast 2)$ and
$(w,j),(w,j\star 2)$ for some $\ast, \star \in \{+,-\}$. Suppose, without
loss of generality, that this papillon graph consists of vertices $(v,i),
(v,i+2)$ and $(w,j),(w,j+2)$. This is equivalent to stating that there
is an edge joining $(v,i+1)$ and $(w,j+1)$ in $R_k$; since $\psi$
is an isomorphism, this implies that there is an edge joining
$(v',i'+1)$ and $(w',j'+1)$ in $R_{k'}$. Therefore, in the corresponding
papillon graph of the connected component $S_{k'}$ associated with $R_{k'}$,
there exists an edge joining $(v',i')$ and $(w',j')$. This gives the
assertion.
\end{proof}

We have proven in Theorem \ref{iso} that the isomorphism classes of the connected components
$\{S_i\}$ of $G\zig C_{2d}$ are characterized by the isomorphism classes of the
replacement graphs $\{R_i\}$. On the other hand, it is not true, in general, that the isomorphism classes of the connected
components $\{S_i\}$ of the zig-zag product $G\zig C_{2d}$ are
characterized by the isomorphism classes of the parity blocks $\{P_i\}$, regarded as non bi-labelled graphs, in the
parity block decomposition of $G$.
\begin{example}\label{nonisomorfi}\rm
The following parity blocks in $K_9$ are isomorphic as
non-labelled graphs; however, one can show that the
associated pseudo-replacement graphs are not isomorphic.

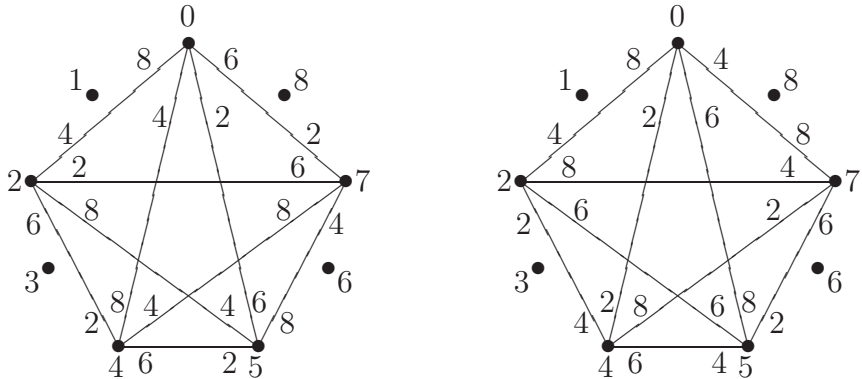
\begin{center}
\begin{picture}(625,165)\unitlength=0.23mm
\letvertex A=(170,180)\letvertex B=(115,150)\letvertex C=(80,100)
\letvertex D=(90,50)\letvertex E=(130,5)\letvertex F=(210,5)\letvertex G=(250,50)
\letvertex H=(260,100)\letvertex I=(225,150)
\put(165,190){$0$} \put(101,153){$1$} \put(66,95){$2$}
\put(76,37){$3$}\put(124,-12){$4$}   \put(265,95){$7$}
\put(230,153){$8$} \put(204,-12){$5$}\put(255,37){$6$}

\put(140,165){$8$} \put(149,132){$4$} \put(185,132){$2$} \put(190,165){$6$}

\put(95,122){$4$} \put(103,103){$2$} \put(110,78){$8$}\put(77,70){$6$}

\put(110,13){$2$}\put(125,25){$8$} \put(144,23){$4$} \put(141,-10){$6$}

\put(189,-10){$2$}\put(188,23){$4$}  \put(206,25){$6$} \put(222,13){$8$}

\put(250,70){$4$}\put(220,78){$8$}\put(228,103){$6$} \put(237,122){$2$}

\drawundirectededge(A,C){}\drawundirectededge(A,E){}\drawundirectededge(A,F){}\drawundirectededge(A,H){}\drawundirectededge(C,E){}
\drawundirectededge(C,F){}\drawundirectededge(C,H){}\drawundirectededge(E,F){}

\drawundirectededge(E,H){}\drawundirectededge(F,H){} \drawvertex(A){$\bullet$}\drawvertex(B){$\bullet$}
\drawvertex(C){$\bullet$}\drawvertex(D){$\bullet$}\drawvertex(E){$\bullet$}\drawvertex(F){$\bullet$}
\drawvertex(G){$\bullet$}\drawvertex(H){$\bullet$}\drawvertex(I){$\bullet$}
%%%%%%%%%%%%%%%%%%%%%%%%%%%%%%%%%%%%%%%%%%%%%%%%%%%%%%%%%%%%%%%%%%%%%%%%%%%%%%%%%%%%%%%%%%%%%%%%%%%%%%%%%%%%%%%%%%%%%%%%%%%%%%%%%%%%%%%%%%%%%%%%
%%%%%%%%%%%%%%%%%%%%%%%%%%%%%%%%%%%%%%%%%%%%%%%%%%%%%%%%%%%%%%%%%%%%%%%%%%%%%%%%%%%%%%%%%%%%%%%%%%%%%%%%%%%%%%%%%%%%%%%%%%%%%%%%%%%%%%%%%%%%%%%%%

\letvertex a=(450,180)\letvertex b=(395,150)\letvertex c=(360,100)
\letvertex d=(370,50)\letvertex e=(410,5)\letvertex f=(490,5)\letvertex g=(530,50)
\letvertex h=(540,100)\letvertex i=(505,150)

\put(445,190){$0$} \put(381,153){$1$} \put(346,95){$2$}
\put(356,37){$3$}\put(404,-12){$4$}   \put(545,95){$7$}
\put(510,153){$8$} \put(484,-12){$5$}\put(535,37){$6$}

\drawundirectededge(a,c){}\drawundirectededge(a,e){}\drawundirectededge(a,f){}\drawundirectededge(a,h){}\drawundirectededge(c,e){}
\drawundirectededge(c,f){}\drawundirectededge(c,h){}\drawundirectededge(e,f){}      \drawundirectededge(e,h){}\drawundirectededge(f,h){}

\put(420,165){$8$} \put(429,132){$2$} \put(465,132){$6$}
\put(470,165){$4$}

\put(375,122){$4$} \put(383,103){$8$}
\put(390,78){$6$}\put(357,70){$2$}

\put(390,13){$4$}\put(405,25){$2$} \put(424,23){$8$}
\put(421,-10){$6$}

\put(469,-10){$4$}\put(468,23){$6$}  \put(486,25){$8$}
\put(502,13){$2$}

\put(530,70){$6$}\put(500,78){$2$}\put(508,103){$4$}
\put(517,122){$8$}

\drawvertex(a){$\bullet$}\drawvertex(b){$\bullet$}
\drawvertex(c){$\bullet$}\drawvertex(d){$\bullet$}\drawvertex(e){$\bullet$}\drawvertex(f){$\bullet$}
\drawvertex(g){$\bullet$}\drawvertex(h){$\bullet$}\drawvertex(i){$\bullet$}
\end{picture}
\end{center}

\end{example}
              \vspace{0.5cm}
On the other hand, in the particular case $d=4$, one has that two
pseudo-replacement graphs are isomorphic if and only if they are
associated with two parity blocks of the same size, as the following corollary shows.

\begin{cor}\label{corollario4}
Let $G$ be a $4$-regular graph, with parity block decomposition $G=\cup_{j\in J} P_j$, and let $\{S_j\}_{j\in J}$ be
the set of connected components of $G\zig C_4$. If $|P_i|=|P_j|$
then $S_i\simeq S_j$.
\end{cor}
\begin{proof}
Consider the vertex $v$ in $P(v,i)$, with $i\in \{e,o\}$. Since the
degree of $G$ is $4$, the set $[4_i]$ contains only two indices, let they be $j_1, j_2$
(actually, one has $[4_e]=\{2,4\}$ and $[4_o]=\{1,3\}$) satisfying
$\rot_G(v,j_1)=(w,h)$ and $\rot_G(v,j_2)=(u,k)$, for some $w,u\in P(v,i)$ and $h,k\in [4]$. Therefore, the
path $\{u,v,w\}$ is contained in the parity block $P(v,i)$. The
same argument applies to any other vertex in the same block, and we can construct in this way an
ordered sequence of vertices $v_1, v_2, \ldots$ of $P(v,i)$. Since $G$ is finite, there exists $n$ such
that $v_{1+i}=v_{n+i}$ for all $i\geq 0$. In fact, if this is not the case, there exists $j>1$ such that
$v_{j+i}=v_{n+i}$ for each $i\geq 0$, which is absurd since $v_j$ would have degree $3$, but the vertices of $P(v,i)$ must have degree $2$ or $4$. Observe that,
if $v_k$ is an odden vertex, then it occurs twice in the sequence. On the other hand, if $v_k$
is either odd or even, then it appears only once.
Hence, any parity block $P$ is determined by a
sequence of vertices $\{v=v_1, \ldots, v_n = v\}$ in $G$. In
particular, $R$ is constituted by an alternate sequence of simple
edges and cycles of length $2$ (as in the figure of Example
\ref{examplefigure}), whose number is equal to twice the number of odden vertices plus
the number of non odden vertices. As a consequence, the structure of $R$ depends only on the size of $P$, and this concludes the proof.
 \end{proof}

Note that the sequence of vertices $\{v_1=v,v_2, \ldots, v_n=v\}$ constitutes an
Eulerian circuit in the parity block $P(v,i)$, satisfying the further property that
$\rot_G(v_i,h_i)=(v_{i+1},k_i)$, with $k_i\equiv h_{i+1} \pmod{2}$, for every $i=1,\ldots, n-2$, and
with $k_{n-1}\equiv h_1 \pmod{2}$. Moreover, a vertex which is either odd or even is visited
once by the cycle, whereas an odden vertex is visited twice. We will call such a circuit a \textit{spanning path} of the parity block $P(v,i)$.
\begin{remark}\rm
It is known that every connected $2d$-regular graph $G=(V,E)$ admits an Eulerian circuit;
this ensures that, given a connected $2d$-regular graph $G$, there exists at least a bi-labelling of
the edges of $G$ such that the graph $G\zig C_{2d}$ is connected. We will see an explicit application
of this property in Proposition \ref{goodlabel}, where $G$ is the complete graph.
\end{remark}

\subsection{Explicit description of isomorphisms of pseudo-replacement
graphs}\label{subsectionexplicit} In this section we describe
the isomorphisms between pseudo-replacement graphs associated with two isomorphic
(as non-labelled graphs) parity blocks. We have seen in Example \ref{nonisomorfi} that
two parity blocks which are isomorphic as non-labelled graphs can be associated
with two non isomorphic pseudo-replacement graphs (and so with non isomorphic connected components).
Therefore, in what follows, given a $2d$-regular graph $G$, and considering the zig-zag product $G\zig C_{2d}$,
we focus on a parity block $P$, and the natural
question arising in this context is the following. \textit{Given two distinct
bi-labellings $\mathcal{L}, \mathcal{L}'$ of $P$ (and the
relative pseudo-replacement graphs $R$, $R'$ associated with $P$),
under which conditions on $\mathcal{L}$ and $\mathcal{L}'$ are the corresponding connected components $S$ and
$S'$ of $G\zig C_{2d}$ isomorphic?}

In order to answer this question we need to fix some notations. In
what follows, $D_d=\langle a,b | \ a^d=b^2=1\rangle$ denotes the dihedral group with $2d$ elements.
Notice that $D_d$ can be identified with the automorphism group of
the cycle graphs $C_d^e$ and $C_d^o$, with vertices given by
$[(2d)_e]$ or $[(2d)_o]$, respectively. Observe that there exists
a natural bijection $\theta: [(2d)_o] \rightarrow [(2d)_e]$
defined by $\theta(i) =i+1$, for each $i\in [(2d)_o]$. Given an
automorphism $\phi\in D_d \simeq Aut(C_d^e) \simeq Aut(C_d^o)$, we
put
$$
\overline{\phi}(i) =
  \begin{cases}
\phi\big(\theta(i)\big),&\text{if}\ i\in C_d^o\\
\phi\big(\theta^{-1}(i)\big),&\text{if}\ i\in C_d^e.
  \end{cases}
$$
We recall that if $w$ is an odden vertex in a parity block
$P$, so that its degree is $2d$ in $P$, then in the corresponding
pseudo-replacement $R$, $w$ is replaced by two subgraphs isomorphic
to $C_d^o$ and $C_d^e$, respectively. The vertices of $C_d^o$ and $C_d^e$
%is natural to think of $u$ as a pair $(u^e,u^o)$ and the vertices
%corresponding to such subgraphs in $R$
are denoted by $(w,i)$ and $(w,j)$, with $i\in [(2d)_o]$ and
$j\in [(2d)_e]$, respectively.

Let $V$ and $W$ be the subsets of vertices of $P$ having degree
$d$ and $2d$, respectively. Any automorphism $f$ of
$P$ bijectively maps $V$ into $V$, and $W$ into $W$.
For each vertex $w\in W$, let us introduce a permutation
$\varepsilon_w\in Sym (\{e,o\})$ that will enable us to take into
account the possibility of switching the subgraphs $C_d^e$ and
$C_d^o$ associated with an odden vertex.
%When we write $(u,i)$ and $u$ is odden, we mean that $u=u^i$ if $i\in [2d_i]$.
%By using this convention, in what follows, if $f(u)=u'$ is an
%automorphism of $P$, we assume that this corresponds in $R$ to a
%(depending on $f$) bijection between $\{f(u^e),f(u^o)\}$ and
%$\{u'^{e},u'^{o}\}$.
% We distinguish in this case denoting the two automorphisms $g^e$ and $g^o$, respectively.
Now for any $v\in V$ with parity $i$, let $g_v\in Aut(C_d^i)$, and
for any $w\in W$ let $g_w^e\in Aut(C_d^e)$ and $g_w^o\in
Aut(C_d^o)$. Finally, for each $w\in  W$, we put:
$$
\widetilde{g}_w^{\varepsilon_w(i)}(h)=
  \begin{cases}
    g_w^i(h), & \text{if}\ h\ \text{has parity}\ i\text{ and}\ \varepsilon_w=id, \\
& \\
\overline{g_w^{\varepsilon_w(i)}}(h), & \text{if}\ h\ \text{has
parity}\ i\text{ and}\ \varepsilon_w\neq id.
  \end{cases}
$$
\begin{thm}\label{isoclass}
Let $G$ be a $2d$-regular graph. With the above notations, let us
define the map
%Let $f\in Aut(P)$. Let $V$ and $W$
%be the subsets of vertices of $P$ having degree $n$ and $2n$,
%respectively. For any $v\in V$ with parity $i$, let $g_v\in
%Aut(C_n^i)$, and for any $w\in W$ let $g_w^e\in Aut(C_n^e)$ and
%$g_w^o\in Aut(C_n^o)$, and let $\varepsilon_w\in Sym(\{e,o\})$.
$F:R\rightarrow R'$ as
$$
F(u, k)=\left\{
  \begin{array}{ll}
    \big(f(u), g_u(k)\big), & \hbox{if $u\in V$, $u$ and $f(u)$ have parity $i$;}\\
    \big(f(u), \overline{g_u}(k)\big), & \hbox{if $u\in V$, $u$ and $f(u)$ have different parities}\\
    %(f(u), g_u^i(k)), & \hbox{if $u\in W$, $k\in [d_i]$ and $f(u)$ is with parity $i$;} \\
    %(f(u), \theta g_u^i(k)), & \hbox{if $u\in W$, $k\in [d_i]$ and $f(u)$ is not with parity $i$}
  \end{array}
\right.
$$
and
$$
F(u, k)= \Big(f(u), \widetilde{g}_u^{\varepsilon_u(i)}(k)\Big), \qquad
\qquad \text{if } u\in W \ \text{and }k\in [(2d)_i].
$$
%\left\{
 % \begin{array}{ll}
 %   (f(u), g_u^i(k)), & \hbox{if $u\in W$, $k\in [(2n)_i]$ and $\varepsilon_w =id$;}\\
 %   (f(u), \overline{g}_u^{\varepsilon_u(i)}(k)), & \hbox{if $u\in W$, $k\in [(2n)_i]$ and $\varepsilon_u\neq id$.}\\
                        %     & \hbox{$f(u)$ is with parity different from $i$;}\\
    %(f(u), g_u^i(k)), & \hbox{if $u\in W$, $k\in [d_i]$ and $f(u)$ is with parity $i$;} \\
    %(f(u), \theta g_u^i(k)), & \hbox{if $u\in W$, $k\in [d_i]$ and $f(u)$ is not with parity $i$.}
%  \end{array}
%\right.
If the conditions
\begin{enumerate}
\item \quad $\rot_G(u,h)=(v,k) \ \ \Longrightarrow \ \ \rot_G(f(u),g_u(h))=\big(f(v),g_v(k)\big), \quad \text{ if} \ \ u,v\in V$;
\item \quad $\rot_G(u,h) = (v,k) \ \ \Longrightarrow \ \ \rot_G(f(u),
\widetilde{g}_u^{\varepsilon_u(i)}(h))= \big(f(v), g_v(k)\big)\\ \quad \qquad \text{ \quad if } \ u\in W, v \in V, h\in [(2d)_i]$;
\item \quad $\rot_G(u,h) = (v,k) \ \ \Longrightarrow \ \ \rot_G(f(u),
\widetilde{g}_u^{\varepsilon_u(i)}(h))= \big(f(v),
\widetilde{g}_v^{\varepsilon_v(j)}(k)\big)\\ \quad \qquad \text{ \quad if }\
u,v\in W, h\in [(2d)_i],k\in [(2d)_j]$
\end{enumerate}
hold, then $F$ is an isomorphism between $R$ and $R'$.
\end{thm}
\begin{proof}
It is clear that, once fixed $f\in Aut(P)$, together with the
automorphisms $g_u^i$ for $u\in W$, and the automorphisms
$g_u$ for $u\in V$, the map $F$ is a bijection between the set
of vertices of $R$ and $R'$.

Let us prove that $F$ preserves the adjacency relations. Let $u\in
V$ and suppose, without loss of generality, that $u$ is even. In
particular, since $f$ is an automorphism of $P$, one has $f(u)\in
V$. We can suppose that also $f(u)$ is even.
Take a pair of adjacent vertices in $R$. We have two
possibilities. The first one is that the vertices have the form
$(u,h)$ and $(u,h\pm 2)$: therefore $(f(u), g_u(h))$ and $(f(u),
g_u(h\pm 2))=(f(u), g_u(h)\pm 2)$ are adjacent vertices in $R'$.
The second possibility is that the adjacent vertices in $R$ are
$(v,h)$ and $(\widetilde{v}, \widetilde{h})$, with $v\neq
\widetilde{v}$, and $v,\widetilde{v}\in V$, so that the vertices
$f(v)$ and $ f(\widetilde{v})$ are adjacent in $P$, with $f(v)\neq
f(\widetilde{v})$. Moreover, by the definition of $R$, it must be that
$\rot_G(v,h)=(\widetilde{v},\widetilde{h})$. Condition (1) implies that
$\rot_G(f(v),g_v(h))=\big(f(\widetilde{v}),g_{\widetilde{v}}(\widetilde{h})\big)$, and
so $F(v,h)$ and $F(\widetilde{v},\widetilde{h})$ are adjacent
in $R'$.

Consider now the case when the adjacent vertices in $R$
have the form $(w,h)$ and $(v,k)$, with $w\in W$ and $v\in V$. By the
construction of $R$, it must be $\rot_G(w,h)=(v,k)$, with $h\in
[(2d)_i]$, for some $i\in\{e,o\}$. Moreover the definition of $F$ gives
$F(w,h)=\big(f(w), \widetilde{g}_w^{\varepsilon_w(i)}(h)\big)$ and
$F(v,k)=\big(f(v),g_v(k)\big)$. By condition (2), the vertices
$F(w,h)$ and $F(v,k)$ are adjacent in $R'$, and so $F$ preserves
the adjacency relation also in this case.

The case of two adjacent vertices $(w,h)$ and $(w',k)$ in $R$,
with $w,w'\in W$, can be similarly discussed by using condition (3), and this completes the
proof.
\end{proof}

\section{Special cases}\label{sectionspecial}

For a better understanding of the results on zig-zag product $G_1\zig G_2$ obtained in the previous sections,
we consider here two particular cases.

\begin{enumerate}
\item $G_1=K_{2d+1}$ and $G_2=C_{2d}$.
\item $G_2 = C_4$.
\end{enumerate}

\subsection{The case of the complete graph}\label{subsectioncomplete}

Theorem \ref{decomposition} shows that, fixing a bi-labelling
of $K_{2d+1}$, the connected components of $K_{2d+1}\zig C_{2d}$
are in bijection with the set of parity blocks $\{P_j\}$ in the
parity block decomposition of $K_{2d+1}$. Here, one has a
constraint on the size of the parity blocks.
First of all, recall that if $P$ is a parity block
of the graph $K_{2d+1}$, then all its vertices have degree either $d$ or
$2d$. The following proposition holds.
\begin{prop}\label{congruenze}
Consider the complete graph $K_{2d+1}$ on $2d+1$ vertices, endowed with a bi-labelling of its edges, and let $P$ be a parity block
in the parity block decomposition of $K_{2d+1}$. Let $p$ be the
number of vertices in $P$, and let $i$ be the
number of vertices of degree $2d$ in $P$. Then:
\begin{itemize}
\item if $i>0$, then $p=2d+1$ and
$i$ satisfies
$$
(i-1)(d-1)\equiv 0 \pmod{2};
$$
\item if $i=0$, then $p$ satisfies
$$
p\geq d+1 \qquad \textrm{and} \qquad pd \equiv 0 \pmod{2}.
$$
\end{itemize}
Viceversa, given a positive integer $i$ satisfying
$(1-i)(d-1)\equiv 0 \pmod{2}$, there exists a bi-labelling of the edges of $K_{2d+1}$ such that the
associated parity block decomposition contains a parity block of
$2d+1$ vertices, with $i$ vertices of degree $2d$; similarly,
given an integer $p\geq d+1$ satisfying $pd\equiv 0 \pmod{2}$, there
exists a bi-labelling of the edges of $K_{2d+1}$
such that the associated parity block decomposition contains a
parity block of $p$ vertices, all having degree $d$.
\end{prop}

\begin{proof}
Let $i>0$ be the number of vertices of degree $2d$ in $P$.
We must have $p = 2d+1$, so that there are
$p-i=2d+1-i$ vertices of degree $d$ in $P$. We want to establish
for which value of $i$ such a configuration is allowed. First of
all, it is easy to check that it must be either $i\leq d$ or
$i=p=2d+1$. In fact, if $i>d$, then necessarily it must be
$i=2d+1$ (since the remaining $2d+1-i$ vertices of $P$ would have
degree at least $d+1$). The configuration $i=2d+1$ is allowed, as
shown in Example \ref{esempiosabato} in the case of $K_5$.

Let us restrict
now to the case $0< i \leq d$. The problem of establishing for
which values of $i$ this situation can occur, is equivalent to the
problem of the existence of a $(d-i)$-regular graph on $2d+1-i$
vertices. On the other hand, it
is known that there exists a $k$-regular graph on $m$ vertices if
and only if $m\geq k+1$ and $mk$ is an even integer \cite[Exercise 8.8]{danielecitazione}.
In our case, these conditions become $2d-i+1\geq d-i+1$ (always
verified) and $(2d+1-i)(d-i) \equiv 0 \pmod{2}$. We can summarize
what we said above, by stating that a parity block of the
complete graph $K_{2d+1}$ has $i$ vertices of degree $2d$,
 and $2d+1-i$ vertices of degree $d$, with $0<i\leq d$, if and only if the
integer $i$ satisfies the condition
$$
(1-i)(d-i)\equiv 0 \pmod{2}.
$$
Now consider the case $i=0$ and let $p$ be the number of vertices
of $P$. Observe that in this case, the problem of establishing
which values of $p$ are allowed is equivalent to the problem of
the existence of a $d$-regular graph on $p$ vertices. The
argument above says that such a configuration is possible if and only if
$$
p\geq d+1 \qquad \textrm{and} \qquad pd \equiv 0 \pmod{2}.
$$
\end{proof}

\begin{example}\label{esempiosabato}\rm
In the picture below the graph $K_5$ has a bi-labelling such that
its parity decomposition consists of only one parity block: therefore, we have in this case $d=2$ and $p=i=5$. In other words, all the vertices of the unique parity block are odden.
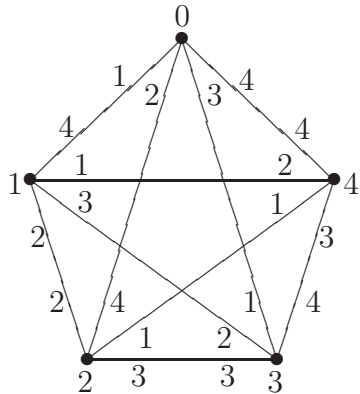
\begin{center}  \unitlength=0.25mm
\begin{picture}(250,195) \letvertex A=(120,185)\letvertex B=(40,110)\letvertex C=(70,15)
\letvertex D=(170,15)\letvertex E=(200,110)

\drawundirectededge(A,B){}\drawundirectededge(B,C){}\drawundirectededge(C,D){}\drawundirectededge(D,E){}\drawundirectededge(A,E){}
\drawundirectededge(A,C){}\drawundirectededge(E,C){}\drawundirectededge(E,B){}\drawundirectededge(D,B){}\drawundirectededge(A,D){}

\drawvertex(A){$\bullet$}\drawvertex(B){$\bullet$}
\drawvertex(C){$\bullet$}\drawvertex(D){$\bullet$}\drawvertex(E){$\bullet$}

\put(116,192){$0$} \put(28,103){$1$} \put(65,-2){$2$}
\put(165,-2){$3$}\put(205,103){$4$}

%%%%%%%%%%%%%%%%%%%%%%%%%%%%%%%%%%%%%%%%%%%%%%%%%%%%%%%%%%%%%%%%%%%%%%%%%%%%%%%%%%%%%%%%%%%%%%%%%%%%%%%%%%%%%%%%%%%%%%%%%%%%%%%%%%%%%%%%%%%%%%%
\put(83,160){$1$} \put(100,150){$2$} \put(133,150){$3$}
\put(150,158){$4$}

\put(55,133){$4$}\put(63,113){$1$}\put(65,93){$3$}
\put(40,75){$2$}

\put(50,40){$2$}\put(82,40){$4$}\put(97,21){$1$}\put(93,1){$3$}

\put(140,1){$3$}\put(138,21){$2$}\put(152,40){$1$}\put(185,40){$4$}

\put(192,75){$3$}\put(166,92){$1$}\put(170,113){$2$}\put(179,132){$4$}
\end{picture}
\end{center}
The following pictures represent: the case $d=4$, $i=1$, $p=9$; the case $d=4$, $i=0$, $p=7$; the case $d=4$, $i=2$, $p=9$.
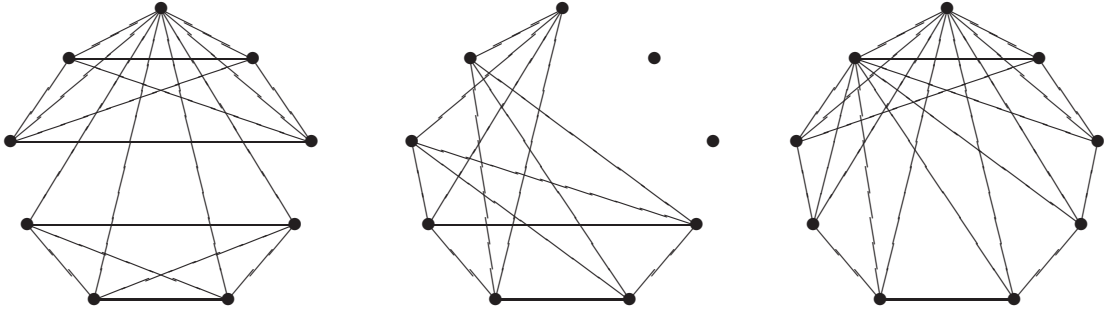
\begin{center}
\begin{picture}(700,135)\unitlength=0.22mm
\letvertex A=(120,180)\letvertex B=(65,150)\letvertex C=(30,100)
\letvertex D=(40,50)\letvertex E=(80,5)\letvertex F=(160,5)\letvertex G=(200,50)
\letvertex H=(210,100)\letvertex I=(175,150)

\drawundirectededge(A,B){}\drawundirectededge(A,C){}\drawundirectededge(A,D){}\drawundirectededge(A,E){}\drawundirectededge(A,F){}
\drawundirectededge(A,G){}\drawundirectededge(A,H){}\drawundirectededge(A,I){}

\drawundirectededge(C,B){}\drawundirectededge(B,H){}\drawundirectededge(B,I){}

\drawundirectededge(C,H){}\drawundirectededge(C,I){}

\drawundirectededge(D,E){}\drawundirectededge(D,F){}\drawundirectededge(D,G){}

\drawundirectededge(E,F){}\drawundirectededge(E,G){}

\drawundirectededge(F,G){}\drawundirectededge(H,I){}
\drawvertex(A){$\bullet$}\drawvertex(B){$\bullet$}
\drawvertex(C){$\bullet$}\drawvertex(D){$\bullet$}\drawvertex(E){$\bullet$}\drawvertex(F){$\bullet$}
\drawvertex(G){$\bullet$}\drawvertex(H){$\bullet$}\drawvertex(I){$\bullet$}
%%%%%%%%%%%%%%%%%%%%%%%%%%%%%%%%%%%%%%%%%%%%%%%%%%%%%%%%%%%%%%%%%%%%%%%%%%%%%%%%%%%%%%%%%%%%%%%%%%%%%%%%%%%%%%%%%%%%%%%%%%%%%%%%%%%%%%%%%%%%%%%%%%%%%%%%%%%%%%%%%%%%%%%%%%%%%
\letvertex AA=(360,180)\letvertex BB=(305,150)\letvertex CC=(270,100)
\letvertex DD=(280,50)\letvertex EE=(320,5)\letvertex FF=(400,5)\letvertex GG=(440,50)
\letvertex HH=(450,100)\letvertex II=(415,150)

\drawundirectededge(AA,BB){} \drawundirectededge(AA,DD){}
\drawundirectededge(CC,AA){}
\drawundirectededge(AA,EE){}
\drawundirectededge(EE,BB){}
\drawundirectededge(FF,BB){}
         \drawundirectededge(GG,BB){} \drawundirectededge(CC,DD){}
\drawundirectededge(CC,FF){}
\drawundirectededge(CC,GG){}
\drawundirectededge(EE,DD){}
\drawundirectededge(DD,GG){}

   \drawundirectededge(FF,GG){}
\drawundirectededge(EE,FF){}

 \drawvertex(AA){$\bullet$}\drawvertex(BB){$\bullet$}
\drawvertex(CC){$\bullet$}\drawvertex(DD){$\bullet$}\drawvertex(EE){$\bullet$}\drawvertex(FF){$\bullet$}
\drawvertex(GG){$\bullet$}\drawvertex(HH){$\bullet$}\drawvertex(II){$\bullet$}

%%%%%%%%%%%%%%%%%%%%%%%%%%%%%%%%%%%%%%%%%%%%%%%%%%%%%%%%%%%%%%%%%%%%%%%%%%%%%%%%%%%%%%%%%%%%%%%%%%%%%%%%%%%%%%%%%%%%%%%%%%%%%%%%%%%%%%%%%%%%%%%%%

\letvertex a=(590,180)\letvertex b=(535,150)\letvertex c=(500,100)
\letvertex d=(510,50)\letvertex e=(550,5)\letvertex f=(630,5)\letvertex g=(670,50)
\letvertex h=(680,100)\letvertex i=(645,150)

\drawundirectededge(a,b){}\drawundirectededge(a,c){}\drawundirectededge(a,d){}\drawundirectededge(a,e){}\drawundirectededge(a,f){}
\drawundirectededge(a,g){}\drawundirectededge(a,h){}\drawundirectededge(a,i){}

\drawundirectededge(b,c){}\drawundirectededge(b,d){}\drawundirectededge(b,e){}
\drawundirectededge(b,f){}\drawundirectededge(b,g){}\drawundirectededge(b,h){}\drawundirectededge(b,i){}

\drawundirectededge(c,d){}\drawundirectededge(c,i){}\drawundirectededge(d,e){}
\drawundirectededge(e,f){}\drawundirectededge(f,g){}

\drawundirectededge(g,h){}\drawundirectededge(h,i){}

\drawvertex(a){$\bullet$}\drawvertex(b){$\bullet$}
\drawvertex(c){$\bullet$}\drawvertex(d){$\bullet$}\drawvertex(e){$\bullet$}\drawvertex(f){$\bullet$}
\drawvertex(g){$\bullet$}\drawvertex(h){$\bullet$}\drawvertex(i){$\bullet$}
\end{picture}
\end{center}
\end{example}
We recall now the classical definition and the basic properties of circulant and block circulant matrices,
which will play a special role in the description of the adjacency matrix of the graphs in Proposition \ref{goodlabel}
and in the investigation of the spectral properties of a particular sequence of zig-zag product graphs studied in Section \ref{Basilicasection}.
\begin{defi}\label{circulant}
A \textit{circulant matrix} $C$ of size $n$ is a square matrix of
the form
\begin{eqnarray}\label{circulantshape}
 C = \begin{pmatrix}
  c_{0} & c_1 & c_2   & \cdots & \cdots & c_{n-2} & c_{n-1} \\
  c_{n-1}   & c_0     & c_1  & \cdots  & \cdots & c_{n-3} & c_{n-2} \\
  c_{n-2}   & c_{n-1}     & c_0      & \cdots  & \cdots & c_{n-4} & c_{n-3} \\
\vdots  & \vdots & & \ddots & &  \vdots &  \vdots \\
  \vdots & \vdots & & &  \ddots & \vdots & \vdots \\
  c_2 & c_3 & \cdots & \cdots & c_{n-1} &  c_0 & c_1 \\
  c_1 & c_2 & c_3 & \cdots & \cdots & c_{n-1} & c_0
\end{pmatrix},
\end{eqnarray}
with $c_i\in \mathbb{C}$, for every $i=0,1,\ldots, n-1$.
\end{defi}
The spectral analysis of circulant matrices is well known
\cite{davis}. More precisely, it is known that the eigenvectors of
the matrix $C$ are the vectors ${\bf v}_j=\big(1, w^j, w^{2j}, \ldots,
w^{(n-1)j}\big)$, for $j=0,1,\ldots, n-1$, where
$$
w = \exp\left(\frac{2\pi i}{n}\right), \qquad \qquad i^2=-1.
$$
The associated eigenvalues are the complex numbers
$$
\lambda_j = \sum_{k=0}^{n-1}c_k w^{jk}, \qquad \qquad
j=0,1,\ldots, n-1.
$$
More generally, a \textit{block circulant} matrix of type $(n,m)$
is a matrix of the form \eqref{circulantshape}, where $c_i$ is a
square matrix of size $m$ with entries in $\mathbb{C}$, for every $i=0,1,\ldots,n-1$. For the
spectral analysis of a block circulant matrix one can refer, for instance, to the paper \cite{tee}.\\

In the following proposition, we explicitly describe a particular bi-labelling of the
graph $K_{2d+1}$ such that $K_{2d+1}\zig C_{2d}$ is connected for each $d\geq 1$.
We will use the notation $V(K_{2d+1})=\{0,1,\ldots, 2d\}$ and $V(C_{2d})=\{1,\ldots, 2d\}$.
The vertices of $K_{2d+1}\zig C_{2d}$ will be pairs of type $(i,j)$,
with $i\in V(K_{2d+1})$ and $j\in V(C_{2d})$, with the convention
that the sum in the first coordinate is mod $2d+1$, and in the
second component the sum is given by equation \eqref{somma}, where $d$ is now replaced by $2d$.
\begin{prop}\label{goodlabel}
Let $\mathcal{L}$ be the bi-labelling of $K_{2d+1}$ such that, with
the above notation:
$$
\rot_{K_{2d+1}}(i,j)=(i+j, 2d-j+1).
$$
Then $K_{2d+1}\zig C_{2d}$ is connected and there exists an
ordering of the vertices of $K_{2d+1}\zig C_{2d}$ such that the
associated adjacency matrix is a block circulant matrix.
\end{prop}
\begin{proof}
By using the symmetry of our construction and the fact that the
graph $K_{2d+1}$ is complete, we have just to prove, by
Theorem \ref{decomposition}, that the parity block
decomposition of $K_{2d+1}$ consists of only one parity block $P$
such that $deg (i)=2d$ in $P$, for every $i\in V(K_{2d+1})$.
Without loss of generality, we can suppose that the vertex $i$ is
odd in $P$. Furthermore, by the construction of $\mathcal{L}$, we have that the condition
$\rot_{K_{2d+1}}(i,j)=(i', j')$ implies $j+j'=2d+1$, and so $j,
j'$ have different parities. Note that in $K_{2d+1}$ the path
$\{v=i,v_1=i+3,v_2=i+2,i=v\}$ described by
$\rot_{K_{2d+1}}(i,3)=(i+3,2d-2)$,
$\rot_{K_{2d+1}}(i+3,2d)=(i+2,1)$,
$\rot_{K_{2d+1}}(i+2,2d-1)=(i,2)$ satisfies the parity
conditions. This implies that $i$ is also even, so that $i$ is odden. We can repeat the same
construction for every other vertex, possibly reversing the path
for vertices with even parity, so that every vertex of $P$ has degree $2d$.
The second statement follows by observing that
$\rot_{K_{2d+1}}(i,j)=(i+j, 2d-j+1)$ implies
$\rot_{K_{2d+1}}(i+1,j)=(i+j+1, 2d-j+1)$, and so the lexicographic
order of the vertices $\{(i,j) :\ i\in V(K_{2d+1})$, $j\in V(C_{2d})\}$ of $K_{2d+1}\zig C_{2d}$ produces a block
circulant adjacency matrix. More precisely, there are $2d+1$
blocks, indexed by the vertices of $K_{2d+1}$, each of size $2d$.
\end{proof}
\begin{example}\rm
In the following picture, the bi-labelling $\mathcal{L}$ is
represented in the case of $K_5$. Observe that the labels around each vertex of the graph, regarded as integer numbers in $\{1,2,3,4\}$, increase in anticlockwise sense.
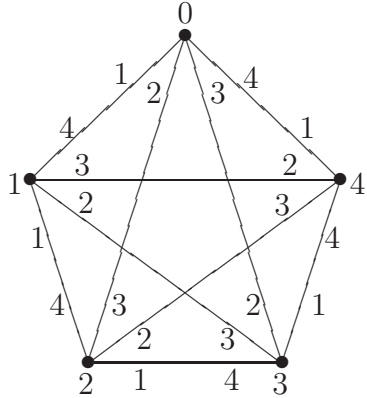
\begin{center}
\begin{picture}(250,180)     \unitlength=0.255mm
\letvertex A=(120,185)\letvertex B=(40,110)\letvertex C=(70,15)
\letvertex D=(170,15)\letvertex E=(200,110)

\drawundirectededge(A,B){}\drawundirectededge(B,C){}\drawundirectededge(C,D){}\drawundirectededge(D,E){}\drawundirectededge(A,E){}
\drawundirectededge(A,C){}\drawundirectededge(E,C){}\drawundirectededge(E,B){}\drawundirectededge(D,B){}\drawundirectededge(A,D){}

\drawvertex(A){$\bullet$}\drawvertex(B){$\bullet$}
\drawvertex(C){$\bullet$}\drawvertex(D){$\bullet$}\drawvertex(E){$\bullet$}

\put(116,192){$0$} \put(28,103){$1$} \put(65,-1){$2$}
\put(165,-1){$3$}\put(205,103){$4$}

%%%%%%%%%%%%%%%%%%%%%%%%%%%%%%%%%%%%%%%%%%%%%%%%%%%%%%%%%%%%%%%%%%%%%%%%%%%%%%%%%%%%%%%%%%%%%%%%%%%%%%%%%%%%%%%%%%%%%%%%%%%%%%%%%%%%%%%%%%%%%%%
\put(83,160){$1$} \put(100,150){$2$} \put(133,150){$3$}
\put(150,158){$4$}

\put(55,132){$4$}\put(63,113){$3$}\put(65,93){$2$}
\put(40,75){$1$}

\put(50,40){$4$}\put(82,40){$3$}\put(95,22){$2$}\put(93,1){$1$}

\put(140,1){$4$}\put(138,22){$3$}\put(151,40){$2$}\put(185,40){$1$}

\put(192,75){$4$}\put(166,92){$3$}\put(170,113){$2$}\put(179,132){$1$}
\end{picture}
\end{center}
The adjacency matrix of the graph $K_5\zig C_4$, with the lexicographic order of its vertices, is   given by
$$
\left(
  \begin{array}{ccccc}
    C_0 & C_1 & C_2 & C_3 & C_4 \\
    C_4 & C_0 & C_1 & C_2 & C_3 \\
    C_3 & C_4 & C_0 & C_1 & C_2 \\
    C_2 & C_3 & C_4 & C_0 & C_1 \\
    C_1 & C_2 & C_3 & C_4 & C_0 \\
  \end{array}
\right),
$$
where
$$
C_1=C_3=\left(
      \begin{array}{cccc}
        0 & 0 & 0 & 0 \\
        1 & 0 & 1 & 0 \\
        0 & 0 & 0 & 0 \\
        1 & 0 & 1 & 0 \\
      \end{array}
    \right)        \qquad  \qquad C_2=C_4=\left(
      \begin{array}{cccc}
        0 & 1 & 0 & 1 \\
        0 & 0 & 0 & 0 \\
        0 & 1 & 0 & 1 \\
        0 & 0 & 0 & 0 \\
      \end{array}
    \right)
$$
and $C_0$ is the zero matrix of size $4$.
\end{example}

\subsection{The case of a $4$-regular graph}\label{subsectionquattro}
In this subsection we are interested in the special case of the zig-zag
product $G\zig C_4$, where $G$ is a $4$-regular graph. This
particular choice forces the structure of the zig-zag product
$G\zig C_4$ to be highly regular. Actually, one can order the
vertices of $G\zig C_4$ in such a way that the graph is a disjoint
union of connected components, whose adjacency matrices are all
circulant; therefore, a complete spectral analysis can be
performed in this case. We will see an application to the
case of Schreier graphs associated with group actions in Section \ref{Basilicasection}.

Notice that Corollary \ref{corollario4} implies that the
isomorphism classes of the connected components of $G\zig C_4$ are
determined by the size of the corresponding parity blocks of $G$. We recall that the pseudo-replacement graphs are given by an
alternate sequence of simple edges and cycles of length $2$.

In what follows, a \textit{double cycle graph} $DC_n$ of length
$n$ is a $4$-regular graph with $2n$ vertices in which any vertex belongs
exactly to two papillon graphs. The picture below
represents the graph $DC_{16}$.
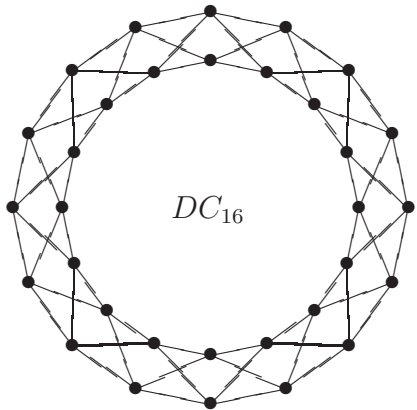
\begin{center}
\begin{picture}(350,180)\unitlength=0,13mm

 \put(380,200){$DC_{16}$}

%%%%%%%%%%%%%%%%%%%%%%%%%%%%%%%%%%%%%%%%%%%%%%%%%%%%%%%%%%%%%%%%%%%%%%%%%%%%%%%%%%%%%%%%%%%%%%%%%%%%
\letvertex A=(420,410)\letvertex B=(496,394)\letvertex C=(560,350)
\letvertex D=(604,286)\letvertex E=(620,210)\letvertex F=(604,134)\letvertex G=(560,70)
\letvertex H=(496,26)\letvertex I=(420,10)\letvertex L=(344,26)\letvertex M=(280,70)
\letvertex N=(236,134)\letvertex O=(220,210)\letvertex P=(236,286)\letvertex Q=(280,350)
\letvertex R=(344,394)

\letvertex a=(420,360)\letvertex b=(477,348)\letvertex c=(525,315)
\letvertex d=(558,267)\letvertex e=(570,210)\letvertex f=(558,153)\letvertex g=(525,105)
\letvertex h=(477,72)\letvertex i=(420,60)\letvertex l=(363,72)\letvertex m=(315,105)
\letvertex n=(282,153)\letvertex o=(270,210)\letvertex p=(282,267)\letvertex q=(315,315)
\letvertex r=(363,348)

\drawvertex(A){$\bullet$}\drawvertex(B){$\bullet$}
\drawvertex(C){$\bullet$}\drawvertex(D){$\bullet$}
\drawvertex(E){$\bullet$}\drawvertex(F){$\bullet$}
\drawvertex(G){$\bullet$}\drawvertex(I){$\bullet$}\drawvertex(M){$\bullet$}\drawvertex(N){$\bullet$}
\drawvertex(H){$\bullet$}\drawvertex(L){$\bullet$}\drawvertex(O){$\bullet$}\drawvertex(P){$\bullet$}
\drawvertex(Q){$\bullet$}\drawvertex(R){$\bullet$}

\drawvertex(a){$\bullet$}\drawvertex(b){$\bullet$}
\drawvertex(c){$\bullet$}\drawvertex(d){$\bullet$}
\drawvertex(e){$\bullet$}\drawvertex(f){$\bullet$}
\drawvertex(g){$\bullet$}\drawvertex(i){$\bullet$}\drawvertex(m){$\bullet$}\drawvertex(n){$\bullet$}
\drawvertex(h){$\bullet$}\drawvertex(l){$\bullet$}\drawvertex(o){$\bullet$}\drawvertex(p){$\bullet$}
\drawvertex(q){$\bullet$}\drawvertex(r){$\bullet$}

\drawundirectededge(A,B){}\drawundirectededge(B,C){}\drawundirectededge(C,D){}\drawundirectededge(D,E){}
\drawundirectededge(E,F){}\drawundirectededge(F,G){}\drawundirectededge(G,H){}\drawundirectededge(H,I){}
\drawundirectededge(I,L){}\drawundirectededge(L,M){}\drawundirectededge(M,N){}
\drawundirectededge(N,O){}\drawundirectededge(O,P){}\drawundirectededge(P,Q){}\drawundirectededge(Q,R){}
\drawundirectededge(R,A){}

\drawundirectededge(a,b){}\drawundirectededge(b,c){}\drawundirectededge(c,d){}\drawundirectededge(d,e){}
\drawundirectededge(e,f){}\drawundirectededge(f,g){}\drawundirectededge(g,h){}\drawundirectededge(h,i){}
\drawundirectededge(i,l){}\drawundirectededge(l,m){}\drawundirectededge(m,n){}
\drawundirectededge(n,o){}\drawundirectededge(o,p){}\drawundirectededge(p,q){}\drawundirectededge(q,r){}
\drawundirectededge(r,a){}

\drawundirectededge(A,b){}\drawundirectededge(b,C){}\drawundirectededge(C,d){}\drawundirectededge(d,E){}
\drawundirectededge(E,f){}\drawundirectededge(f,G){}
\drawundirectededge(G,h){}\drawundirectededge(h,I){}
\drawundirectededge(I,l){}\drawundirectededge(l,M){}
\drawundirectededge(M,n){}\drawundirectededge(n,O){}\drawundirectededge(O,p){}
\drawundirectededge(p,Q){}\drawundirectededge(Q,r){}\drawundirectededge(r,A){}

\drawundirectededge(a,B){}\drawundirectededge(B,c){}\drawundirectededge(c,D){}\drawundirectededge(D,e){}
\drawundirectededge(e,F){}\drawundirectededge(F,g){}
\drawundirectededge(g,H){}\drawundirectededge(H,i){}
\drawundirectededge(i,L){}\drawundirectededge(L,m){}
\drawundirectededge(m,N){}\drawundirectededge(N,o){}\drawundirectededge(o,P){}
\drawundirectededge(P,q){}\drawundirectededge(q,R){}\drawundirectededge(R,a){}
\end{picture}
\end{center}
The following proposition shows that the connected components of the zig-zag
product $G\zig C_4$ are isomorphic to double cycle graphs.

\begin{prop}\label{propprop}
Let $G$ be a $4$-regular graph and let $\{v=v_0, v_1,\ldots,
v_n=v\}$ be a path spanning a parity block $P$ of the parity block decomposition of $G$. Then the
corresponding connected component $S$ in $G\zig C_4$ is isomorphic
to the double cycle graph $DC_n$. More precisely, if
$\rot_G(v_i,h_i)=(v_{i+1},k_i)$, then $(v_i,h_i\pm 1)$ and
$(v_{i\pm 1}, h_{i\pm 1} \pm 1)$ form two adjacent papillon
graphs in $DC_n$.
\end{prop}
\begin{proof}
Two consecutive edges $\{v_{i-1},v_i\}$ and $\{v_i,v_{i+1}\}$ in
the path spanning $P$
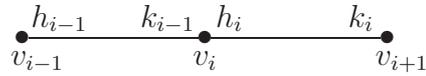
\begin{center}
\begin{picture}(200,20)
\letvertex A=(20,10)\letvertex B=(100,10)\letvertex C=(180,10)

\drawundirectededge(A,B){}\drawundirectededge(B,C){}

\drawvertex(A){$\bullet$}\drawvertex(B){$\bullet$}
\drawvertex(C){$\bullet$}

\put(15,-2){$v_{i-1}$} \put(95,-2){$v_i$} \put(175,-2){$v_{i+1}$}

\put(24,15){$h_{i-1}$}\put(72,15){$k_{i-1}$} \put(105,15){$h_i$}
\put(163,15){$k_i$}
\end{picture}
\end{center}
produce the following papillon subgraphs, respectively:
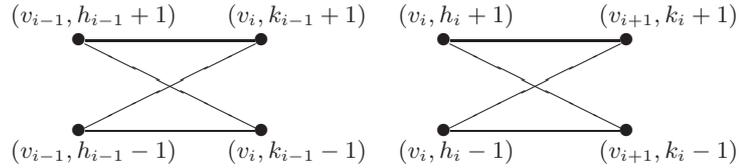
\begin{center}
\begin{picture}(300,70)
\letvertex A=(30,55)\letvertex B=(30,15)\letvertex C=(110,15)\letvertex D=(110,55)
\letvertex E=(190,55)\letvertex F=(190,15)\letvertex G=(270,15)\letvertex H=(270,55)

\drawundirectededge(A,C){}\drawundirectededge(A,D){}\drawundirectededge(B,C){}\drawundirectededge(B,D){}
\drawundirectededge(E,G){}\drawundirectededge(E,H){}\drawundirectededge(F,G){}\drawundirectededge(F,H){}

\drawvertex(A){$\bullet$}\drawvertex(B){$\bullet$}
\drawvertex(C){$\bullet$}\drawvertex(D){$\bullet$}\drawvertex(E){$\bullet$}
\drawvertex(F){$\bullet$}\drawvertex(G){$\bullet$}\drawvertex(H){$\bullet$}

{\scriptsize \put(0,63){$(v_{i-1}, h_{i-1}+1)$}

\put(0,3){$(v_{i-1}, h_{i-1}-1)$}

\put(95,63){$(v_{i}, k_{i-1}+1)$}

\put(95,3){$(v_{i}, k_{i-1}-1)$}

\put(170,63){$(v_{i}, h_{i}+1)$}

\put(170, 3){$(v_{i}, h_{i}-1)$}

\put(258,63){$(v_{i+1}, k_{i}+1)$}

\put(258,3){$(v_{i+1}, k_{i}-1)$}}
\end{picture}
\end{center}
Notice that the sets $\{h_i\pm 1\}$ and $\{k_{i-1}\pm 1\}$
coincide, since the spanning path must satisfy the parity properties. Therefore, we can identify the pair of vertices $(v_i,
k_{i-1}\pm 1)$ with the pair of vertices $(v_i, h_i\pm 1)$,
getting in $G\zig C_4$ the following subgraph obtained by gluing together two single papillon subgraphs.
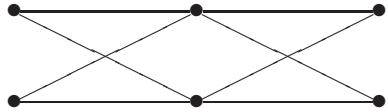
\begin{center}
\begin{picture}(300,60)
\letvertex A=(70,55)\letvertex B=(70,15)\letvertex C=(150,15)\letvertex D=(150,55)
\letvertex G=(230,15)\letvertex H=(230,55)

\drawundirectededge(A,C){}\drawundirectededge(A,D){}\drawundirectededge(B,C){}\drawundirectededge(B,D){}
\drawundirectededge(D,G){}\drawundirectededge(D,H){}\drawundirectededge(C,G){}\drawundirectededge(C,H){}

\drawvertex(A){$\bullet$}\drawvertex(B){$\bullet$}
\drawvertex(C){$\bullet$}\drawvertex(D){$\bullet$}\drawvertex(G){$\bullet$}\drawvertex(H){$\bullet$}
\end{picture}
\end{center}
Since the path spanning $P$ contains $n$ vertices, the assertion
follows from Corollary \ref{corollario4}.
\end{proof}
\begin{example}\label{example6}   \rm
In this example, we consider the graph $G=K_5$ endowed with a bi-labelling
$\mathcal{L}$ such that $K_5=P_1\cup P_2$.
The corresponding connected components $S_1$ and $S_2$ of $K_5\zig C_4$ are given as well.
\begin{center}
\begin{picture}(400,195)\unitlength=0.24mm
\letvertex A=(120,185)\letvertex B=(40,110)\letvertex C=(70,15)
\letvertex D=(170,15)\letvertex E=(200,110)

\letvertex F=(350,160)\letvertex G=(350,60)\letvertex H=(450,60)\letvertex I=(450,160)

\drawundirectededge(A,B){}\drawundirectededge(B,C){}\drawundirectededge(C,D){}\drawundirectededge(D,E){}\drawundirectededge(A,E){}
\drawundirectededge(A,C){}\drawundirectededge(E,C){}\drawundirectededge(E,B){}\drawundirectededge(D,B){}\drawundirectededge(A,D){}

\drawundirectededge(F,G){}\drawundirectededge(G,H){}\drawundirectededge(H,I){}\drawundirectededge(I,F){}

\drawvertex(A){$\bullet$}\drawvertex(B){$\bullet$}
\drawvertex(C){$\bullet$}\drawvertex(D){$\bullet$}\drawvertex(E){$\bullet$}\drawvertex(F){$\bullet$}
\drawvertex(G){$\bullet$}\drawvertex(H){$\bullet$}\drawvertex(I){$\bullet$}

\put(-20,105){$K_5$} \put(500,105){$C_4$}

\put(115,191){$0$} \put(28,103){$1$} \put(65,-1){$2$}
\put(165,-1){$3$}\put(205,103){$4$}

\put(345,165){$1$} \put(347,44){$2$} \put(447,44){$3$}
\put(448,165){$4$}
%%%%%%%%%%%%%%%%%%%%%%%%%%%%%%%%%%%%%%%%%%%%%%%%%%%%%%%%%%%%%%%%%%%%%%%%%%%%%%%%%%%%%%%%%%%%%%%%%%%%%%%%%%%%%%%%%%%%%%%%%%%%%%%%%%%%%%%%%%%%%%%

\put(83,160){$2$} \put(100,150){$3$} \put(133,151){$1$}
\put(150,158){$4$}

\put(55,133){$1$}\put(63,113){$2$}\put(65,92){$3$}
\put(40,74){$4$}

\put(50,40){$2$}\put(80,40){$1$}\put(95,22){$3$}\put(93,2){$4$}

\put(140,2){$3$}\put(138,21){$2$}\put(152,40){$4$}\put(185,40){$1$}

\put(192,72){$3$}\put(166,92){$2$}\put(170,113){$1$}\put(179,132){$4$}

%\put(285,137){$A$}\put(285,80){$A$}\put(404,137){$A$}\put(404,80){$A$}
%\put(320,163){$B$}\put(370,163){$B$}\put(320,45){$B$}\put(370,45){$B$}
\end{picture}
\end{center}
\begin{center}
\begin{picture}(450,330)\unitlength=0.24mm
\letvertex A=(120,285)\letvertex B=(40,210)\letvertex C=(70,115)
\letvertex D=(170,115)\letvertex E=(200,210)

\put(110,20){$P_1$} \put(360,20){$S_1\simeq DC_6$}

\drawundirectededge(A,B){}\drawundirectededge(A,E){}
\drawundirectededge(A,C){}\drawundirectededge(E,C){}\drawundirectededge(D,B){}\drawundirectededge(A,D){}

\drawvertex(A){$\bullet$}\drawvertex(B){$\bullet$}
\drawvertex(C){$\bullet$}\drawvertex(D){$\bullet$}\drawvertex(E){$\bullet$}

\put(115,290){$0$} \put(28,203){$1$} \put(65,99){$2$}
\put(165,99){$3$}\put(205,203){$4$}

\put(85,260){$2$} \put(100,250){$3$} \put(133,250){$1$}
\put(150,258){$4$}

\put(55,232){$1$}\put(65,192){$3$}\put(80,140){$1$}\put(95,123){$3$}\put(137,123){$2$}\put(152,140){$4$}\put(166,192){$2$}\put(179,232){$4$}

%%%%%%%%%%%%%%%%%%%%%%%%%%%%%%%%%%%%%%%%%%%%%%%%%%%%%%%%%%%%%%%%%%%%%%%%%%%%%%%%%%%%%%%%%%%%%%%%%%%%%%%%%%%%%%%%%%%%%%%%%%%%%%%%%%%%%%%%%%%%%%%%
                       \letvertex AA=(400,320)\letvertex BB=(296,260)\letvertex CC=(296,140)
\letvertex DD=(400,80)\letvertex EE=(504,140)\letvertex FF=(504,260)\letvertex a=(400,260)\letvertex b=(348,230)
\letvertex c=(348,170)\letvertex d=(400,140)\letvertex e=(452,170)\letvertex f=(452,230)

\drawundirectededge(AA,BB){}\drawundirectededge(BB,CC){}
\drawundirectededge(CC,DD){}\drawundirectededge(DD,EE){}\drawundirectededge(EE,FF){}\drawundirectededge(AA,FF){}

\drawundirectededge(a,b){}\drawundirectededge(b,c){}
\drawundirectededge(c,d){}\drawundirectededge(d,e){}\drawundirectededge(e,f){}\drawundirectededge(f,a){}

\drawundirectededge(AA,b){}\drawundirectededge(b,CC){}
\drawundirectededge(CC,d){}\drawundirectededge(d,EE){}\drawundirectededge(EE,f){}\drawundirectededge(AA,f){}

\drawundirectededge(a,BB){}\drawundirectededge(BB,c){}
\drawundirectededge(c,DD){}\drawundirectededge(DD,e){}\drawundirectededge(e,FF){}\drawundirectededge(a,FF){}

\drawvertex(AA){$\bullet$}\drawvertex(BB){$\bullet$}
\drawvertex(CC){$\bullet$}\drawvertex(DD){$\bullet$}\drawvertex(EE){$\bullet$}
\drawvertex(FF){$\bullet$}\drawvertex(a){$\bullet$}
\drawvertex(b){$\bullet$}\drawvertex(c){$\bullet$}\drawvertex(d){$\bullet$}\drawvertex(e){$\bullet$}\drawvertex(f){$\bullet$}
   {\footnotesize
\put(385,328){$(0,1)$}\put(275,269){$(4,1)$}
\put(278,120){$(2,4)$}\put(385,65){$(0,2)$}\put(492,120){$(3,3)$}\put(505,267){$(1,2)$}

\put(385,240){$(0,3)$}\put(350,215){$(4,3)$}
\put(350,175){$(2,2)$}\put(385,153){$(0,4)$}\put(418,175){$(3,1)$}\put(418,215){$(1,4)$}      }
\end{picture}
\end{center}
%%%%%%%%%%%%%%%%%%%%%%%%%%%%%%%%%%%%%%%%%%%%%%%%%%%%%%%%%%%%%%%%%%%%%%%%%%%%%%%%%%%%%%%%%%%%%%%%%%%%%%%%%%%%%%%%%%%%%%%%%%%%%%%%%%%%%%%%%%%%%%%%%%%%%%
\begin{center}
\begin{picture}(470,150)\unitlength=0.25mm
\letvertex B=(40,150)\letvertex C=(70,55)
\letvertex D=(170,55)\letvertex E=(200,150)
\put(110,-5){$P_2$} \put(360,-5){$S_2\simeq DC_4$}

\drawundirectededge(B,C){}\drawundirectededge(C,D){}\drawundirectededge(D,E){}\drawundirectededge(B,E){}

\drawvertex(B){$\bullet$}\drawvertex(C){$\bullet$}\drawvertex(D){$\bullet$}\drawvertex(E){$\bullet$}

\put(28,143){$1$}\put(65,38){$2$}
\put(165,38){$3$}\put(205,143){$4$}

%%%%%%%%%%%%%%%%%%%%%%%%%%%%%%%%%%%%%%%%%%%%%%%%%%%%%%%%%%%%%%%%%%%%%%%%%%%%%%%%%%%%%%%%%%%%%%%%%%%%%%%%%%%%%%%%%%%%%%%%%%%%%%%%%%%%%%%%%%%%%%%
\put(63,153){$2$}\put(40,115){$4$}

\put(50,80){$2$}\put(93,41){$4$}

\put(140,41){$3$}\put(185,80){$1$}

\put(192,115){$3$}\put(170,153){$1$}

\letvertex AA=(340,160)\letvertex BB=(340,40)\letvertex CC=(460,40)
\letvertex DD=(460,160)\letvertex a=(370,130)\letvertex b=(370,70)
\letvertex c=(430,70)\letvertex d=(430,130)

\drawundirectededge(AA,BB){}\drawundirectededge(BB,CC){}
\drawundirectededge(CC,DD){}\drawundirectededge(DD,AA){}

\drawundirectededge(a,b){}\drawundirectededge(b,c){}
\drawundirectededge(c,d){}\drawundirectededge(d,a){}

\drawundirectededge(AA,b){}\drawundirectededge(b,CC){}
\drawundirectededge(CC,d){}\drawundirectededge(d,AA){}

\drawundirectededge(a,BB){}\drawundirectededge(BB,c){}
\drawundirectededge(c,DD){}\drawundirectededge(DD,a){}

\drawvertex(AA){$\bullet$}\drawvertex(BB){$\bullet$}
\drawvertex(CC){$\bullet$}\drawvertex(DD){$\bullet$}\drawvertex(a){$\bullet$}
\drawvertex(b){$\bullet$}\drawvertex(c){$\bullet$}\drawvertex(d){$\bullet$}

 {\footnotesize
\put(324,168){$(1,1)$}\put(324,25){$(2,3)$}
\put(448,25){$(3,4)$}\put(448,168){$(4,2)$}     }
 {\tiny
\put(373,115){$(1,3)$}\put(373,77){$(2,1)$}
\put(403,77){$(3,2)$}\put(403,115){$(4,4)$} }

\end{picture}
\end{center}
\end{example}

%%%%%%%%%%%%%%%%%%%%%%%%%%%%%%%%%%%%%%%%%%%%%%%%%%%%%%%%%%%%%%%%%%%%%%%%%%%%%%%%%%%%%%%%%%%%%%%%%%%%%%%%%%%%%%%%%%%%%%%%%%%%%%%%%%%%%

\section{Schreier graphs: an application}\label{Basilicasection}

Let $X=\{0,1\}$ be a binary alphabet. Denote by $X^0$ the set consisting of the empty word, and by $X^n=\{w=x_1\ldots
x_n : x_i \in X\}$ the set of words of length $n$ over the
alphabet $X$, for each $n\geq 1$. Put $X^{\ast}=\bigcup_{n\geq 0} X^n$ and let
$X^{\infty} = \{w=x_1x_2\ldots \}$ be the set of infinite words
over $X$.

Consider the so-called Basilica group $B$ acting on the set $X^\ast \cup
X^{\infty}$, which is the group generated by the following three-state automaton:
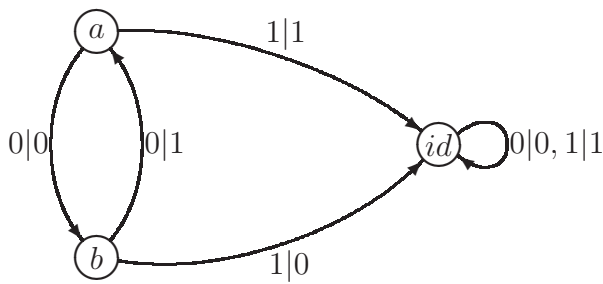
\begin{center}
\begin{picture}(300,145)
\thicklines \setvertexdiam{20} \setprofcurve{15}\setloopdiam{20}
\letstate A=(75,140)
\letstate C=(225,90)
\letstate B=(75,40)

\drawstate(A){$a$} \drawstate(B){$b$} \drawstate(C){$id$}

\drawcurvededge(A,C){$1|1$}
\setprofcurve{-20}\drawcurvededge[r](B,C){$1|0$}
\drawloop[r](C){$0|0,1|1$} \drawcurvededge[r](A,B){$0|0$}
\drawcurvededge[r](B,A){$0|1$}
\end{picture}
\end{center}
The states $a$ and $b$ of the automaton are the generators of the
group, whereas $id$ represents the identity action: therefore, it
can be read from the automaton that the action of $a$ and $b$ is
given by
$$
a(0w) = 0b(w) \qquad \qquad \qquad a(1w) = 1w
$$
$$
b(0w) = 1a(w) \qquad \qquad \qquad b(1w) = 0w,
$$
for every $w\in X^\ast\cup X^\infty$. In particular, $B$ maps
$X^n$ into $X^n$, for each $n \geq 1$, and $X^\infty$ into
$X^\infty$. Furthermore, it is easy to check that the action of
$B$ on $X^n$ is transitive for every $n$.

The Basilica group belongs to the important class of self-similar groups
and was introduced by R. Grigorchuk and A. \.{Z}uk \cite{grizuk}.
It is a remarkable fact due to Nekrashevych \cite{nekrashevych}
that it can be described as the iterated monodromy group
$IMG(z^2-1)$ of the complex polynomial $z^2-1$. Moreover, $B$ is
the first example of an amenable group (a highly non--trivial and
deep result of Bartholdi and Vir\'ag \cite{amen}) not belonging to
the class $SG$ of subexponentially amenable groups, which is the
smallest class containing all groups of subexponential growth and
closed after taking subgroups, quotients, extensions and direct
unions. In \cite{ADM}, the action of the Basilica group on the set
$X^\ast\cup X^\infty$ is studied from the point of view of
Gelfand pairs theory.

For each $n\geq 1$, let $\Gamma_n$ be the (orbital) \textit{Schreier
graph} associated with the action of $B$ on $X^n$. By definition, the vertices
of $\Gamma_n$ are the elements of $X^n$, and two vertices $u,v$
are connected by an edge labelled by $s$ close to $u$ and by
$s^{-1}$ close to $v$ if $s(u)=v$ (so that $s^{-1}(v)=u)$. Here,
we are assuming $s\in \{a,b\}$. Observe that $\Gamma_n$ is a
connected graph on $2^n$ vertices, since $B$ acts transitively on
each level; moreover, there are $4$ edges issuing from every
vertex, so that $\Gamma_n$ is a $4$-regular graph, and the labels near $v$ are given by $a^{\pm 1}, b^{\pm 1}$, for every $v\in X^n$.

Similarly, one can consider the action of $B$ on $X^\infty$ and
define the (orbital) infinite Schreier graph $\Gamma_\xi$, describing the orbit $B(\xi)$ of the element $\xi\in X^\infty$
under the action of the generators of $B$. Note
that, even though the group acts transitively on $X^n$, for each $n\geq 1$, there exist uncountably many orbits in $X^\infty$ under
the action of $B$. The infinite Schreier graph $\Gamma_{\xi}$ can
be approximated (as a rooted graph) by finite Schreier graphs
$\Gamma_n$, as $n\rightarrow\infty$, in the compact space of
rooted graphs of uniformly bounded degree endowed with {\it
pointed Gromov-Hausdorff convergence} \cite[Chapter 3]{Grom}: if
$\xi=x_1x_2\ldots \in X^{\infty}$ and $\xi_n=x_1\ldots x_n$ is its
prefix of length $n$, then one has
$$
\lim_{n\to \infty}(\Gamma_n,\xi_n)=(\Gamma_\xi,\xi),
$$
where we denote by $(\Gamma_n,\xi_n)$ the graph $\Gamma_n$
regarded as a graph rooted at the vertex $\xi_n$, and by $(\Gamma_\xi, \xi)$
the graph $\Gamma_\xi$ regarded as a graph rooted at the vertex $\xi$.

In \cite{JMD2}, finite and infinite Schreier graphs of the
Basilica group are investigated. Precise substitutional rules
allowing to construct recursively the sequence of finite Schreier
graphs are provided, and a topological (up to isomorphism of rooted
graphs) classification of the infinite Schreier graphs is given
there. In the following pictures, the graphs $\Gamma_1, \Gamma_2,
\Gamma_3$ are depicted. \unitlength=0,42mm
\setloopdiam{6}\setprofcurve{4}
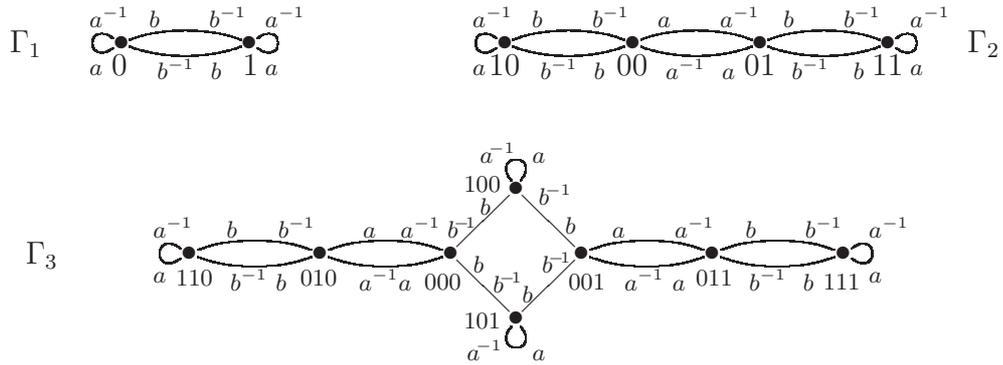
\begin{center}
\begin{picture}(300,20)
\letvertex A=(30,15)\letvertex B=(70,15)\letvertex C=(150,15)\letvertex D=(190,15)
\letvertex E=(230,15)\letvertex F=(270,15)

\put(-5,12){$\Gamma_1$} \put(295,12){$\Gamma_2$}

\put(27,5){0} \put(68,5){1}\put(145,5){10}
\put(185,5){00}\put(225,5){01}\put(265,5){11}
\drawvertex(A){$\bullet$}\drawvertex(B){$\bullet$}
\drawvertex(C){$\bullet$}\drawvertex(D){$\bullet$}
\drawvertex(E){$\bullet$}\drawvertex(F){$\bullet$}

\scriptsize{ \put(20,20){$a^{-1}$} \put(20,5){$a$}

\put(75,20){$a^{-1}$} \put(75,5){$a$}

\put(277,20){$a^{-1}$} \put(277,5){$a$}

\put(140,20){$a^{-1}$} \put(140,5){$a$}

\drawundirectedcurvededge(A,B){$\ \ \ b\ \ \ \ \ \ b^{-1}$}
\drawundirectedcurvededge(B,A){$\ b^{-1} \ \   b $}
\drawundirectedcurvededge(C,D){$\ \ \ b\ \ \ \ \ \ b^{-1}$}
\drawundirectedcurvededge(D,C){$\ b^{-1} \ \  b $}
\drawundirectedcurvededge(D,E){$\ \ \ a\ \ \ \ \ \ a^{-1}$}
\drawundirectedcurvededge(E,D){$\ a^{-1} \ \   a $}
\drawundirectedcurvededge(E,F){$\ \ b\ \ \ \ \ \ b^{-1}$}
\drawundirectedcurvededge(F,E){$\ b^{-1} \ \ \  b $}}
\drawundirectedloop[l](A){}\drawundirectedloop[r](B){}\drawundirectedloop[l](C){}
\drawundirectedloop[r](F){}
\end{picture}
\end{center}
\begin{center}
\begin{picture}(300,69)  \unitlength=0,43mm
\letvertex A=(50,30)\letvertex B=(90,30)\letvertex C=(130,30)\letvertex D=(150,50)
\letvertex E=(150,10)\letvertex F=(170,30)\letvertex G=(210,30)\letvertex H=(250,30)

\drawvertex(A){$\bullet$}\drawvertex(B){$\bullet$}
\drawvertex(C){$\bullet$}\drawvertex(D){$\bullet$}
\drawvertex(E){$\bullet$}\drawvertex(F){$\bullet$}
\drawvertex(G){$\bullet$}\drawvertex(H){$\bullet$}
\put(0,27){$\Gamma_3$}

{\scriptsize \drawundirectededge(C,D){} \drawundirectededge(E,C){}
\drawundirectededge(F,E){} \drawundirectededge(D,F){}

\put(139,42){$b$} \put(157,45){$b\!^{-1}$}

\put(129,35){$b\!^{-\!1}$} \put(165,36){$b$}

\put(137,24){$b$} \put(158,24){$b\!^{-\!1}$}

\put(143,18){$b\!^{-\!1}$} \put(152,15){$b$}

\put(39,35){$a^{-1}$} \put(39,21){$a$}

 \put(258,35){$a^{-1}$}
\put(258,21){$a$}

 \put(139,58){$a\!^{-1}$}
\put(155,58){$a$}

 \put(135,-3){$a\!^{-1}$}
\put(155,-3){$a$}

\put(46,20){110} \put(84,20){010}\put(122,19){000}
\put(134,49){100}\put(134,7){101}\put(166,19){001}
\put(205,20){011}\put(244,19){111}

\drawundirectedcurvededge(A,B){$\ \ \ \ b\ \ \ \ \  b^{-1}$}
\drawundirectedcurvededge(B,A){$\ b^{-1} \   b $}
\drawundirectedcurvededge(B,C){$\ \ \ \ a\ \ \  a^{-1}$}
\drawundirectedcurvededge(C,B){$ a^{-1}   a $}
\drawundirectedcurvededge(F,G){$\ \ \ \ a\ \ \ \ \ \ a^{-1}$}
\drawundirectedcurvededge(G,F){$\ \ a^{-1} \   a $}
\drawundirectedcurvededge(G,H){$\ \ \ \ b\ \ \ \ \ \ b^{-1}$}
\drawundirectedcurvededge(H,G){$\ b^{-1} \ \   b $}

\drawundirectedloop[l](A){}\drawundirectedloop(D){}\drawundirectedloop[b](E){}
\drawundirectedloop[r](H){}}
\end{picture}
\end{center}
For instance, the fact that the edge connecting $000$ and $101$ in
$\Gamma_3$ is labelled by $b$ near $000$ and by $b^{-1}$ near
$101$ means that $b(000) = 101$, so that $b^{-1}(101) = 000$. The
graphs $\Gamma_4$ and $\Gamma_5$ are represented below (the
bi-labelling of the edges and the words corresponding to each vertex
are omitted in order to simplify the picture).
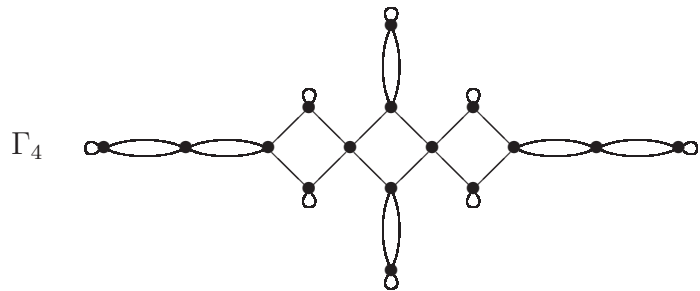
\begin{center}
\begin{picture}(200,120)\unitlength=0.27mm
\put(-35,67){$\Gamma_4$}

\letvertex A=(10,70)
\letvertex B=(50,70)\letvertex C=(90,70)\letvertex D=(110,90)\letvertex E=(110,50)
\letvertex F=(130,70)\letvertex G=(150,90)\letvertex H=(150,50)\letvertex I=(170,70)
\letvertex J=(150,130)\letvertex K=(150,10)
\letvertex L=(190,90)
\letvertex M=(190,50)\letvertex N=(210,70)
\letvertex O=(250,70)\letvertex P=(290,70)

%\scriptsize{ \put(4,62){1110} \put(42,62){0110}\put(80,62){0010}
%\put(90,88){1010}\put(90,48){1000}\put(112,67){0000}
%\put(132,87){0100}\put(132,47){0101} \put(173,67){0001}
%\put(132,127){1100}\put(132,7){1101}
%\put(193,87){1001}\put(193,47){1011}\put(206,62){0011}
%\put(243,62){0111}\put(281,62){1111}}

\drawvertex(A){$\bullet$}\drawvertex(B){$\bullet$}
\drawvertex(C){$\bullet$}\drawvertex(D){$\bullet$}
\drawvertex(E){$\bullet$}\drawvertex(F){$\bullet$}
\drawvertex(G){$\bullet$}\drawvertex(H){$\bullet$}
\drawvertex(I){$\bullet$}\drawvertex(L){$\bullet$}
\drawvertex(M){$\bullet$}\drawvertex(N){$\bullet$}
\drawvertex(O){$\bullet$}\drawvertex(P){$\bullet$}
\drawvertex(J){$\bullet$}\drawvertex(K){$\bullet$}

\drawundirectedcurvededge(A,B){}\drawundirectedcurvededge(B,A){}
\drawundirectedcurvededge(B,C){}\drawundirectedcurvededge(C,B){}
\drawundirectededge(C,D){} \drawundirectededge(D,F){}
\drawundirectededge(F,E){}
\drawundirectededge(E,C){}\drawundirectededge(F,G){}
\drawundirectededge(G,I){} \drawundirectededge(I,H){}
\drawundirectededge(H,F){} \drawundirectededge(I,L){}
\drawundirectededge(L,N){} \drawundirectededge(N,M){}
\drawundirectededge(M,I){}
\drawundirectedcurvededge(G,J){}\drawundirectedcurvededge(J,G){}
\drawundirectedcurvededge(H,K){}\drawundirectedcurvededge(K,H){}
\drawundirectedcurvededge(N,O){}\drawundirectedcurvededge(O,N){}
\drawundirectedcurvededge(O,P){}\drawundirectedcurvededge(P,O){}
\drawundirectedloop[l](A){}\drawundirectedloop(D){}\drawundirectedloop[b](E){}
\drawundirectedloop(J){}\drawundirectedloop[b](K){}\drawundirectedloop(L){}\drawundirectedloop[b](M){}
\drawundirectedloop[r](P){}
\end{picture}
\end{center} 
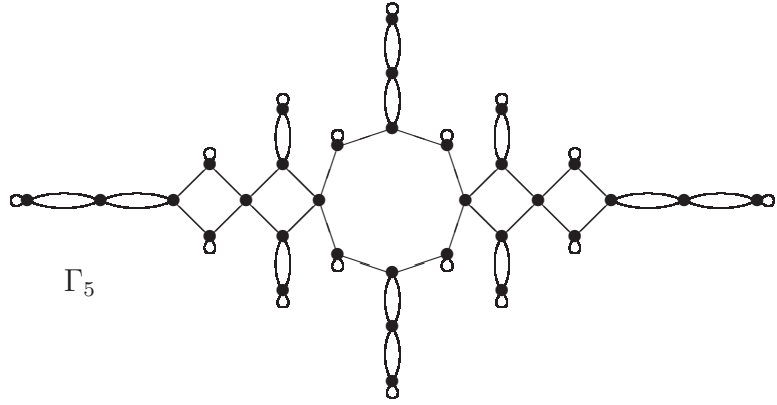
\begin{center}
\begin{picture}(235,150)\unitlength=0.24mm
\put(20,60){$\Gamma_5$}
\letvertex A=(0,110)\letvertex B=(40,110)\letvertex C=(80,110)\letvertex D=(100,130)
\letvertex E=(100,90)\letvertex F=(120,110)\letvertex G=(140,130)\letvertex H=(140,160)
\letvertex I=(160,110)\letvertex L=(140,90)\letvertex M=(140,60)\letvertex N=(170,140)
\letvertex O=(200,150)\letvertex R=(230,140)\letvertex S=(240,110)\letvertex T=(230,80)
\letvertex U=(200,70)\letvertex V=(170,80)\letvertex P=(200,180)\letvertex Q=(200,210)
\letvertex Z=(200,40)\letvertex J=(200,10)\letvertex K=(260,130)\letvertex X=(280,110)
\letvertex W=(260,90)\letvertex g=(260,160)\letvertex h=(260,60)\letvertex c=(300,130)
\letvertex Y=(300,90)\letvertex d=(320,110)\letvertex e=(360,110)\letvertex f=(400,110)

%\tiny{ \put(-7,102){11110} \put(33,102){01110}\put(69,102){00110}
%\put(80,128){10110}\put(80,87){10010}\put(101,108){00010}
%\put(121,128){01010}\put(121,158){11010} \put(141,108){00000}
%\put(121,88){01000}\put(121,58){11000}
%\put(152,137){10000}\put(192,143){00100}\put(232,137){10100}
%\put(243,108){00001}\put(232,78){10001}\put(192,74){00101}
%\put(152,77){10101}\put(182,177){01100} \put(182,207){11100}
%\put(182,37){01101}\put(182,7){11101}
%\put(263,128){01001}\put(283,108){00011}\put(263,88){01011}
%\put(263,157){11001}\put(263,58){11011}
%\put(303,128){10011}\put(303,88){10111}\put(316,102){00111}
%\put(352,102){01111}\put(390,102){11111}}

\drawvertex(A){$\bullet$}\drawvertex(B){$\bullet$}
\drawvertex(C){$\bullet$}\drawvertex(D){$\bullet$}
\drawvertex(E){$\bullet$}\drawvertex(F){$\bullet$}
\drawvertex(G){$\bullet$}\drawvertex(H){$\bullet$}
\drawvertex(I){$\bullet$}\drawvertex(L){$\bullet$}
\drawvertex(M){$\bullet$}\drawvertex(N){$\bullet$}
\drawvertex(O){$\bullet$}\drawvertex(P){$\bullet$}
\drawvertex(J){$\bullet$}\drawvertex(K){$\bullet$}
\drawvertex(Q){$\bullet$}\drawvertex(R){$\bullet$}
\drawvertex(S){$\bullet$}\drawvertex(T){$\bullet$}
\drawvertex(U){$\bullet$}\drawvertex(V){$\bullet$}
\drawvertex(W){$\bullet$}\drawvertex(X){$\bullet$}
\drawvertex(Y){$\bullet$}\drawvertex(Z){$\bullet$}
\drawvertex(g){$\bullet$}\drawvertex(h){$\bullet$}
\drawvertex(c){$\bullet$}\drawvertex(f){$\bullet$}
\drawvertex(d){$\bullet$}\drawvertex(e){$\bullet$}

\drawundirectedcurvededge(A,B){}\drawundirectedcurvededge(B,A){}
\drawundirectedcurvededge(B,C){}\drawundirectedcurvededge(C,B){}
\drawundirectededge(C,D){} \drawundirectededge(D,F){}
\drawundirectededge(F,E){} \drawundirectededge(E,C){}

\drawundirectededge(F,G){} \drawundirectededge(G,I){}
\drawundirectededge(I,L){} \drawundirectededge(L,F){}

\drawundirectedcurvededge(G,H){}\drawundirectedcurvededge(H,G){}
\drawundirectedcurvededge(L,M){}\drawundirectedcurvededge(M,L){}

\drawundirectededge(I,N){} \drawundirectededge(N,O){}
\drawundirectededge(O,R){} \drawundirectededge(R,S){}
\drawundirectededge(S,T){} \drawundirectededge(T,U){}
\drawundirectededge(U,V){} \drawundirectededge(V,I){}

\drawundirectedcurvededge(O,P){}\drawundirectedcurvededge(P,O){}
\drawundirectedcurvededge(Q,P){}\drawundirectedcurvededge(P,Q){}

\drawundirectedcurvededge(U,Z){}\drawundirectedcurvededge(Z,U){}
\drawundirectedcurvededge(Z,J){}\drawundirectedcurvededge(J,Z){}

\drawundirectededge(S,K){} \drawundirectededge(K,X){}
\drawundirectededge(X,W){} \drawundirectededge(W,S){}
\drawundirectededge(X,c){} \drawundirectededge(c,d){}
\drawundirectededge(d,Y){} \drawundirectededge(Y,X){}

\drawundirectedcurvededge(d,e){}\drawundirectedcurvededge(e,d){}
\drawundirectedcurvededge(e,f){}\drawundirectedcurvededge(f,e){}
\drawundirectedcurvededge(K,g){}\drawundirectedcurvededge(g,K){}
\drawundirectedcurvededge(W,h){}\drawundirectedcurvededge(h,W){}

\drawundirectedloop[l](A){}\drawundirectedloop(D){}\drawundirectedloop[b](E){}
\drawundirectedloop(H){}\drawundirectedloop[b](M){}\drawundirectedloop(Q){}\drawundirectedloop[b](J){}
\drawundirectedloop(g){}\drawundirectedloop[b](h){}\drawundirectedloop(c){}\drawundirectedloop[b](Y){}
\drawundirectedloop[r](f){}\drawundirectedloop[b](V){}\drawundirectedloop(N){}\drawundirectedloop[b](T){}
\drawundirectedloop(R){}
\end{picture}
\end{center}

\unitlength=0,3mm\setloopdiam{17}\setprofcurve{10}

As the graph $\Gamma_n$ is a $4$-regular graph, for each $n\geq
1$, it is natural to construct the sequence $\{\Gamma_n \zig
C_4\}_{n\geq 1}$. We label the graph $C_4$ as follows:
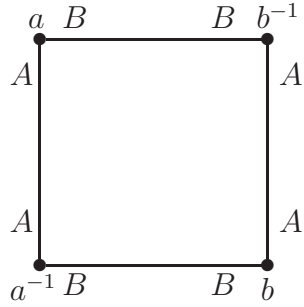
\begin{center}
\begin{picture}(400,120)
\letvertex A=(150,110)\letvertex B=(150,10)\letvertex C=(250,10)\letvertex D=(250,110)

\drawvertex(A){$\bullet$}\drawvertex(B){$\bullet$}
\drawvertex(C){$\bullet$}\drawvertex(D){$\bullet$}

\drawundirectededge(A,B){}\drawundirectededge(B,C){}
\drawundirectededge(C,D){}\drawundirectededge(D,A){}

\put(145,115){$a$}\put(137,-5){$a^{-1}$}\put(247,-5){$b$}\put(245,115){$b^{-1}$}

\put(137,90){$A$}\put(137,25){$A$}
\put(160,-3){$B$}\put(225,-3){$B$}\put(255,25){$A$}
\put(255,90){$A$}\put(160,115){$B$} \put(225,115){$B$}
\end{picture}
\end{center}
We identify in a natural manner the ordered set $\{a,a^{-1}, b,
b^{-1} \}$ with the ordered set $[4]=\{1,2,3,4\}$. Then the
following result holds.
\begin{prop}\label{basilicaruote}
Let $\Gamma_n$ be the Schreier graph of the action of the Basilica group on $\{0,1\}^n$.
Then the parity block decomposition of $\Gamma_n$ consists of only
one parity block, so that the graph $\Gamma_n\zig C_4$ is a connected
graph, isomorphic to the double cycle graph $DC_{2^{n+1}}$.
\end{prop}
\begin{proof}
Take an arbitrary vertex $v\in \Gamma_n$: by the construction of $\Gamma_n$, we have
$$
\rot_{\Gamma_n}(v,b^{-1}) = (b^{-1}(v), b) \qquad
\text{ and } \qquad
\rot_{\Gamma_n}(b^{-1}(v),a) = (ab^{-1}(v), a^{-1}).
$$
Observe that the vertices $a$ and $b$ in $C_4$ have been
identified with the integers $1$ and $3$, respectively: this ensures that if we move in
$\Gamma_n$ by using the generators $b^{-1}$ and $a$
alternately, we do not leave our parity block, since the parity properties are satisfied.
By iterating this argument, we describe a path in $\Gamma_n$,
consisting of the vertices that one obtains starting from $v$ and applying alternately
$b^{-1}$, $a$, $b^{-1}$, $a$ and so on. On
the other hand, it is not difficult to prove, by induction on $n$,
that the action of the product $ab^{-1}$ on $X^n$ has order $2^n$,
so that the construction produces a spanning path of $\Gamma_n$, of length $2^{n+1}$,
where each vertex of $\Gamma_n$ occurs twice. Therefore, we
get a unique parity block coinciding with $\Gamma_n$; by
Proposition \ref{propprop}, we conclude that the zig-zag product
$\Gamma_n \zig C_4$ consists of only one connected component
isomorphic to $DC_{2^{n+1}}$.
\end{proof}

\begin{cor}
Let $\Gamma$ be the Schreier graph of the action of a group $G$, generated by the symmetric set
$S=\{a^{\pm 1}, b^{\pm 1}\}$, on a set $X$ with $|X|=n$, so that $\rot_{\Gamma}(x,s) = (s(x),s^{-1})$,
with $s\in S$. Let $C_4$ be the cycle graph whose vertices are labelled as above. Then the graph
$\Gamma \zig C_4$ is connected if and only if the action of $ab^{-1}$ is transitive on $X$. If this is the case, one has $\Gamma\zig C_4 \simeq DC_{2n}$.
\end{cor}
\begin{example}\rm
The zig-zag product $\Gamma_1\zig C_4$ gives rise to the following
graph:
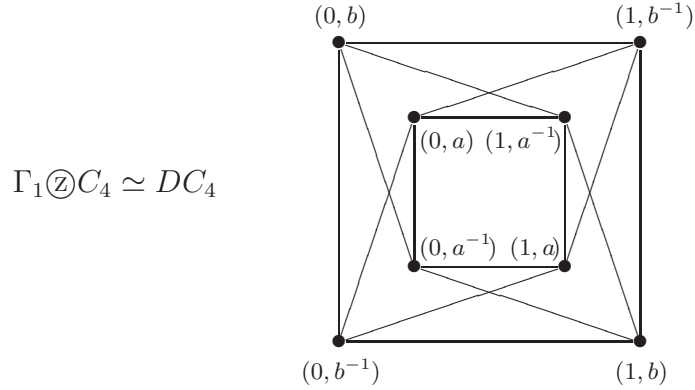
\begin{center}
\begin{picture}(270,190)
\unitlength=0.33mm
\letvertex AA=(140,160)\letvertex BB=(140,40)\letvertex CC=(260,40)
\letvertex DD=(260,160)\letvertex a=(170,130)\letvertex b=(170,70)
\letvertex c=(230,70)\letvertex d=(230,130)

\drawundirectededge(AA,BB){}\drawundirectededge(BB,CC){}
\drawundirectededge(CC,DD){}\drawundirectededge(DD,AA){}

\drawundirectededge(a,b){}\drawundirectededge(b,c){}
\drawundirectededge(c,d){}\drawundirectededge(d,a){}

\drawundirectededge(AA,b){}\drawundirectededge(b,CC){}
\drawundirectededge(CC,d){}\drawundirectededge(d,AA){}

\drawundirectededge(a,BB){}\drawundirectededge(BB,c){}
\drawundirectededge(c,DD){}\drawundirectededge(DD,a){}

\drawvertex(AA){$\bullet$}\drawvertex(BB){$\bullet$}
\drawvertex(CC){$\bullet$}\drawvertex(DD){$\bullet$}\drawvertex(a){$\bullet$}
\drawvertex(b){$\bullet$}\drawvertex(c){$\bullet$}\drawvertex(d){$\bullet$}

\put(10,100){$\Gamma_1\zig C_4\simeq DC_4$}

{\scriptsize \put(130,168){$(0,b)$}\put(125,25){$(0,b^{-1})$}
\put(250,25){$(1,b)$} \put(250,168){$(1,b^{-1})$}

\put(172,118){$(0,a)$}\put(172,75){$(0,a^{-1}\!)$}
\put(208,75){$(1,a\!)$}\put(198,118){$(1,a^{-1}\!)$}}
\end{picture}
\end{center}
For $n=2,3$, we get the double cycle graphs $DC_8$ and $DC_{16}$,
respectively (labels are omitted in the pictures below).
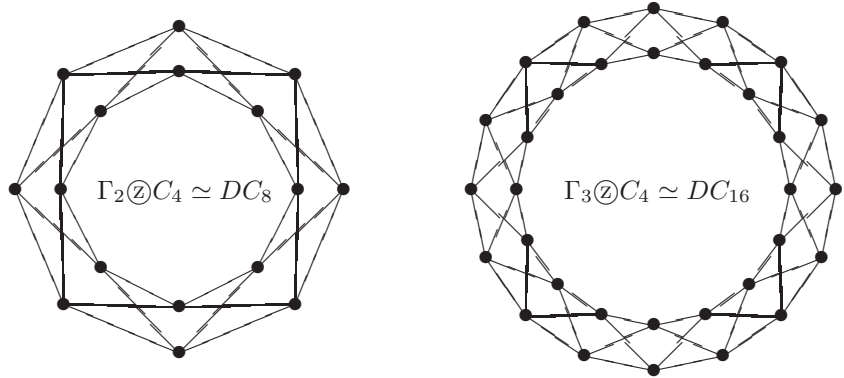
\begin{center}
\begin{picture}(340,170)\unitlength=0,12mm
{\footnotesize \put(110,200){$\Gamma_2\zig C_4\simeq DC_8$}
\put(622,200){$\Gamma_3\zig C_4\simeq DC_{16}$}}

\letvertex AA=(200,390)\letvertex BB=(73,337)\letvertex CC=(20,210)
\letvertex DD=(73,83)\letvertex EE=(200,30)\letvertex FF=(327,83)\letvertex GG=(380,210)
\letvertex HH=(327,337)

\letvertex aa=(200,340)\letvertex bb=(114,296)\letvertex cc=(70,210)\letvertex dd=(114,124)
\letvertex ee=(200,80)\letvertex ff=(286,124)\letvertex gg=(330,210)\letvertex hh=(286,296)

\drawvertex(AA){$\bullet$}\drawvertex(BB){$\bullet$}
\drawvertex(CC){$\bullet$}\drawvertex(DD){$\bullet$}
\drawvertex(EE){$\bullet$}\drawvertex(FF){$\bullet$}
\drawvertex(GG){$\bullet$}\drawvertex(HH){$\bullet$}

\drawvertex(aa){$\bullet$}\drawvertex(bb){$\bullet$}
\drawvertex(cc){$\bullet$}\drawvertex(dd){$\bullet$}
\drawvertex(ee){$\bullet$}\drawvertex(ff){$\bullet$}
\drawvertex(gg){$\bullet$}\drawvertex(hh){$\bullet$}

\drawundirectededge(AA,BB){}\drawundirectededge(BB,CC){}\drawundirectededge(CC,DD){}\drawundirectededge(DD,EE){}
\drawundirectededge(EE,FF){}\drawundirectededge(FF,GG){}\drawundirectededge(GG,HH){}\drawundirectededge(HH,AA){}

\drawundirectededge(aa,bb){}\drawundirectededge(bb,cc){}\drawundirectededge(cc,dd){}\drawundirectededge(dd,ee){}
\drawundirectededge(ee,ff){}\drawundirectededge(ff,gg){}\drawundirectededge(gg,hh){}\drawundirectededge(hh,aa){}

\drawundirectededge(AA,bb){}\drawundirectededge(bb,CC){}\drawundirectededge(CC,dd){}\drawundirectededge(dd,EE){}
\drawundirectededge(EE,ff){}\drawundirectededge(ff,GG){}\drawundirectededge(GG,hh){}\drawundirectededge(hh,AA){}

\drawundirectededge(aa,BB){}\drawundirectededge(BB,cc){}\drawundirectededge(cc,DD){}\drawundirectededge(DD,ee){}
\drawundirectededge(ee,FF){}\drawundirectededge(FF,gg){}\drawundirectededge(gg,HH){}\drawundirectededge(HH,aa){}

%%%%%%%%%%%%%%%%%%%%%%%%%%%%%%%%%%%%%%%%%%%%%%%%%%%%%%%%%%%%%%%%%%%%%%%%%%%%%%%%%%%%%%%%%%%%%%%%%%%%
\letvertex A=(720,410)\letvertex B=(796,394)\letvertex C=(860,350)
\letvertex D=(904,286)\letvertex E=(920,210)\letvertex F=(904,134)\letvertex G=(860,70)
\letvertex H=(796,26)\letvertex I=(720,10)\letvertex L=(644,26)\letvertex M=(580,70)
\letvertex N=(536,134)\letvertex O=(520,210)\letvertex P=(536,286)\letvertex Q=(580,350)
\letvertex R=(644,394)

\letvertex a=(720,360)\letvertex b=(777,348)\letvertex c=(825,315)
\letvertex d=(858,267)\letvertex e=(870,210)\letvertex f=(858,153)\letvertex g=(825,105)
\letvertex h=(777,72)\letvertex i=(720,60)\letvertex l=(663,72)\letvertex m=(615,105)
\letvertex n=(582,153)\letvertex o=(570,210)\letvertex p=(582,267)\letvertex q=(615,315)
\letvertex r=(663,348)

\drawvertex(A){$\bullet$}\drawvertex(B){$\bullet$}
\drawvertex(C){$\bullet$}\drawvertex(D){$\bullet$}
\drawvertex(E){$\bullet$}\drawvertex(F){$\bullet$}
\drawvertex(G){$\bullet$}\drawvertex(I){$\bullet$}\drawvertex(M){$\bullet$}\drawvertex(N){$\bullet$}
\drawvertex(H){$\bullet$}\drawvertex(L){$\bullet$}\drawvertex(O){$\bullet$}\drawvertex(P){$\bullet$}
\drawvertex(Q){$\bullet$}\drawvertex(R){$\bullet$}

\drawvertex(a){$\bullet$}\drawvertex(b){$\bullet$}
\drawvertex(c){$\bullet$}\drawvertex(d){$\bullet$}
\drawvertex(e){$\bullet$}\drawvertex(f){$\bullet$}
\drawvertex(g){$\bullet$}\drawvertex(i){$\bullet$}\drawvertex(m){$\bullet$}\drawvertex(n){$\bullet$}
\drawvertex(h){$\bullet$}\drawvertex(l){$\bullet$}\drawvertex(o){$\bullet$}\drawvertex(p){$\bullet$}
\drawvertex(q){$\bullet$}\drawvertex(r){$\bullet$}

\drawundirectededge(A,B){}\drawundirectededge(B,C){}\drawundirectededge(C,D){}\drawundirectededge(D,E){}
\drawundirectededge(E,F){}\drawundirectededge(F,G){}\drawundirectededge(G,H){}\drawundirectededge(H,I){}
\drawundirectededge(I,L){}\drawundirectededge(L,M){}\drawundirectededge(M,N){}
\drawundirectededge(N,O){}\drawundirectededge(O,P){}\drawundirectededge(P,Q){}\drawundirectededge(Q,R){}
\drawundirectededge(R,A){}

\drawundirectededge(a,b){}\drawundirectededge(b,c){}\drawundirectededge(c,d){}\drawundirectededge(d,e){}
\drawundirectededge(e,f){}\drawundirectededge(f,g){}\drawundirectededge(g,h){}\drawundirectededge(h,i){}
\drawundirectededge(i,l){}\drawundirectededge(l,m){}\drawundirectededge(m,n){}
\drawundirectededge(n,o){}\drawundirectededge(o,p){}\drawundirectededge(p,q){}\drawundirectededge(q,r){}
\drawundirectededge(r,a){}

\drawundirectededge(A,b){}\drawundirectededge(b,C){}\drawundirectededge(C,d){}\drawundirectededge(d,E){}
\drawundirectededge(E,f){}\drawundirectededge(f,G){}
\drawundirectededge(G,h){}\drawundirectededge(h,I){}
\drawundirectededge(I,l){}\drawundirectededge(l,M){}
\drawundirectededge(M,n){}\drawundirectededge(n,O){}\drawundirectededge(O,p){}
\drawundirectededge(p,Q){}\drawundirectededge(Q,r){}\drawundirectededge(r,A){}

\drawundirectededge(a,B){}\drawundirectededge(B,c){}\drawundirectededge(c,D){}\drawundirectededge(D,e){}
\drawundirectededge(e,F){}\drawundirectededge(F,g){}
\drawundirectededge(g,H){}\drawundirectededge(H,i){}
\drawundirectededge(i,L){}\drawundirectededge(L,m){}
\drawundirectededge(m,N){}\drawundirectededge(N,o){}\drawundirectededge(o,P){}
\drawundirectededge(P,q){}\drawundirectededge(q,R){}\drawundirectededge(R,a){}
\end{picture}
\end{center}
\end{example}
\begin{remark}\rm
Observe that we get a connected zig-zag product $\Gamma_n\zig C_4$ for every $n$, even though the neighborhood graph
$\mathcal{N}$ associated with $C_4$ is disconnected, according with the fact that the condition in Theorem \ref{teoremaconnessione}
is only sufficient but not necessary in order to have the connectedness property.
\end{remark}

\begin{thm}\label{basilicaspettro}
For every $n\geq 1$, the spectrum of the graph $\Gamma_n\zig C_4\simeq DC_{2^{n+1}}$
is given by
$$
\Sigma_n = \left\{\underbrace{0,\ldots, 0}_{2^{n+1} \
\text{times}}, 4\cos \left(\frac{\pi j }{2^n}\right),
j=0,1,\ldots, 2^{n+1}-1\right\}.
$$
Let ${\bf u}_0=(1,-1)$, ${\bf u}_1=(1,1)$ and ${\bf v}_j=(1, w^j,
w^{2j}, \ldots, w^{(2^{n+1}-1)j})$. Then for each $j=0,1,\ldots, 2^{n+1}-1$,
${\bf u}_0\otimes {\bf v}_j$ is an eigenvector associated with the
eigenvalue $0$, and $ {\bf u}_1\otimes {\bf v}_j$ is an
eigenvector associated with the eigenvalue $4\cos \left(\frac{\pi
j }{2^n}\right)$.
\end{thm}
\begin{proof}
It can be deduced from the structure of the graph $DC_{2^{n+1}}$ that the
adjacency matrix of the graph $\Gamma_n\zig C_4$ is a circulant
matrix of size $2^{n+2}$.
More precisely, we choose the following order of the vertices of
$\Gamma_n\zig C_4$:
$$
(0^n,a); (b^{-1}(0^n),a^{-1}); (ab^{-1}(0^n),a); (b^{-1}ab^{-1}(0^n),a^{-1}); \ldots ; ((ab^{-1})^{2^n-1}(0^n),a);
$$
$$
(b^{-1}(ab^{-1})^{2^n-1}(0^n),a^{-1}); (0^n,b);  (b^{-1}(0^n),b^{-1}); (ab^{-1}(0^n),b);
(b^{-1}ab^{-1}(0^n),b^{-1}); \ldots
$$
$$
\ldots; ((ab^{-1})^{2^n-1}(0^n),b);  (b^{-1}(ab^{-1})^{2^n-1}(0^n),
b^{-1}).
$$
In other words, we are applying alternately  $b^{-1}$ and $a$ to
the word $0^n$ (observe that the $2^n$-iteration of $ab^{-1}$ is
the identity on $\{0,1\}^n$), with an alternation of $a,a^{-1}$ in the second coordinate of the first
$2^n$ vertices (the inner cycle) and of $b,b^{-1}$ in the second coordinate of the second
$2^n$ vertices (the outer cycle). It is straightforward to check that, with this ordering of the
vertices, the adjacency matrix $M_n$ of the graph $\Gamma_n\zig C_4$,
for each $n\geq 1$, is given by $U\otimes \widetilde{M}_n$,
where $U =
\begin{pmatrix}
  1 & 1 \\
  1 & 1
\end{pmatrix}$ and $\widetilde{M}_n$ is the square matrix of size
$2^{n+1}$ of type
$$
\widetilde{M}_n = \begin{pmatrix}
  0 & 1 & 0 &  & \ldots &  & 0 & 1 \\
  1 & 0 & 1 &   &  &  & 0 & 0 \\
  0 &  1 & 0 &  1 &  &  & &  \\
   &  &  1& 0 &  \ddots &  &  & \vdots \\
 \vdots  &  &  &  \ddots &  \ddots &  1 &  &  \\
   &  &  &  &  1&  0 & 1 & 0 \\
 0  & 0 &  &  &  & 1 & 0 & 1 \\
 1  & 0 &  & \ldots & & 0 & 1 & 0
\end{pmatrix}.
$$
The matrix $\widetilde{M}_n$ is a circulant matrix, according with
Definition \ref{circulant}. More precisely, $\widetilde{M}_n$ is a circulant matrix of
size $2^{n+1}$, satisfying
$$
c_k =
\begin{cases}
1,&\text{if }k = 1, 2^{n+1}-1\\
0,&\text{otherwise}.
\end{cases}
$$
We deduce that, for every $j=0,1,\ldots, 2^{n+1}-1$, the vector
$$
\textbf{v}_j=(1, w^j, w^{2j}, \ldots, w^{(2^{n+1}-1)j}),
$$
where $w= \exp\left(\frac{2\pi i }{2^{n+1}}\right)$, and $i^2=-1$,
is an eigenvector of $\widetilde{M}_n$ with associated eigenvalue
$$
\lambda_j = (-1)^j \cdot 2 \cos \left(\pi j - \frac{\pi j
}{2^n}\right).
$$
On the other hand, we have
$$
\cos \left(\pi j - \frac{\pi j }{2^n}\right) = \cos(\pi
j)\cos\left(\frac{\pi j }{2^n}\right) + \sin (\pi j) \sin \left(
\frac{\pi j }{2^n}\right) = (-1)^j \cos\left(\frac{\pi j
}{2^n}\right),
$$
so that the $j$-th eigenvalue of $\widetilde{M}_n$, for
$j=0,1,\ldots, 2^{n+1}-1$, is
$$
\lambda_j = 2\cos \left(\frac{\pi j }{2^n}\right).
$$
Notice that the eigenvalues of the matrix $U$ are $\mu_0=0$, with
eigenvector $\textbf{u}_0 = (1,-1)$, and $\mu_1 = 2$, with
eigenvector $\textbf{u}_1=(1,1)$. As $M_n = U\otimes
\widetilde{M}_n$, the eigenvectors of the matrix $M_n$ are given
by $\textbf{u}_i\otimes \textbf{v}_j$, with associated eigenvalue
$\mu_i\lambda_j$, for $i=0,1$ and $j=0,1,\ldots, 2^{n+1}-1$. This
gives the assertion.
\end{proof}
Observe that in \cite{annals} the authors defined the zig-zag product
$G_1\zig G_2$ of two finite connected regular graphs $G_1=(V_1,E_1)$ and $G_2=(V_2,E_2)$.
The definition can be naturally extended to the case where $G_1$ is an infinite regular graph,
and the degree of $G_1$ is equal to $|V_2|$. The analysis of zig-zag products of infinite graphs
will be considered in the upcoming paper \cite{DADSH}. In the following example, we will consider
the infinite case, where $G_1$ is the infinite Schreier graph of a word in $\{0,1\}^\infty$ under
the action of the Basilica group, and $G_2$ is the cycle graph of length $4$.

\begin{example}\rm
It is not difficult to see that the zig-zag product $\Gamma_{\xi}\zig C_4$, where $\Gamma_{\xi}$ is the
infinite $4$-regular graph describing the orbit of the vertex $\xi$, with $\xi\in
\{0,1\}^\infty$, and $\xi\not\in B(0^\infty)$, is an infinite connected $4$-regular graph
isomorphic to the following graph:
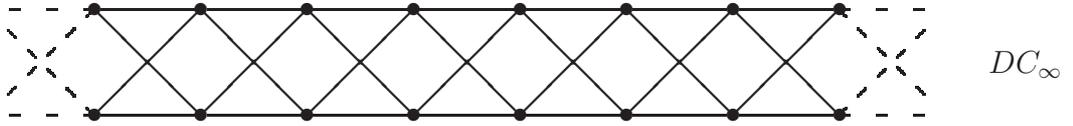
\begin{center}
\begin{picture}(430,60)\unitlength=0,28mm \thicklines
\letvertex A=(40,60)\letvertex B=(90,60)\letvertex C=(140,60)
\letvertex D=(190,60)\letvertex E=(240,60)\letvertex F=(290,60)\letvertex G=(340,60)
\letvertex H=(390,60)

\letvertex a=(40,10)\letvertex b=(90,10)\letvertex c=(140,10)
\letvertex d=(190,10)\letvertex e=(240,10)\letvertex f=(290,10)\letvertex g=(340,10)
\letvertex h=(390,10)

\drawvertex(A){$\bullet$}\drawvertex(B){$\bullet$}
\drawvertex(C){$\bullet$}\drawvertex(D){$\bullet$}
\drawvertex(E){$\bullet$}\drawvertex(F){$\bullet$}
\drawvertex(G){$\bullet$}\drawvertex(H){$\bullet$}

\drawvertex(a){$\bullet$}\drawvertex(b){$\bullet$}
\drawvertex(c){$\bullet$}\drawvertex(d){$\bullet$}
\put(460,30){$DC_\infty$}
\drawvertex(e){$\bullet$}\drawvertex(f){$\bullet$}
\drawvertex(g){$\bullet$}\drawvertex(h){$\bullet$}

\drawundirectededge(A,B){}\drawundirectededge(B,C){}\drawundirectededge(C,D){}\drawundirectededge(D,E){}
\drawundirectededge(E,F){}\drawundirectededge(F,G){}\drawundirectededge(G,H){}

\drawundirectededge(a,b){}\drawundirectededge(b,c){}\drawundirectededge(c,d){}\drawundirectededge(d,e){}
\drawundirectededge(e,f){}\drawundirectededge(f,g){}\drawundirectededge(g,h){}

\drawundirectededge(A,b){}\drawundirectededge(b,C){}\drawundirectededge(C,d){}\drawundirectededge(d,E){}
\drawundirectededge(E,f){}\drawundirectededge(f,G){}
\drawundirectededge(G,h){}

\drawundirectededge(a,B){}\drawundirectededge(B,c){}\drawundirectededge(c,D){}\drawundirectededge(D,e){}
\drawundirectededge(e,F){}\drawundirectededge(F,g){}
\drawundirectededge(g,H){}

\dashline[5]{6}(390,60)(430,60)\dashline[5]{6}(390,10)(430,10)
\dashline[5]{6}(40,60)(0,60)\dashline[5]{6}(40,10)(0,10)

\dashline[5]{6}(390,60)(430,20)\dashline[5]{6}(390,10)(430,50)
\dashline[5]{6}(40,60)(0,20)\dashline[5]{6}(40,10)(0,50)
\end{picture}
\end{center}
Incidentally, this also shows that the graphs $\Gamma_\xi
\zig C_4$ and $\Gamma_\eta \zig C_4$ may be isomorphic, even if
the graphs $\Gamma_\xi$ and $\Gamma_\eta$ are not isomorphic. More
precisely, we have that the uncountably many graphs $\Gamma_\xi
\zig C_4$ are all isomorphic, for every $\xi \not \in
B(0^\infty)$. This property implies that the zig-zag construction is not injective even in the infinite context.

On the other hand, it is easy to check that the zig-zag product of
the graph $\Gamma_{0^\infty}$ with the cycle graph of length $4$
consists of two infinite connected components, each isomorphic to
the graph $DC_{\infty}$.

We can summarize these results as follows. Choose a
root in the graph $C_4$, and let us denote it by $v_0$. As usual,
denote by $(\Gamma_n,v)$ the graph $\Gamma_n$ rooted at the vertex
$v$. Let $\xi=x_1x_2x_3\ldots\in \{0,1\}^\infty$, and let $\xi_n =
x_1x_2\ldots x_n$, then:
\begin{enumerate}
\item if $\xi \not \in B(0^\infty)$, then
$$
\lim_{n\to\infty} (\Gamma_{\xi_n}, \xi_n)\zig (C_4,v_0) =
(DC_{\infty}, (\xi,v_0));
$$
\item if $\xi \in B(0^\infty)$, then
$$
\lim_{n\to\infty} (\Gamma_{\xi_n}, \xi_n)\zig (C_4,v_0) =
(DC_{\infty}, (\xi,v_0)) \cup (DC_{\infty}, (\xi,v_0+1)).
$$
\end{enumerate}

\end{example}

%%%%%%%%%%%%%%%%%%%%%%%%%%%%%%%%%%%%%%%%%%%%%%%%%%%%%%%%%%%%%%%%%%%%%%%%%%%%%%%%%%%%%%%%%%%%%%%%%%%%%%%%%%%%%%%%%%%%%%%%%%%%%%%%%%%%%%%%%%%%%%%%%%

\section*{Acknowledgments}
The authors want to thank Wolfgang Woess and Tullio Ceccherini-Silberstein for useful and stimulating discussions on the subjects of this
paper. Daniele D'Angeli was supported by Austrian Science Fund (FWF) P24028-N18.  Alfredo Donno was partially supported by the European Science
Foundation (Research Project RGLIS 4915). Ecaterina Sava-Huss was supported by Austrian Science Fund (FWF): W1230.

%%%%%%%%%%%%%%%%%%%%%%%%%%%%%%%%%%%%%%%%%%%%%%%%%%%%%%%%%%%%%%%%%%%%%%%%%%%%%%%%%%%%


\begin{thebibliography}{99}
\bibitem{abdollahi} A. Abdollahi and A. Loghman, On one-factorizations of replacement
products, to appear in {\it Filomat}, Published by Faculty of
Sciences and Mathematics, University of Ni\v{s}, Serbia, available
at \texttt{http://www.pmf.ni.ac.rs/filomat}

\bibitem{groups} N. Alon, A. Lubotzky and A. Wigderson,
Semi-direct product in groups and zig-zag product in graphs:
connections and applications (extended abstract), in:
\textit{\lq\lq $42$-nd IEEE Symposium on Foundations of Computer
Science, Las Vegas, NV, 2001\rq\rq}, 630--637. IEEE Computer
Society, Los Alamitos, CA, 2001.

\bibitem{amen} L. Bartholdi and B. Vir\'{a}g, Amenability via random
walks, {\it Duke Math Journal} {\bf 130} (2005), no. 1, 39--56.

\bibitem{tutte1} T. Ceccherini-Silberstein, A. Donno and D. Iacono, Tutte polynomial of the Schreier graphs of the Grigorchuk
group and the Basilica group, in: \textit{Ischia Group Theory 2010
(Proceedings of the Conference)} (M. Bianchi, P. Longobardi, M.
Maj and C. M. Scoppola editors), World Scientific 2011, 45--68.

\bibitem{libro} T. Ceccherini-Silberstein, F. Scarabotti and F.
Tolli, Harmonic Analysis on Finite Groups: Representation theory,
Gelfand pairs and Markov chains. {\it Cambridge Studies in
Advanced Mathematics}, {\bf 108}, Cambridge University Press,
2008. xiv + 440 pp.

\bibitem{ADM} D. D'Angeli and A. Donno, Self-similar groups and finite Gelfand pairs,
{\it Algebra Discrete Math.}, no. 2, (2007), 54--69.

\bibitem{JMD2} D. D'Angeli, A. Donno, M. Matter and T. Nagnibeda, Schreier graphs of the
Basilica group, {\it J. Mod. Dyn.} {\bf 4} (2010), no. 1,
167--205.

\bibitem{DIMERI} D. D'Angeli, A. Donno and T. Nagnibeda, Counting dimer coverings on self-similar Schreier
graphs, {\it European J. Combin.} {\bf 33} (2012), no. 7,
1484--1513.

\bibitem{ISING} D. D'Angeli, A. Donno and T. Nagnibeda, Partition functions of the Ising model on some self-similar Schreier
graphs, in: {\it Progress in Probability: Random Walks, Boundaries
and Spectra} (D. Lenz, F. Sobieczky and W. Woess editors), {\bf
64} (2011), 277--304, Springer Basel.

\bibitem{DADSH} D. D'Angeli, A. Donno and E. Sava-Huss, Zig-zag products of infinite graphs. In preparation.

\bibitem{valettebook} G. Davidoff, P. Sarnak and A. Valette, Elementary number theory, group theory, and Ramanujan graphs.
{\it London Mathematical Society Student Texts}, {\bf 55}.
Cambridge University Press, Cambridge, 2003. x + 144 pp.

\bibitem{davis} P. J. Davis, Circulant matrices. A Wiley-Interscience
Publication. {\it Pure and Applied Mathematics}. John Wiley \&
Sons, New York-Chichester-Brisbane, 1979. xv + 250 pp.

\bibitem{alfredoIJGT} A. Donno, Replacement and zig-zag products,
Cayley graphs and Lamplighter random walk, {\it Int. J. Group
Theory} {\bf 2} (2013) No. 1, 11--35.

\bibitem{donnozigzag2} A. Donno, Generalized wreath products of graphs and
groups, to appear in {\it Graphs and
Combinatorics}, published online at \texttt{http://link.springer.com/article/10.1007/s00373-014-1414-4}, DOI 10.1007/s00373-014-1414-4

\bibitem{tutte2} A. Donno and D. Iacono, The Tutte polynomial of the Sierpi\'{n}ski and Hanoi graphs, {\it Adv. Geom.}, Vol. 13
(2013), Issue 4, 663--694.

\bibitem{grigorchuksolved} R. I. Grigorchuk,
Solved and unsolved problems around one group, Infinite groups:
geometric, combinatorial and dynamical aspects, 117--218, {\it
Progr. Math.}, {\bf 248}, Birkh\"{a}user, Basel, 2005.

\bibitem{grizuk} R. I. Grigorchuk and A. \.{Z}uk, On a torsion-free
weakly branch group defined by a three-state automaton. {\it
Internat. J. Algebra Comput.}, {\bf 12}, (2002), no. 1, 223--246.

\bibitem{gromov} M. Gromov, Filling Riemannian manifolds, {\it J. Differential
Geom.} {\bf 18} (1983), no. 1, 1--147.

\bibitem{Grom} M. Gromov, Structures m\'{e}triques pour les vari\'{e}t\'{e}s riemanniennes,
\textit{Textes Math\'{e}matiques}, J. Lafontaine and P. Pansu (Eds.),
1. CEDIC, Paris, 1981. iv+152 pp. ISBN: 2-7124-0714-8.

\bibitem{expander} S. Hoory, N. Linial and A. Wigderson, Expander
graphs and their application, {\it Bull. Amer. Math. Soc. (N.S.)}
{\bf 43} (2006), no. 4, 439--561.

\bibitem{communications} C. A. Kelley, D. Sridhara and J. Rosenthal, Zig-zag and replacement product graphs and LDPC codes, {\it Adv. Math.
Commun.} {\bf 2} (2008), no. 4, 347--372.

\bibitem{lubotullio} A. Lubotzky, Expander graphs in pure and
applied mathematics, {\it Bull. Amer. Math. Soc.} {\bf 49} (2012),
no. 1, 113--162.

\bibitem{nekrashevych} V. Nekrashevych, {\it Self-similar Groups}, Volume 117 of Mathematical Surveys and Monographs,
American Mathematical Society, Providence, RI, 2005. xii + 231 pp.

\bibitem{nekrateplyaev} V. Nekrashevych and A. Teplyaev, Groups and analysis on fractals. In:
\lq\lq{Analysis on Graphs and its Applications}\rq\rq, Proceedings
of Symposia in Pure Mathematics, Amer. Math. Soc., {\bf 77},
143--180. Amer. Math. Soc., Providence (2008).

\bibitem{annals} O. Reingold, S. Vadhan and A. Wigderson, Entropy
Waves, the Zig-Zag Graph Product, and New Constant-Degree
Expanders, {\it Ann. of Math. (2)} {\bf 155} (2002), no. 1,
157--187.

\bibitem{tee} G. J. Tee, Eigenvectors of block circulant and
alternating circulant matrices, {\it Res. Lett. Inf. Math. Sci.}
{\bf 8} (2005), 123--142.

\bibitem{danielecitazione} I. Tomescu, Problems in Combinatorics and Graph Theory. Translated from the Romanian by R. A. Melter. Wiley Interscience Series in Discrete Mathematics. {\it John Wiley \& Sons, Ltd. Chichester}, 1985. xvii + 335 pp.
\end{thebibliography}
\end{document}